\documentclass[reqno]{amsart}
\usepackage{amsmath,amssymb,mathrsfs,amsthm,amsfonts,mathtools}
\usepackage[inline]{enumitem} 
\usepackage[usenames,dvipsnames]{xcolor}
\definecolor{hypercolor}{HTML}{003399}
\usepackage{hyperref}
\hypersetup{%
	colorlinks=true, linkcolor=hypercolor,
	citecolor=hypercolor
}
\usepackage[paper=letterpaper,margin=1in]{geometry}

\newtheorem{theorem}{Theorem}[section]
\newtheorem{lemma}[theorem]{Lemma}
\newtheorem{proposition}[theorem]{Proposition}
\newtheorem{corollary}[theorem]{Corollary}
\theoremstyle{definition}

\newtheorem{remark}[theorem]{Remark}

\numberwithin{equation}{section}
\allowdisplaybreaks

\renewcommand{\ge}{\geqslant}
\renewcommand{\leq}{\leqslant}
\renewcommand{\geq}{\geqslant}

\usepackage{acronym}
\acrodef{SHE}{Stochastic Heat Equation}
\acrodef{EW}{Edwards--Wilkinson}
\acrodef{KPZ}{Kardar--Parisi--Zhang}
\acrodef{LHS}{Left Hand Side}
\acrodef{RHS}{Right Hand Side}

\newcommand{\e}{\varepsilon}
\newcommand{\eps}{\varepsilon}
\newcommand{\ic}{\mathrm{ic}}		
\newcommand{\cmp}{\mathrm{c}}		
\newcommand{\sym}{\mathrm{sym}}		
\newcommand{\fr}{\text{fr}}			

\newcommand{\tilp}{\widetilde{p}}

\newcommand{\M}{M}

\newcommand{\Io}{\mathbf{I}}	
\newcommand{\Ho}{\mathcal{H}}	
\newcommand{\Ro}{\mathcal{R}}	
\newcommand{\Do}{\mathcal{V}}	
\newcommand{\Dio}{\mathcal{D}}	
\newcommand{\So}{\mathcal{S}}	
\newcommand{\Po}{\Omega_\phi}	
\newcommand{\Pro}{\Pi}			
\newcommand{\Jo}{\mathcal{J}}	
\newcommand{\logo}{\mathcal{L}}					
\newcommand{\Go}{\mathcal{G}}
\newcommand{\Soo}{\mathcal{G}}
\newcommand{\Pt}{\mathcal{P}}

\newcommand{\Ptj}[1]{\mathcal{P}^{\Jo}_{#1}}
\newcommand{\geno}{\mathcal{Q}}		
\newcommand{\auxo}{\mathcal{T}}		
\newcommand{\smallo}{\mathcal{A}}	
\newcommand{\Gtwo}{\mathsf{G}}		
\newcommand{\hktwo}{\mathsf{p}}		
\newcommand{\hk}{P}					
\newcommand{\smallt}{A}				

\newcommand{\C}{\mathbb{C}}

\newcommand{\R}{\mathbb{R}}
\newcommand{\Z}{\mathbb{Z}}
\newcommand{\Lsp}{\mathscr{L}}	
\newcommand{\Ssp}{\mathscr{S}}	
\newcommand{\Hsp}{\mathscr{H}}	
\newcommand{\Csp}{\mathscr{C}}	
\newcommand{\diag}{\mathrm{Dgm}}
\newcommand{\genH}{\mathscr{K}}		
\newcommand{\E}{\mathbb{E}}			
\newcommand{\ind}{\mathbf{1}}		
\newcommand{\1}{\mathbf{1}}			
\newcommand{\img}{\mathbf{i}}		
\renewcommand{\d}{\mathrm{d}}	
\renewcommand{\Re}{\mathrm{Re}}		
\renewcommand{\Im}{\mathrm{Im}}		
\newcommand{\set}[1]{{\{#1\}}}
\newcommand{\Img}{\mathrm{Img}}		
\newcommand{\Dom}{\mathrm{Dom}}		
\newcommand{\dsum}{{\sum}^\mathrm{d}} 
\newcommand{\jfn}{\mathsf{j}} 		
\newcommand{\dker}{D} 				
\newcommand{\cEM}{\gamma_\mathrm{EM}}		
\newcommand{\betaf}{\beta_\mathrm{fine}} 	
\newcommand{\betafe}{\beta_{\e,\mathrm{fine}}} 
\newcommand{\betaphi}{\beta_\Phi} 	
\newcommand{\betaC}{\beta_\star}  	
\newcommand{\betaCe}{\beta_{\star,\e}}
\newcommand{\norm}[1]{\Vert #1\Vert} 			
\newcommand{\Norm}[1]{\big\Vert #1\big\Vert} 	

\newcommand{\normo}[1]{\Vert #1\Vert_{\mathrm{op}}}	
\newcommand{\Normo}[1]{\big\Vert #1\big\Vert_{\mathrm{op}}} 	
\newcommand{\NOrmo}[1]{\Big\Vert #1\Big\Vert_{\mathrm{op}}} 	
\newcommand{\ip}[2]{\langle #1, #2\rangle} 		
\newcommand{\Ip}[2]{\big\langle #1, #2\big\rangle}
\newcommand{\IP}[2]{\Big\langle #1, #2\Big\rangle}

\renewcommand{\bar}{\overline}

\renewcommand{\tilde}{\widetilde}
\renewcommand{\hat}{\widehat}
\renewcommand{\Vec}{\overrightarrow}
\newcommand{\pFou}{\overbracket[.5pt]}

\usepackage{graphicx}
\newcommand*{\Cdot}{{\raisebox{-0.5ex}{\scalebox{1.8}{$\cdot$}}}} 

\title[Moments of the 2D SHE at criticality]{Moments of the 2D SHE at criticality}
\author{Yu Gu, Jeremy Quastel, and Li-Cheng Tsai}

\address[Yu Gu]{\hspace{33pt}Department of Mathematics, Carnegie Mellon University}
\address[Jeremy Quastel]{Department of Mathematics, University of Toronto}
\address[Li-Cheng Tsai]{\hspace{4pt}Departments of Mathematics, Rutgers University --- New Brunswick}


\begin{document}
\begin{abstract}
We study the stochastic heat equation in two spatial dimensions with a multiplicative white noise, as the limit of the equation driven by a noise that is mollified in space and white in time.
As the mollification radius $ \e\to 0 $, we tune the coupling constant near the critical point,
and show that the single time correlation functions converge to a limit written in terms of 
an explicit non-trivial semigroup.
Our approach consists of two steps.
First we show the convergence of the resolvent of the (tuned) two-dimensional delta Bose gas,
by adapting the framework of \cite{dimock04} to our setup of spatial mollification.
Then we match this to the Laplace transform of our semigroup.
\end{abstract}

\maketitle
\section{Introduction and main result}
\label{s.intro}

In this paper, we study the \ac{SHE}, which informally reads
\begin{align*}
	\partial_t Z
	=
	\tfrac12\nabla^2 Z
	+
	\sqrt{\beta}\xi Z, 
	\quad
	Z = Z(t,x),
	\quad
	(t,x) \in \R_+\times\R^d,
\end{align*}
where $ \nabla^2 $ denotes the Laplacian, $ d\in\Z_+ $ denotes the spatial dimension, $ \xi $ denotes the spacetime white noise,
and $ \beta>0 $ is a tunable parameter.
In broad terms, the \ac{SHE} arises from a host of physical phenomena including
the particle density of diffusion in a random environment, 
the partition function for a directed polymer in a random environment,
and, through the inverse Hopf--Cole transformation, the height function of a random growth surface; the two-dimensional \ac{KPZ} equation.
We refer to \cite{corwin12,khoshnevisan14,comets17} and the references therein.

When $ d=1 $, the \ac{SHE} enjoys a well-developed solution theory:
For any $ Z(0,x)=Z_\ic(x) $ that is bounded and continuous, and for each $ \beta>0 $,
the \ac{SHE} (in $ d=1 $) admits a unique $ \Csp([0,\infty)\times\R) $-valued mild solution, where $\Csp$ denotes continuous functions, c.f.,
\cite{walsh86,khoshnevisan14}.
Such a solution theory breaks down in $ d \geq 2 $,
due to the deteriorating regularity of the spacetime white noise $ \xi $, as the dimension $ d $ increases.
In the language of stochastic PDE~\cite{hairer14,gubinelli15}, $ d=2 $ corresponds to the critical, and $d=3,4,\ldots$ the supercritical regimes.

Here we focus on the critical dimension $ d=2 $. 
To set up the problem,  fix a mollifier $ \varphi\in \Csp_\cmp^\infty(\R^2) $, where $\Csp_\cmp^\infty$ denotes smooth functions with compact support,  with $ \varphi \geq 0 $ and $ \int \varphi  \, \d x =1 $, and
mollify the noise  as
\begin{align*}
	\xi_\eps(t,x)
	:=
	\int_{\R^2} \varphi_\eps(x-y)\xi(t,y) \d y, 
	\quad 
	\varphi_\eps(x)
	:=
	\tfrac{1}{\eps^2} \varphi(\tfrac{x}{\eps}).
\end{align*}
Consider the corresponding \ac{SHE} driven by $ \xi_\e $,
\begin{align}\label{e.she}
	\partial_t Z_\eps
	=
	\tfrac12\nabla^2 Z_\eps
	+
	\sqrt{\beta_\eps}\xi_\eps Z_\eps, 
	\quad
	Z_\e = Z_\e(t,x),
	\quad
	(t,x) \in \R_+\times\R^2,
\end{align}
with a parameter $ \beta_\e>0 $ that has to be finely tuned  as $ \e\to 0 $. 
The noise $\xi_\eps$ is white in time, and we interpret the product between $ \xi_\e $ and $ Z_\e $ in the It\^o sense.
Let $\hktwo(t,x) := \frac{1}{2\pi t}\exp(-\frac{|x|^2}{2t})$, $x\in\R^2$, denote the standard heat kernel in two dimensions.
For fixed $ Z(0,x) = Z_\ic \in \Lsp^2(\R^2) $ and $ \e>0 $,
it is standard, though tedious, to show that the unique $ \Csp((0,\infty)\times\R^2) $-valued mild solution of~\eqref{e.she}
is given by the chaos expansion
\begin{align}
	\label{e.chaos.expansion}
	&Z_\eps(t,x)
	=
	\int_{\R^2}\hktwo(t,x-x') Z_\ic(x')\,\d x'
	+
	\sum_{k=1}^\infty I_{\e,k}(t,x),
\\
	\label{e.chaos}
	&I_{\e,k}(t,x)
	:=
	\int
	\Big( \prod_{s=1}^k \hktwo(\tau_{s+1}-\tau_{s},x^{(s+1)}-x^{(s)}) 
	\sqrt{\beta_\e} \xi_\e(\tau_s,x^{(s)}) \d \tau_s \d x^{(s)} \Big)
	\hktwo(\tau_1,x^{(1)}-x') Z_\ic(x')\,\d x',
\end{align}
where the integral goes over all $ 0<\tau_1<\ldots<\tau_k<t $ and $ x',x^{(1)},\ldots,x^{(k)} \in \R^2 $,
with the convention $ x^{(k+1)}:=x $ and $ \tau_{k+1}:=t $.

From the expression~\eqref{e.chaos} of $ I_{\e,k} $,
it is straightforward to check that, for fixed $ \beta_\e=\beta>0 $ as $ \e\to 0 $,
the variance $ \mathrm{Var}[I_{\e,k}] $ diverges, confirming the breakdown of the standard theory in $ d=2 $.
We hence seek to tune $ \beta_\e \to 0 $ in a way 
so that a meaningful limit of $ Z_\e $ can be observed.
A close examination shows that the divergence of $ \mathrm{Var}[I_{\e,k}] $ originates from the singularity of $ \hktwo(t,0) = (2\pi t)^{-1} $ near $ t=0 $,  
so it is natural to propose 
$ \beta_\e = \frac{\beta_0}{|\log \e|} \to 0 $, $ \beta_0>0 $.
The $ \e\to 0 $ behavior of $ Z_\e $ for small values of $ \beta_0 $ has attracted much attention recently.
For fixed $ \beta_0\in(0,2\pi) $,
\cite{caravenna17} showed that the fluctuations of $ Z_\e(t,\Cdot) $ converge (as a random measure) 
to a Gaussian field, more precisely, the solution of the two-dimensional \ac{EW} equation.
For $ \beta_0=\beta_{0,\e}\to 0 $,
\cite{feng15} showed that the corresponding polymer measure exhibits diffusive behaviors.
The logarithm $ h_\e(t,x):= \beta_\e^{-1/2} \log Z_\e(t,x) $ is also a quantity of interest:  it
describes the free energy of random polymers and the height function in surface growth phenomena 
which solves the two dimensional  \ac{KPZ} equation.
The tightness of the centered height function was obtained in~\cite{chatterjee18}
for small enough $ \beta_0 $.
It was then shown in \cite{caravenna18a} that the centered height function converges to the \ac{EW} equation
for all $ \beta_0\in(0,2\pi) $,
and in \cite{gu18} for small enough $ \beta_0 $, i.e., the limit is Gaussian.

However, the $ \e\to 0 $ behavior of $ Z_\e $ goes through a \emph{transition} at $ \beta_0 =2\pi $. 
Consider the $ n $-th order correlation function of the solution of the mollified \ac{SHE} \eqref{e.she} at a fixed time:
\begin{align}\label{e.mom.semigroup}
	u_\e(t, x_1,\ldots,x_n) := \E\Big[ \prod_{i=1}^n Z_\e(t,x_i) \Big].
\end{align}
By the It\^o calculus, this function satisfies the $ n $ particle (approximate) delta Bose gas
\begin{align}\label{e.dbg}
	\partial_t \, u_\e(t, x_1,\ldots,x_n)
	=
	-\big( \Ho_\e u_\e \big) (t, x_1,\ldots,x_n),
	\quad
	x_i\in\R^2,\quad u_\eps(0)=Z_\ic^{\otimes n},
\end{align}
where $ \Ho_\e $ is the Hamiltonian 
\begin{align}
	\label{e.hame.}
	\Ho_\e
	:=
	-\frac12\sum_{i=1}^n \nabla_i^2 -\beta_\eps\sum_{1\leq i<j\leq n} \delta_\eps(x_i-x_j),
	\qquad
	\delta_\e(x) := \e^{-2} \Phi(\e^{-1}x), \qquad
	\Phi(x) := \int_{\R^2} \varphi(x+y)\varphi(y)\, \d y,
\end{align}
with the shorthand notation $ \nabla^2_i := \nabla^2_{x_i} $.
It can be shown (e.g., from \cite[Equation~(I.5.56)]{albeverio88})
that, for $ n=2 $, the Hamiltonian $ \Ho_\e $ has a vanishing/diverging principal eigenvalue as $ \e\to 0 $, 
respectively for $ \beta_0<2\pi $ and $ \beta_0>2\pi $.
This phenomenon in turn suggests a transition in behaviors of $ Z_\e $ at $ \beta_0=2\pi $.
This transition is also demonstrated at the level of pointwise limit (in distribution) of $ Z_\e(t,x) $ as $ \eps\to0 $ by \cite{caravenna17}.

The preceding observations point to an intriguing question of
understanding the behavior of $ Z_\e $ and $u_\eps$ at this critical value $ \beta_0 =2\pi $. 
For the case of two particles ($ n=2 $),
by separating the center-of-mass and the relative motions, 
the delta Bose gas can be reduced to a system of one particle with a delta potential at the origin.
%
Based on this reduction and the analysis of the one-particle system in \cite[Chapter~I.5]{albeverio88}, 
\cite{bertini98} gave an explicit $ \e\to 0 $ limit of the second order correlation functions (tested against $ \Lsp^2 $ functions).
Further, given the radial symmetry of the delta potential, 
the one particle system (in $ d=2 $) can be reduced to an one-dimensional problem along the radial direction.  Despite its seeming simplicity, this one-dimensional problem already requires sophisticated analysis.  Although the final answer is non-trivial, it does not rule out a lognormal limit.
For $ n>2 $, these reductions no longer exist,
and to obtain information on the correlation functions stands as a challenging problem.  The only prior results are for $ n=3 $.
The work \cite{feng15} showed that for $Z_\e$ the limiting ratio of the cube root of the third pointwise moment to the square root of the second moment is not what one would expect from a lognormal distribution, indicating (but not proving) non-trivial fluctuations.
Using techniques developed in~\cite{caravenna18b} to control the chaos series,
\cite{caravenna18} obtained the convergence of the third order correlations of $ Z_\e $
to a limit given in terms of a sum of integrals.

In this paper, we proceed through a different, functional analytic route, 
and obtain a unified description of the $ \e\to 0 $ limit of all correlation functions of $ Z_\e $. 
We now prepare some notation for stating our main result.
Hereafter throughout the paper, we set
\begin{align}
	\label{e.betaeps}
	\beta_\eps
	:=
	\frac{2\pi}{|\log \eps|} + \frac{2\pi\betaf}{|\log \eps|^2},
\end{align}
where $ \betaf\in\R $ is a fixed, fine-tuning constant. 
This fine-tuning constant does not complicate our analysis, 
though the limiting expressions do depend on $ \betaf $.
Let $ \cEM =0.577\ldots $ denote the Euler--Mascheroni constant,
and, with $ \Phi $ as in~\eqref{e.she} and \eqref{e.hame.}, set
\begin{align}
	\label{e.betaC}
	\betaC := 2 \, ( \log 2 + \betaf - \betaphi - \cEM ), \qquad \betaphi := \int_{\R^4} \Phi(x_1) \log|x_1-x'_1| \Phi(x'_1) \, \d x_1 \d x'_1,
\end{align}
and
\begin{align}
	\label{e.jfn}
	\jfn(t,\betaC) &:= \int_0^\infty \frac{t^{\alpha-1}e^{\betaC\alpha}}{\Gamma(\alpha)} \, \d \alpha.
\end{align}

We will often work with vectors $ x=(x_1,\ldots,x_n) \in \R^{2n} $, where $ x_i\in\R^2 $,
and similarly $ y=(y_2,\ldots,y_n)\in\R^{2n-2} $, $ y_i\in\R^2 $.
We say $ x_i $ is the \textbf{$ i $-th component} of $ x $.
For $n\ge 2$ and $ 1\leq i<j \leq n $, consider the linear transformation $S_{ij}: \R^{2n-2} \to \R^{2n}$ %
that takes the first component of $ \R^{2n-2} $ and repeats it in the $i$-th and $j$-th components of $\R^{2n}$,
\begin{align}
	\label{e.S}
	S_{ij}(y_2,\ldots,y_n)
	:=
	(y_3,\ldots, \underbrace{y_2}_{i\text{-th}},\ldots,\underbrace{y_2}_{j\text{-th}},\ldots, y_n).
\end{align}
This operator $ S_{ij} $ the induces the \emph{lowering} operator $ \So_{ij}:\Lsp^2(\R^{2n})\to\Lsp^2(\R^{2n-2}) $
\begin{align}
	\label{e.So}
	\big( \So_{ij} u \big)(y) := u(S_{ij}y).
\end{align}

Let $ \Hsp^{\alpha}(\R^{2n}) $ denote the Sobolev space of degree $ \alpha\in\R $.
As we will show in Lemma~\ref{l.So},
\eqref{e.So} defines an \emph{unbounded} operator $ \Lsp^2(\R^{2n})\to\Lsp^2(\R^{2n-2}) $, 
and there exists an adjoint  $ \So^*_{ij}: \Lsp^2(\R^{2n-2})\to \cap_{a>1} \Hsp^{-a}(\R^{2n}) $. 
Let $$ \Pt_t := e^{\frac{t}2\sum_{i=1}^n \nabla^2_i}$$ denote the heat semigroup on $ \Lsp^2(\R^{2n}) $; its integral kernel will be denoted $ \hk(t,x) := \prod_{i=1}^n \frac{1}{2\pi t} \exp(-\frac{|x_i|^2}{2t}) $.
Define the operator $\Ptj{t}: \Lsp^2(\R^{2n-2}) \to \Lsp^2(\R^{2n-2})$,\\
\begin{align}
	\label{e.Ptj}
	\Ptj{t} 
	:= 
	\jfn(t,\betaC)e^{\frac{t}4\nabla^2_2+\frac{t}{2}\sum_{i=3}^n\nabla^2_i}.
\end{align}
This operator `squeezes' the first component $ x_1 $ in the heat semigroup and multiplies the result by the function $ \jfn(t,\betaC) $. %
The function is related to the operator $ \Jo_z $ defined later in \eqref{e.Jo}, c.f., Lemma~\ref{l.jfn}, and hence the notation $ \Ptj{t} $. %

We need to prepare some index sets.
Hereafter we write $ i<j $ for a pair of ordered indices in $ \{1,\ldots,n\} $, i.e.,
two elements $ i<j $ of $\{1,\ldots,n\} $.
For $ n,m\in \Z_+ $, 
we consider $ \Vec{(i,j)} = ((i_k,j_k))_{k=1}^m $ such that 
$ (i_{k}<j_{k}) \neq (i_{k+1}<j_{k+1}) $, i.e., $m$ \emph{ordered pairs with consecutive pairs non-repeating}.
Let
\begin{align}
	\label{e.diag.set.m}
	\diag(n,m) 
	&:= 
	\big\{ 
		\Vec{(i,j)} \in (\{1,\ldots,n\}^2)^m
		:
		(i_{k}<j_{k}) \neq (i_{k+1}<j_{k+1})
	\big\},
\\
	\label{e.diag.set}
	\diag(n)
	&:= 
	\bigcup_{m=1}^\infty
	\diag(n,m)
\end{align}
denote the sets of all such indices, with the convention that $ \diag(1,m):=\varnothing $, $ m\in\Z_+ $.
The notation $ \diag(n) $ refers to `diagrams', as will be explained  in Section~\ref{sect.diagram}.
Let
\begin{align}
	\label{e.Sigma.m}
	\Sigma_m(t) := \big\{ \vec{\tau}=(\tau_a)_{a\in\frac12\Z\cap[0,m]} \in \R^{2m+1}_+ : \tau_0+\tau_{1/2}+\ldots+\tau_{m} =t \big\},
\end{align}
so that for a fixed $ t\in\R_+ $, 
the integral $ \int_{\Sigma_m(t)} (\,\Cdot\,) \d\vec{\tau} $ denotes a $ (2m+1) $-fold convolution over the set $ \Sigma_m(t) $.
For a bounded operator $ \geno : \genH \to \genH' $ between Hilbert spaces $ \genH $ and $ \genH' $,
let 
$
	\normo{\geno} := \sup_{\norm{u}_{\genH}=1} \norm{\geno u}_{\genH'} 
$
denote the inherited operator norm.
We use the subscript `op' (standing for `operator') to distinguish the operator norm from the vector norm,
and omit the dependence on $ \genH $ and $ \genH' $, since the spaces will always be specified along with a given operator. 
The $ \Lsp^2 $ spaces in this paper are over $ \C $,
and we write $ \ip{f}{g} := \int_{\R^d} \bar{f(x)} g(x) \, \d x $ for the inner product.
(Note our convention of taking complex conjugate in the first function.)
Throughout this paper we use $ C(a,b,\ldots) $ to denote a generic positive finite constant that may change
from line to line, but depends only on the designated variables $ a,b,\ldots $.
We view the mollifier $ \varphi $ as fixed throughout this paper, so the dependence on $ \varphi $ will not be specified.

We can now state our main result.
\begin{theorem}\label{t.main}
\begin{enumerate}[label=(\alph*), leftmargin=20pt]
\item[] 
\item\label{t.main.diagram}
The operators
\begin{align}
	\label{e.Diagram.op}
	\hspace{.15\textwidth}
	\Pt_t + \Dio^{\diag(n)}_t,
	\hspace{.05\textwidth}
	\Dio^{\diag(n)}_t
	:=
	\sum\nolimits_{\Vec{(i,j)}\in\diag(n)}\Dio^{\Vec{(i,j)}}_t,
	\quad
	t \geq 0,	
\end{align}
define a norm-continuous semigroup on $ \Lsp^2(\R^{2n}) $, where, 
for $ \Vec{(i,j)} = ((i_k,j_k))_{k=1}^m $,
\begin{align}
	\label{e.diagram.op}
	\Dio^{\Vec{(i,j)}}_t
	&:=
	\int_{\Sigma_m(t)}
	\Pt_{\tau_0} \So^*_{i_1j_1} \big( 4\pi\Pt^\Jo_{\tau_{1/2}} \big) 
	\Big( \prod_{k=1}^{m-1} \So_{i_{k}j_{k}} \Pt_{\tau_{k}} \So^*_{i_{k+1}j_{k+1}} \, (4\pi\Pt^\Jo_{\tau_{k+1/2}}) \Big) 
	\So_{i_mj_m} \Pt_{\tau_{m}}
	\
	\d \vec{\tau}.
\end{align}
The sum in~\eqref{e.Diagram.op} converges absolutely in operator norm, uniformly in $ t $ over compact subsets of $ [0,\infty) $.
%
\item\label{t.main.dcnvg}
Start the mollified \ac{SHE}~\eqref{e.she} from $ Z_\e(0,\Cdot) = Z_\ic(\Cdot) \in \Lsp^2(\R^2) $.
For any $ f(x)=f(x_1,\ldots,x_n)\in\Lsp^2(\R^{2n}) $, $ n\in\Z_+ $, we have
\begin{align}
\label{e.mom.cnvg}
\begin{split}
	\E\big[ \ip{f}{Z_{\e,t}^{\otimes n}} \big] :=
	\E\Big[ \int_{\R^{2n}} \bar{f(x)} \prod_{i=1}^n Z_{\e}(t,x_i) \, \d x \Big] 	
	\longrightarrow 
	\IP{ f } 
	{%
		\big( \Pt_t + \Dio^{\diag(n)}_t \big)
	\, 
	 Z_\ic^{\otimes n}
	}
	\quad
	\text{as } \e\to 0,
\end{split}
\end{align}
uniformly in $ t $ over compact subsets of $ [0,\infty) $.
\end{enumerate}
\end{theorem}
%

\begin{remark} 
Since the method is through explicit construction of 
a convergent series for the resolvent on $\Lsp^2(\R^{2n})$, 
our result does not apply to the flat initial condition $ Z_\ic(x)\equiv 1 $.
We conjecture that Theorem~\ref{t.main} extends to such initial data,
and leave this to future work. 
\end{remark}


Theorem~\ref{t.main} gives a complete characterization of the $ \e\to0 $ limit of fixed time, 
correlation functions  of the \ac{SHE} with an $ \Lsp^2 $ initial condition.
We will show in Section~\ref{sect.diagram} that for each $ \Vec{(i,j)}\in\diag(n) $, $ \Dio^{\Vec{(i,j)}} $ possesses an explicit integral kernel.
Hence the limiting correlation functions (i.e., r.h.s.\ of~\eqref{e.mom.cnvg}) can be expressed as a sum of integrals.
From this expression,
we check (in Remark~\ref{rmk.BCmatching}) that for $ n=2 $ our result matches that of \cite{bertini98},
and for $ n=3 $, we derive (in Proposition~\ref{p.3rdmom}) an analogous expression of \cite[Equations~(1.24)--(1.26)]{caravenna18}.

A question of interest arises as to whether one can uniquely characterize the limiting process of $ Z_\e $.
This does not follow directly from correlation functions, or moments,
since we expect a very fast moment growth in $ n $ (see Remark~\ref{rmk.eigenvalue}).
Still, as a simple corollary of Theorem~\ref{t.main},
we are able to infer that every limit point of $ Z_\e $ must have correlation functions given by the r.h.s.\ of \eqref{e.mom.cnvg}.
The corollary is mostly concretely stated in terms of the vague topology of measures,
or equivalently testing measures against compactly supported continuous functions.
One could generalize to $ \Lsp^2 $ test functions but we do not pursue this here.
\begin{corollary}\label{c.main}
Let $ Z_\ic $ and $ Z_{\e}(t,x) $ be as in Theorem~\ref{t.main},
and, for each fixed $ t $, view $ \mu_{\e,t}(\d x) := Z_{\e}(t,x) \d x $ as a random measure.
Then, for any fixed $ t\in\R_{+} $, 
the law of $ \{\mu_{\e,t}(\d x) \}_{\e\in(0,1)} $ is tight in the vague topology,
and, for any limit point $ \mu_{*,t}(\d x) $ of $ \{\mu_{\e,t}(\d x) \}_{\e\in(0,1)} $,
and for any compactly supported, continuous $ f_1,\ldots,f_n \in\Csp_\mathrm{c}(\R^2) $, $ n\in\Z_+ $, 
\begin{align}
	\label{e.limit.point.mom}
	\E\Big[ \prod_{i=1}^n \int_{\R^2} \bar{ f_i(x_i) } \mu_{*,t}(\d x_i)  \Big] 
	=
	\IP{ f_1\otimes\cdots\otimes f_n } 
	{%
		\big( \Pt_t + \Dio^{\diag(n)}_t \big)
	\, 
		Z_\ic^{\otimes n}
	}.
\end{align}
Furthermore, if $ Z_\ic(x), f(x) \geq 0 $ are nonnegative and not identically zero, then
\begin{align}
	\label{e.nonGaussian}
	\E\Big[ \Big( \int_{\R^2} f(x) \mu_{*,t}(\d x) - \E\Big[\int_{\R^2} f(x) \mu_{*,t}(\d x)\Big] \Big)^3 \Big] 
	>0.
\end{align}
\end{corollary}

Due to the critical nature of our problem, as $ \e \to 0 $
the moments go  through a non-trivial transition as $ \beta_0$ passes through $2\pi $.
To see this, in~\eqref{e.chaos.expansion},
use the orthogonality $ \E[I_{\e,k}(t,x_1)I_{\e,k'}(t,x_2)]=0 $, $ k\neq k' $, to express the second ($ n=2 $) moment as
\begin{align*}
	\E\Big[ \Big(\int_{\R^2} Z_\e(t,x) f(x) \d x\Big)^2 \Big]
	=
	\int_{\R^8} \prod_{i=1}^2 \hktwo(t,x_i-x'_i) f(x_i) Z_\ic(x'_i)\,\d x'_i \d x_i
	+
	\sum_{k=1}^\infty \int_{\R^4} \E\Big[ \prod_{i=1}^2 I_{\e,k}(t,x_i) f(x_i) \Big] \d x_1\d x_2.
\end{align*}
As seen in \cite{caravenna18},
the major contribution of the sum spans across a \emph{divergent} number of terms ---
across all $ k $'s of order $ |\log\e| \to \infty $.
We are probing a regime where the limiting process `escapes' to indefinitely high order chaos as $\eps\to0$, reminiscent of the large time behavior of the \ac{SHE}/KPZ equation in $d=1$.

Because of this,
obtaining the $ \e\to 0 $ limit from chaos expansion requires elaborate and delicate analysis.
In fact, just to obtain an $ \e $-independent bound (for fixed $ Z_\ic $ and test functions $ f_i $'s) from the chaos expansion is a challenging task.
Such analysis is carried out for $ n=2,3 $ in \cite{caravenna18} 
(in a discrete setting and in the current continuum setting, both with $ Z_\ic\equiv 1 $).  
Here, we progress through a different route.  
From \eqref{e.mom.semigroup}, \eqref{e.dbg}, and \eqref{e.hame.} obtaining the limit of the correlation functions 
is equivalent to obtaining the limit of the semigroup $ e^{-t\Ho_\e} $,
which reduces to the study of $ \Ho_\e $ itself, or its resolvent.

The delta Bose gas enjoys a long history of study, 
motivated in part by phenomena such as unbounded ground-state energy and infinite discrete spectrum observed in $ d=3 $.
We do not survey the literature here, and refer to the references in \cite{albeverio88}.
Of most relevance to this paper is the work~\cite{dimock04},
which studied  $ d=2 $  with a momentum cutoff, 
and established the convergence of the resolvent of the Hamiltonian to an explicit limit~\cite[Equation~(90)]{dimock04}.
%
Here, we follow the framework of~\cite{dimock04},
but instead of the momentum cutoff, we work with the space-mollification scheme as in~\eqref{e.hame.},
in order to connect the delta Bose gas to the \ac{SHE}.

Hereafter we always assume $ n\geq 2 $, since the $ n=1 $ case of Theorem~\ref{t.main} is trivial.
We write $ \Io $ for the identity operator in Hilbert spaces.
For $ z\in\C\setminus[0,\infty) $, let
\begin{align}
	\label{e.Green}
	\Go_z 
	:= \Big( -\frac12 \sum_{i=1}^n \nabla^2_i - z\Io \Big)^{-1}
\end{align}
denote the resolvent of the free Laplacian in $ \R^{2n} $. Let $ \Jo_{z} $ be the unbounded operator $ \Lsp^2(\R^{2n-2})\to \Lsp^2(\R^{2n-2}) $ defined via its Fourier transform
\begin{align}
	\label{e.Jo}
	\hat{ \Jo_{z}v }(p_{2-n}) := \log(\tfrac12|p|^2_{2-n}-z) \hat{v}(p_{2-n}),
\end{align} 
where $ p_{2-n} :=(p_2,\ldots,p_n)\in\R^{2n-2} $ 
and
\begin{align*}
	|p|^2_{2-n} := \tfrac12|p_2|^2+|p_3|^2+\ldots+|p_n|^2,
\end{align*}
with domain
$
	\Dom(\Jo_{z}) := \{ v\in\Lsp^2(\R^{2n-2}) : \int_{\R^{2n}} \big| \hat{v}(p_{2-n})\log(|p|^2_{2-n}+1) \big|^2 \d p_{2-n} < \infty \}$.

Let $ \Lsp^{2}_\sym(\R^{2n}) $ denote the subspace of $ \Lsp^2(\R^{2n}) $ consisting of functions symmetric in the $ n $-components,
i.e., $ u(x_1,\ldots,x_n)=u(x_{\sigma(1)},\ldots,x_{\sigma(n)}) $, for all permutation $ \sigma\in\mathbb{S}_n $.
Recall $ \betaC $ and $ \betaf $ from~\eqref{e.betaC}.
As the major step toward proving Theorem~\ref{t.main}, in Sections~\ref{sect.resolvent1}--\ref{sect.resolcnvg}, we show
\begin{proposition}[Limiting resolvent]
\label{p.resolvent}
There exists $ C<\infty $ such that, for $ z\in\C $ with $ \Re(z) < -e^{Cn^2+\betaC} $, 
\begin{enumerate}[label=(\alph*), leftmargin=20pt]
\item\label{p.resolvent.}
	the following defines a bounded operator on $ \Lsp^2(\R^{2n})\to\Lsp^2(\R^{2n}) $:
	\begin{align}
	\label{e.resolvent}
		\Ro_{z}
		&=
		\Go_{z} 
		+ 
		\sum_{m=1}^\infty \ \sum_{\Vec{(i,j)}\in\diag(n,m)}
	\Go_z \So^{*}_{i_1j_1} 
		\, 
		\big( 4\pi(\Jo_{z}-\betaC\Io)^{-1} \big) \, \prod_{s=2}^{m} \Big(\So_{i_{s-1}j_{s-1}}\Go_z\So^*_{i_sj_s} \, \big( 4\pi(\Jo_{z}-\betaC\Io)^{-1} \big) \Big) 
		\, 
		\So_{i_{m}j_{m}}\Go_z,
	\end{align}
	where the sum converges absolutely in operator norm;
\item \label{p.resolvent.sym}
	when restricted to $ \Lsp^2_\sym(\R^{2n}) $, the operator takes a simpler form,
	\begin{align}
	\begin{split}
		\label{e.resolvent.sym}
		&\Ro^\sym_{z}
		:=
		\Go_{z} 
	+ 
		\frac{2}{n(n-1)}\Big( \sum_{i<j} \Go_z\So^*_{ij} \Big) 
		\, 
		\Big( \frac{1}{4\pi} \big( \Jo_{z} - \betaC \Io \big) - \frac{2}{n(n-1)} \dsum \So_{ij}\Go_z\So^*_{k\ell} \Big)^{-1} 
		\, 
		\Big( \sum_{i<j} \So_{ij}\Go_z \Big).
	\end{split}
	\end{align}
	The sum $ \dsum $ is over distinct pairs $ (i<j)\neq (k<\ell) $.
\end{enumerate}
\end{proposition}

\begin{remark}
The leading term $\frac{2\pi}{|\log \eps|} $ of $\beta_\eps$ in \eqref{e.betaeps} is easily seen to arise from the divergence in $\So_{ij}\Go_z\So^*_{ ij}$ when we replace $\So_{ ij}$  by approximate
versions  $\So_{\e ij}$.  See the discussion following \eqref{e.Pro}.
\end{remark}

\begin{theorem}[Convergence of the resolvent]
\label{t.resolcnvg}
There exist constants $ C_1, C_2(\betaf)<\infty $, where $ C_1 $ is universal while $ C_2(\betaf) $ depends only on $ \betaf $,
such that for all $ \e\in(0,1/C_2) $, for $ z\in\C $ with $ \Re(z) < -e^{C_1n^2+\betaC} $, and for $\Ho_{\e}$ defined in \eqref{e.hame.},
\begin{enumerate}[label=(\alph*), leftmargin=20pt]
\item\label{t.resolcnvg.e}
	$ (\Ho_{\e} -z) $ has a bounded inverse $ \Lsp^2(\R^{2n})\to\Lsp^2(\R^{2n}) $;
\item \label{t.resolcnvg.cnvg}$
	\Ro_{\e,z} 
	:=
	(\Ho_{\e} -z)^{-1}
	\longrightarrow
	\Ro_{z}$
	 strongly on $\Lsp^2(\R^{2n})$, as $\eps\to 0$.
\end{enumerate}
\end{theorem}

\begin{remark}
In stating and proving Proposition~\ref{p.resolvent} and Theorem~\ref{t.resolcnvg} we have highlighted the dependence on $ \betaf $. %
For the purpose of this paper, keeping the dependence is unnecessary (since $ \betaf $ can be fixed throughout), %
but we choose to do so for its potential future applications. %
\end{remark}

\begin{remark}
\label{rmk.eigenvalue}Given Theorem~\ref{t.resolcnvg},
by the Trotter--Kato Theorem, c.f., \cite[Theorem~VIII.22]{reed72}, 
there exists an (unbounded) self-adjoint operator $ \Ho $ on $ \Lsp^2(\R^{2n}) $, \emph {the limiting Hamiltonian},
such that $ \Ro_z =(\Ho-z\Io)^{-1} $, $ \Im(z)\neq 0 $. 
As implied by Theorem~\ref{t.resolcnvg}, the spectra of $ \Ho_\e $ and $ \Ho $ are bounded below by $ -e^{Cn^2+\betaC} $.
Such a bound is first obtained under the momentum cutoff in \cite{dell94}.
The prediction \cite{rajeev99}, based on a non-rigorous mean-field analysis, is that the lower end of the spectrum of $ \Ho $ should approximate $ -e^{c_\star n } $, for some $ c_\star\in(0,\infty) $ that depends on $ \betaf $.
\end{remark}

\begin{remark}\label{remheur}
One can match $ e^{-t\Ho} $ to the operator $ \Pt_t+\Dio^{\diag(n)}_t $ on r.h.s.\ of~\eqref{e.mom.cnvg}
\emph{heuristically} by taking the inverse Laplace transform of $ \Ro_z $ in~\eqref{e.resolvent} in $ z $.
At a formal level, doing so turns the operators $ \Go_{\Cdot} $ and $ (\Jo_{{\Cdot}}-\betaC\Io)^{-1} $ into $ \Pt_{\Cdot} $ and $ \Pt^\Jo_{\Cdot} $ respectively,
and the products of operators in $ z $ become the convolutions in $ t $.
\end{remark}

\begin{remark}
It is an interesting question whether the resolvent method, which is applied to the critical window in this paper, also applies to the subcritical regime $ \beta_0<2\pi $. %
In the subcritical regime, it is the \emph{fluctuations} $ |\log\e|^{1/2}(Z_\e-1) $ that converge to the \ac{EW} equation, as shown in \cite{caravenna17} using a chaos expansion. %
In order to apply the resolvent method, one needs to center and scale the correlation functions \eqref{e.mom.semigroup}. %
The result on the convergence of the two point correlation function is a straightforward application of the resolvent method. Analyzing the higher order correlation functions under such centering and scaling is an interesting open question. %
\end{remark}

\begin{remark}[\ac{SHE} in $ d\geq 3 $]
In higher dimensions $d\geq 3$, the appropriate tuning parameter is $\beta_\eps=\beta_0\eps^{d-2}$. 
For small $\beta_0$, the studies on the \ac{EW}-equation limit of the SHE/KPZ equation include \cite{magnen2018scaling,gu2018edwards,dunlap2018fluctuations}, 
and results on the pointwise fluctuations of $Z_\eps$ and the phase transition in $\beta_0$ can be found in \cite{mukherjee2016weak,comets2017rate,comets2018fluctuation,comets2019renormalizing,cosco2019gaussian}. 
For discussions on directed polymers in a random environment,
we refer to \cite{comets17} and the references therein.
\end{remark}

\subsection*{Outline}
In Section ~\ref{sect.diagram} we give an explicit expression for the limiting semigroup in terms of diagrams and use this to derive Corollary ~\ref{c.main} from Theorem ~\ref{t.main}.
In Section~\ref{sect.resolvent1}, we derive the key expression \eqref{e.resolvent.e} for the resolvent $ \Ro_{\e,z} $,
 which allows the limit to be taken term by term: The limits are obtained in 
Sections \ref{sect.in.out} through \ref{sect.diagonal}, and these are used in  Section
\ref{sect.resolcnvg} to
prove  Proposition~\ref{p.resolvent}\ref{p.resolvent.}--\ref{p.resolvent.sym}, Theorem~\ref{t.resolcnvg}\ref{t.resolcnvg.e}--\ref{t.resolcnvg.cnvg} and the convergence part of Theorem~\ref{t.main}\ref{t.main.dcnvg}.
In Section~\ref{sect.laplace}, we complete the proof of Theorem~\ref{t.main} by constructing the semigroup and matching its Laplace transform to the limiting resolvent $ \Ro_z $.

\subsection*{Acknowledgment}
We thank Davar Khoshnevisan, Lawrence Thomas, and Horng-Tzer Yau for useful discussions.
YG was partially supported by the NSF through DMS-1613301/1807748 and the Center for Nonlinear Analysis of CMU.
JQ was supported by an NSERC Discovery grant.
LCT was partially supported by a Junior Fellow award from the Simons Foundation,
and by the NSF through DMS-1712575. 

\section{Diagram expansion}
\label{sect.diagram}
In this section, we give an explicit integral kernel $ \dker^{\Vec{(i,j)}}(t,x,x') $ 
of the operator $ \Dio^{\Vec{(i,j)}}_t $ in Theorem~\ref{t.main}.
and show how the kernel $ \dker^{\Vec{(i,j)}}(t,x,x') $  can be encoded in terms of diagrams.  This is then used to show how Corollary ~\ref{c.main} follows from
Theorem~\ref{t.main}.
%
%
%
%
The operators $\So_{ij}\Pt_t$, $\Pt_t\So_{ij}^*$ and $\So_{ij}\Pt_t\So_{k\ell}^*$ have integral kernels
\begin{align}
	\label{e.Pt.in.ker}
	&\big(\So_{ij}\Pt_t u\big)(y) = \int_{\R^{2n}} \hk(t,S_{ij}y-x)  u(x) \, \d x,
	\qquad
	y=(y_2,\ldots,y_n) \in\R^{2n-2},
\\
	\label{e.Pt.out.ker}
	&\big(\Pt_t\So_{ij}^*v\big)(x) = \int_{\R^{2n-2}} \hk(t,x-S_{ij}y) v(y) \,\d y,
	\qquad
	x=(x_1,\ldots,x_n)\in\R^{2n},
\\
	\label{e.Pt.med.ker}
	&\big(\So_{ij}\Pt_t\So_{k\ell}^* v\big)(y) = \int_{\R^{2n-2}} \hk(t,S_{ij}y-S_{k\ell}y')  v(y') \, \d y',
	\qquad
	y=(y_2,\ldots,y_n) \in\R^{2n-2}.
\end{align}%
From this we see that $ \Dio^{\Vec{(i,j)}}_t $ has  integral kernel
\begin{align}
	\notag
	\dker^{\Vec{(i,j)}}&(t,x,x')
	=
	\int_{\Sigma_m(t)} \d \vec{\tau}
	\int \hk\big(\tau_0,x-S_{i_1j_1}y^{(1/2)}\big) \, \d y^{(1/2)}
	\cdot
	4\pi\hk^\Jo\big(\tau_{1/2},y^{(1/2)}-y^{(1)}\big) \, \d y^{(1)} 
\\
	\label{e.diagram.}
	&\hspace{-10pt}\cdot
	\prod_{k=1}^{m-1} 
	\Big( 
		\hk\big(\tau_k,S_{i_kj_k}y^{(k)}-S_{i_{k+1}j_{k+1}}y^{(k+1/2)}\big) \, \d y^{(k+1/2)}
		\cdot 4\pi\hk^\Jo\big(\tau_{k+1/2},y^{(k+1/2)}-y^{(k+1)}\big) \, \d y^{(k+1)}
	\Big)
\\
	\notag
	& \hspace{-10pt}\cdot \hk\big(\tau_m,S_{i_mj_m}y^{(m)}-x'\big),
\end{align} 
where $ \Sigma_m(t) $ is defined in~\eqref{e.Sigma.m}, $x,x'\in\R^{2n}$,
and $ y^{(a)}\in\R^{2n-2} $ with $ a\in(\frac12\Z)\cap(0,m] $.

We wish to further reduce~\eqref{e.diagram.} to an expression 
that involves only the two-dimensional heat kernel $ \hktwo(\tau,x_i) $ and $ \jfn(\tau,\betaC) $.
Recall from~\eqref{e.S} that $ (S_{ij}y):=x $ is a vector in $ \R^{2n} $ such that $ x_i=x_j $.
In~\eqref{e.diagram.}, we write 
\begin{align*}
	S_{i_kj_k}y^{(a)}
	=
	(y^{(a)}_3,\ldots, \underbrace{y^{(a)}_2}_{i_k\text{-th}},\ldots,\underbrace{y^{(a)}_2}_{j_k\text{-th}},\ldots,y^{(a)}_n)
	=
	(x^{(a)}_1,\ldots,x^{(a)}_n) \ind\set{x^{(a)}_{i_k}=x^{(a)}_{j_k}},
\end{align*}
and accordingly,
$
	\d y^{(a)} = \d' x^{(a)},
$
where $ a=k-\frac12,k $.
The vector $ x^{(a)} $ is in $ \R^{2n} $, 
but the integrator $ \d' x^{(a)} $ is $ (2n-2) $-dimensional due to the contraction $ x^{(a)}_{i_k}=x^{(a)}_{i_k} $.
More explicitly,
\begin{align*}
	\d' x^{(a)} 
	:= 
	\Big(\d x^{(a)}_{i_k}\,\prod_{\ell\neq i_k,j_k} \d x^{(a)}_\ell\Big)
	=
	\Big(\d x^{(a)}_{j_k}\,\prod_{\ell\neq i_k,j_k} \d x^{(a)}_\ell\Big),
	\qquad
	a=k-\tfrac12,k. 
\end{align*}
We express $ \hk $ as the product of two dimensional heat kernels, i.e., $ \hk(\tau,x)=\prod_{\ell=1}^n \hktwo(\tau,x_\ell) $ with $x=(x_1,\ldots,x_n)$,
and similarly for $ \hk^\Jo(\tau,\Cdot) $; see \eqref{e.Ptj.ker} in the following for the explicit expression.
This gives
\begin{align}
	\notag
	\dker^{\Vec{(i,j)}}&(t,x,x')
	:=
	\int_{\Sigma_m(t)} \d \vec{\tau}
	\int \prod_{\ell=1}^n \hktwo\big(\tau_0,x_\ell-x^{(1/2)}_\ell\big) \ind{\big\{ x^{(1/2)}_{i_1}= x^{(1/2)}_{j_1} \big\} } \, \d' x^{(1/2)}
\\
	\notag
	&\cdot
	4\pi\jfn\big(\tau_{1/2},\betaC\big) 
	\hktwo\big(\tfrac12\tau_{1/2},x^{(1/2)}_{i_1}-x^{(1)}_{i_1}\big) 
	\prod_{\ell\neq i_1,j_1} \hktwo(\tau_{1/2},x^{(1/2)}_\ell-x^{(1)}_\ell) \, \d' x^{(1)}
\\
	\label{e.diagram}
	&\cdot
	\prod_{k=1}^{m-1} 
		\Bigg( 
			\prod_{\ell=1}^n  \ind{ \big\{ x^{(k)}_{i_{k}}= x^{(k)}_{j_{k}} \big\} } 
			\hktwo\big( \tau_k, x^{(k)}_\ell-x^{(k+1/2)}_\ell \big) 
			\ind{\big\{ x^{(k+1/2)}_{i_{k+1}}= x^{(k+1/2)}_{j_{k+1}} \big\}} 
			\d' x^{(k)}
\\
	\notag
	&\hphantom{\cdot\prod_{k=2}^m\Bigg(}
			\cdot
			4\pi\jfn\big( \tau_{k+1/2},\betaC \big) 
			\hktwo\big( \tfrac12\tau_{k+1/2},x^{(k+1/2)}_{i_{k+1}}-x^{(k+1)}_{i_{k+1}} \big) 
			\prod_{\ell\neq i_k,j_k} \hktwo\big(\tau_{k+1/2},x^{(k+1/2)}_\ell-x^{(k+1)}_\ell\big) \, \d' x^{(k+1)}
		\Bigg)
\\
	\notag
	&\cdot
	\prod_{\ell=1}^n \hktwo\big(\tau_{m},x^{(m)}_\ell-x'_\ell\big).
\end{align} 

This complicated looking formula can be conveniently recorded in terms of diagrams.
Set $ A:=(\frac12\Z)\cap[0,m+\frac12] $, and adopt the convention $ x^{(0)}:=x $ and $ x^{(m+1/2)}:=x' $.
We schematically represent spacetime $ \R_+\times\R^{2} $ by the plane,
with the horizontal direction being the time axis $ \R_+ $, and the vertical direction representing space $ \R^2 $.
We put dots on the plane representing $ x^{(a)}_\ell $, $ a\in A $.
Dots with smaller $ a $ sit to the left of those with bigger $ a $, and those with the same $ a $ lie on the same vertical line.
The horizontal distance between $ x^{(a-1/2)}_\ell $ and $ x^{(a)}_\ell $, $ a\in A $, 
represents a time lapse $ \tau_{a} >0 $.
We fix the time horizon between $ x_\ell = x^{(0)}_\ell $ and $ x'_{\ell}=x^{(m+1/2)}_\ell $ to be $ t $, 
which forces $ \tau_{0}+\tau_{1/2}+\ldots+\tau_m=t $.
The points $ x^{(a)}_\ell $, are generically represented by distinct dots,
expect that $ x^{(a)}_{i_k} $ and $ x^{(a)}_{j_k} $ are joined for $ k=a-1/2,a $.
In these cases we call the dot double, otherwise single.
See Figure~\ref{f.diagramo} for an example with $ n=4 $ and $ \Vec{(i,j)}=((1<2),(2<3),(3<4)) $.

\begin{figure}[h]
\includegraphics[width=.75\textwidth]{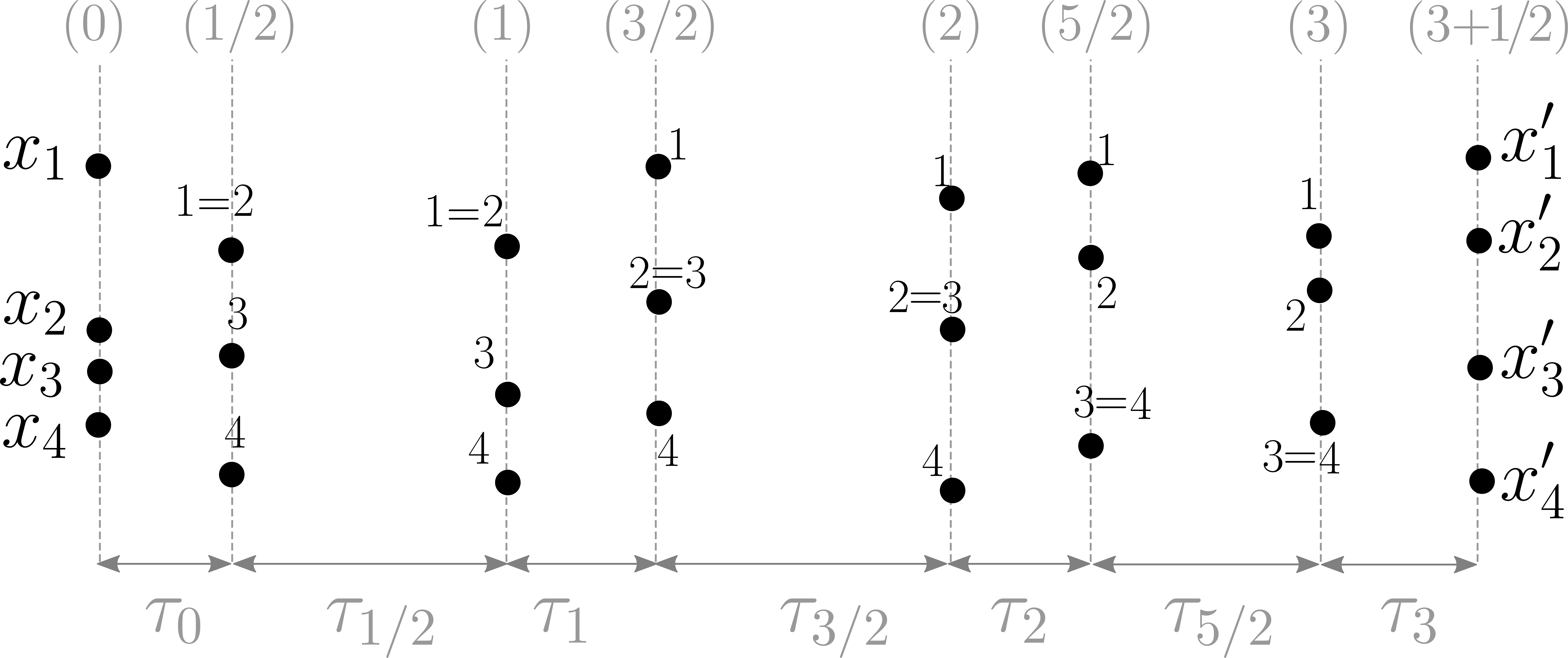}
\caption{%
Schematic representation of $ x^{(a)}_\ell $, with $ n=4 $ and $ \vec{(i,j)}=((1<2),(2<3),(3<4)) $.
Each dot represents a point $ x^{(a)}_\ell $, $ a\in(\frac12\Z)\cap[0,3+\frac12] $,
with the convention $ x_\ell := x^{(0)}_\ell $ and $ x'_\ell := x^{(3+1/2)}_\ell $.
In the figure, the $ \ell $ indices are printed in black next to the dot, 
while the $ a $ superscripts are put over the vertical, dashed line.
The horizontal distances between dash lines represent time lapses $ \tau_a $.%
}
\label{f.diagramo}
\end{figure}

Next, connect dots that represent $ x^{(a-1/2)}_\ell $ and $ x^{(a)}_\ell $ together,
by a `single' line except for the case when both ends are double points, by a `double' line otherwise. 
To each regular line we assign a two-dimensional heat kernel $ \hktwo(\tau_{a},x^{(a-1/2)}_\ell-x^{(a)}_\ell) $,
and to each double line assign the quantity
$ 4\pi\jfn(\tau_{a},\betaC)\hktwo(\tfrac12\tau_{a},x^{(a-1/2)}_\ell-x^{(a)}_\ell) $.
The kernel $ \dker^{\Vec{(i,j)}}(t,x,x') $ is then obtained by multiplying together 
the quantities assigned to the (regular and double) lines,
and integrate the $ x^{(a)} $'s and $ \tau_a $'s, with the points $ x_\ell := x^{(0)}_\ell $ and $ x'_{\ell}=x^{(m+1/2)}_\ell $ being fixed.
See Figure~\ref{f.diagram} for an example with $ n=4 $ and $ \Vec{(i,j)}=((1<2),(2<3),(3<4)) $.
\begin{figure}[h]
\includegraphics[width=.75\textwidth]{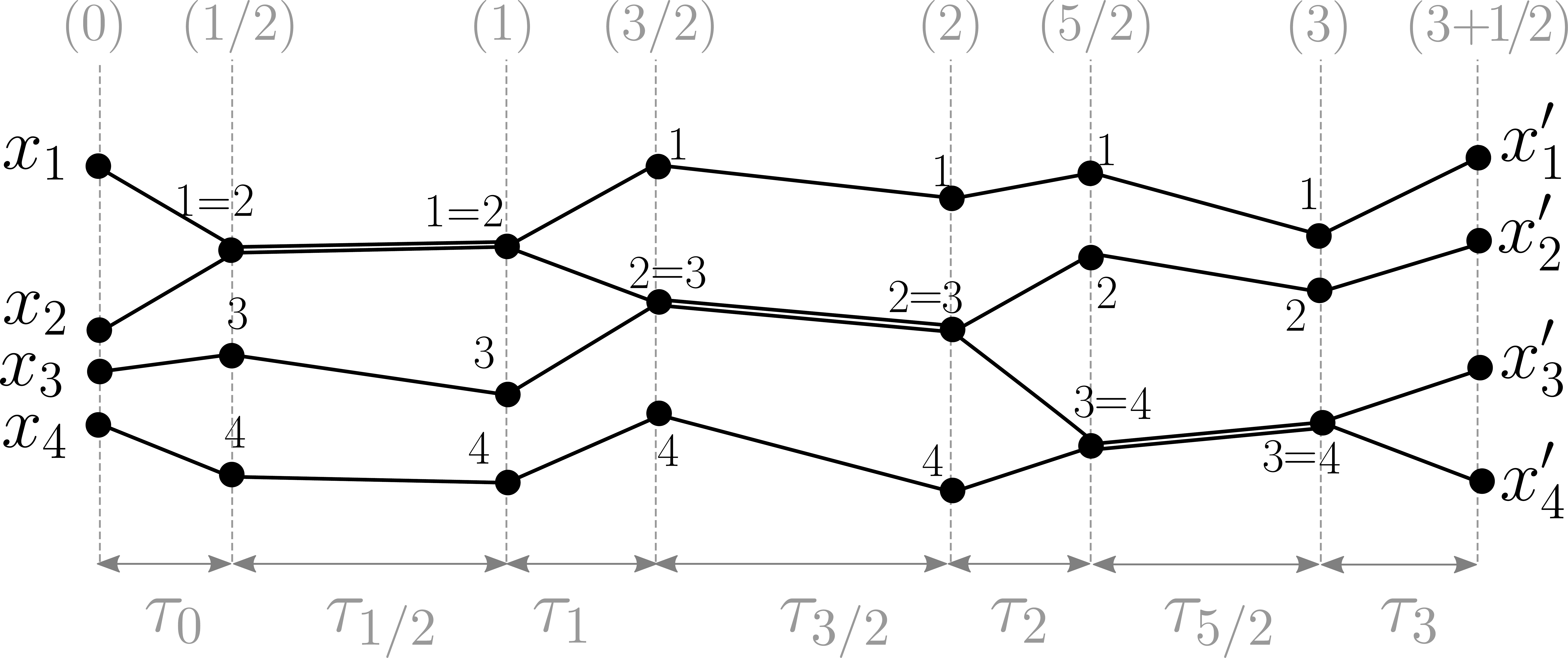}
\caption{%
The diagram representation for $ \dker^{\vec{(i,j)}}(t,x,x') $, with $ n=4 $ and $ \vec{(i,j)}=((1<2),(2<3),(3<4)) $.
Each regular (single) line between dots is assigned $ \hktwo(\tau,x^{(a-1/2)}_\ell-x^{(a)}_\ell) $,
while each double line is assigned $ 4\pi\jfn(\tau,\betaC)\hktwo(\tfrac12\tau,x^{(a-1/2)}_\ell-x^{(a)}_\ell) $,
where $ x^{(a-1/2)}_\ell $ and $ x^{(a)}_\ell $ are represented by the dots at the two ends,
and $ \tau $ is the horizontal distance between these dots.%
}
\label{f.diagram}
\end{figure}

In the follow two subsections, we examine the $ n=2,3 $ cases, and derive some useful formulas.

\subsection{The $ n=2 $ case}
In this case, the only index is the singleton $ \Vec{(i,j)}=((1<2)) $, whereby
\begin{subequations}
\label{e.n=2.}
\begin{align}
	\label{e.n=2.1}
	\big( \hk + \dker^{\diag(2)} \big)(t,x_1,x_2,x'_1,x'_2)
	&=
	\prod_{\ell=1}^2 \hktwo(t,x_\ell-x'_\ell)
	+
	\int_{\tau_0+\tau_{1/2}+\tau_1=t} \d \vec{\tau}
	\int \prod_{\ell=1}^2 \hktwo\big(\tau_0,x_\ell-x^{(1/2)}_1\big) \d x^{(1/2)}_1
\\
	\label{e.n=2.2}
	&\cdot
	4\pi\jfn\big(\tau_{1/2},\betaC\big)
	\hktwo\big(\tfrac12\tau_{1/2},x^{(1/2)}_{1}-x^{(1)}_{1}\big) 
	\d x^{(1)}_1
\\
	\label{e.n=2.3}
	&\cdot
	\prod_{\ell=1}^2 \hktwo\big(\tau_{1},x^{(1)}_1-x'_\ell\big),
\end{align} 
\end{subequations}
and the diagram of $ \dker^{((12))}(t,x,x') $ is given in Figure~\ref{f.n=2}.

\begin{figure}[h]
\includegraphics[width=.45\textwidth]{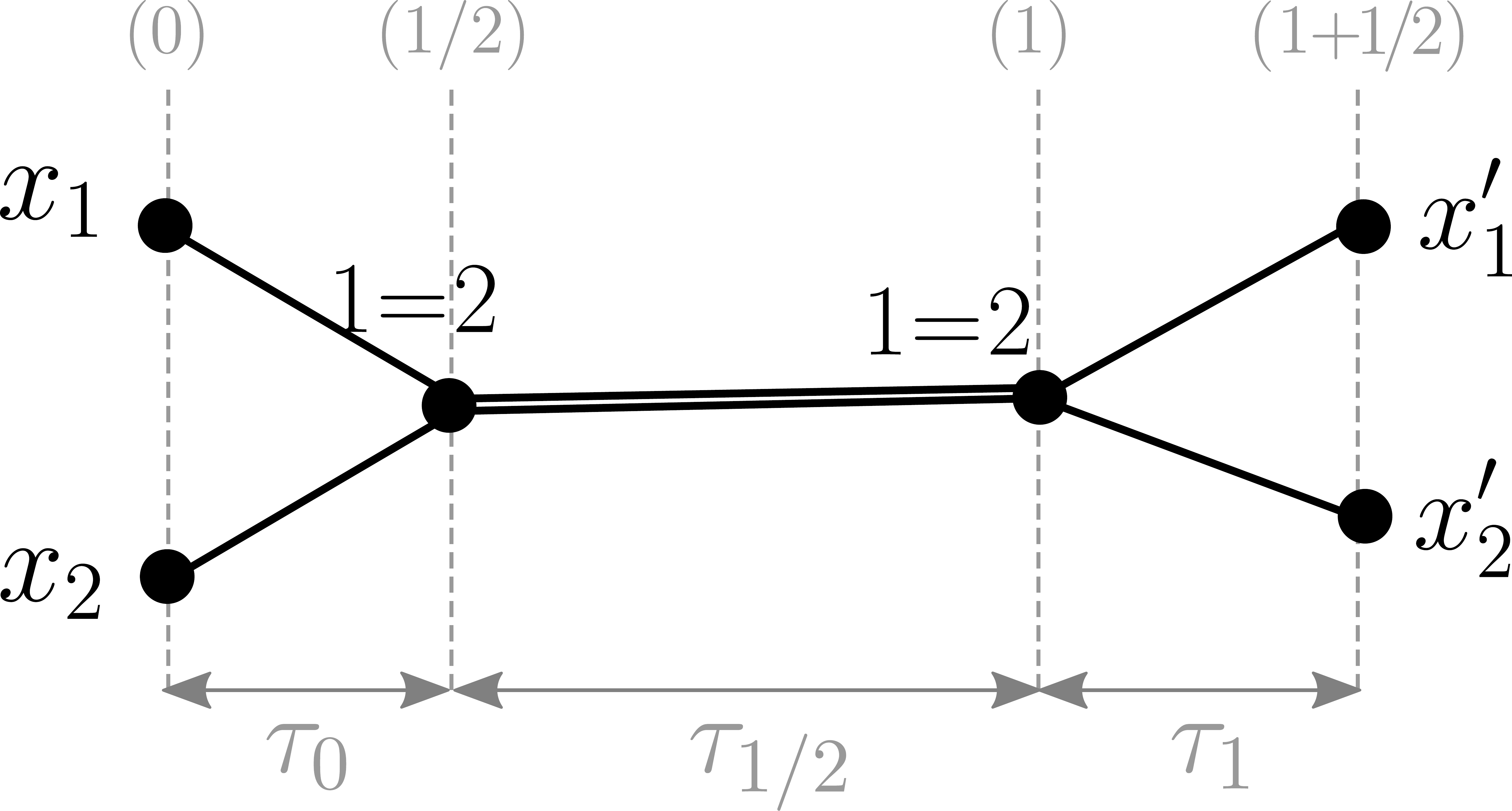}
\caption{The diagram of $ \dker^{((12))}(t,x,x') $.}
\label{f.n=2}
\end{figure}

In~\eqref{e.n=2.1}, rewrite the products in the center-of-mass and relative coordinates, 
\begin{align*}
	\prod_{\ell=1}^2 \hktwo\big(\tau,x_\ell) = \hktwo\big(\tfrac12\tau,\tfrac{x_1+x_2}{2}\big) \hktwo(2\tau,x_1-x_2) ,
\end{align*}
and then integrate over $ x^{(1/2)}_1,x^{(1)}_1\in\R^2 $,
using the semigroup property of $ \hktwo(\Cdot,\Cdot) $.
We then obtain
\begin{align}
\begin{split}
	\label{e.n=2}
	\big( 
		&\hk
		+
		\dker^{\diag(2)}
	\big)(t,x_1,x_2,x'_1,x'_2)
\\
	&=
	\hktwo\big(\tfrac12 t,x_\mathrm{c}-x'_\mathrm{c})
	\Big(
		\hktwo\big(2t,x_\mathrm{d}-x'_\mathrm{d}\big)
		+
		\int_{\tau_0+\tau_{1/2}+\tau_1=t} \d\vec{\tau} \,
		\hktwo\big(2\tau_0,x_\mathrm{d}\big)
		\,
		4\pi\jfn\big(\tau_{1/2},\betaC\big)
		\,
		\hktwo\big(2\tau_1,x'_\mathrm{d}\big)
	\Big),
\end{split}
\end{align} 
where $ x_\mathrm{c} := \frac{x_1+x_2}{2} $, $ x_\mathrm{d} := x_1-x_2 $, and similarly for $ x'_{\Cdot} $.

\begin{remark}\label{rmk.BCmatching}
The formula~\eqref{e.n=2} matches \cite[Equation~(3.11)--(3.12)]{bertini98} after a reparametrization.
Recall $ \betaC $ from~\eqref{e.betaC}.
Comparing our parameterization \eqref{e.betaeps} and with \cite[Equation~(2.6)]{bertini98},
we see that $ \betaC $ here corresponds to $ \log\beta $ in \cite{bertini98}.
The expression in~\eqref{e.n=2} matches \cite[Equation~(3.11)--(3.12)]{bertini98}
upon replacing $ (x_\mathrm{d},x'_\mathrm{d}) \mapsto(x,y) $, $ \betaC \mapsto \log \beta $,
and using the identity:
\begin{align}
	\label{e.khmatch}
	\int_0^\tau \hktwo(2(\tau-s),x_\mathrm{d}) \hktwo(2s,x'_\mathrm{d}) \, \d s
	=
	\frac{1}{8\pi^2\tau} \exp\big( -\tfrac{1}{4\tau}(|x_\mathrm{d}|^2+|x'_\mathrm{d}|^2) \big) K_0\big(\tfrac{|x_\mathrm{d}||x'_\mathrm{d}|}{2\tau}\big),
\end{align}
where $ K_\nu $ denotes the modified Bessel function of the second kind.

To prove~\eqref{e.khmatch}, by scaling in $ \tau $, without lost of generality we assume $ \tau=1 $.
On the l.h.s.\ of~\eqref{e.khmatch},
factor out $ \exp(-\frac{1}{4}(|x_\mathrm{d}|^2+|x'_\mathrm{d}|^2)) $,
decompose the resulting integral into $ s\in(0,1/2) $ and $ s\in(1/2,1) $,
for the former perform the change of variable $ u=(1-s)/s $, and for the latter $ u=s/(1-s) $.
We have
\begin{align*}
	(\text{l.h.s.\ of }\eqref{e.khmatch})
	=
	\exp\big(-\tfrac{1}{4}(|x_\mathrm{d}|^2+|x'_\mathrm{d}|^2)\big) I_\star,
	\qquad
	I_\star:= \ 2\int_1^\infty \frac{1}{(4\pi)^2 u} e^{-\frac{1}{4}(u|x_\mathrm{d}|^2+\frac{1}{u}|x'_\mathrm{d}|^2)} \, \d u.
\end{align*}
The integrand within the last integral stays unchanged upon the change of variable $ u\mapsto 1/u $,
while the range maps to $ (0,1) $.
We hence replace $ 2\int_1^\infty(\,\Cdot\,)\, \d u $ with $ \int_0^\infty(\,\Cdot\,)\,\d u $.
Within the result, perform a change of variable $ v = 2u|x_\mathrm{d}|^2 $,
and from the result recognize $ \frac{1}{2\pi v} e^{-\frac{1}{2v}(|x_\mathrm{d}|^2|x'_\mathrm{d}|^2)} = \hktwo(v,|x_\mathrm{d}||x'_\mathrm{d}|) $. 
We get
\begin{align*}
	I_\star
	=
	\int_0^\infty \frac{1}{(4\pi)^2 v} e^{-\frac{|x_\mathrm{d}|^2|x'_\mathrm{d}|^2}{2v}} e^{-\tfrac{v}8} \, \d v
	=
	\frac{1}{8\pi} \Gtwo_{-\frac18} \big( |x_\mathrm{d}||x'_\mathrm{d}| \big),
\end{align*}
where $ \Gtwo_z(|x|)=\Gtwo_z(x) :=( -\frac12  \nabla^2 - z\Io )^{-1}(0,x) $ denotes two-dimensional Green's function.
We will show in Lemma~\ref{l.Gtwo} that $ \Gtwo_z(x) = \frac{1}{\pi} K_0(\sqrt{-2z}|x|) $.
This gives~\eqref{e.khmatch}.
\end{remark}

\subsection{The $ n=3 $ case}
Here we derive a formula for the limiting centered third moment.
We say $ \Vec{(i,j)}=((i_k<j_k))_{k=1}^m \in \diag(n) $ is \textbf{degenerate} if $ \cup_{k=1}^m \{i_k,j_k\} \subsetneqq \{1,\ldots,n\} $,
and otherwise nondegenerate.
Let $ \diag'(n) $ denote the set of all nondegenerate elements of $ \diag(n) $, and, accordingly,
\begin{align*}
	\Dio^{\diag'(n)}_t
	:=
	\sum\nolimits_{\Vec{(i,j)}\in\diag'(n)} \Dio^{\Vec{(i,j)}}_t.
\end{align*}
\begin{proposition}\label{p.3rdmom}
Start the \ac{SHE} from $ Z_\e(0,\Cdot)=Z_\ic(\Cdot)\in\Lsp^2(\R^2) $.
For any $ f\in\Lsp^2(\R^2) $,
\begin{align}
	\label{e.3rdmom}
	\E\Big[ \Big( \ip{f}{Z_{\e,t}} - \E[\ip{f}{Z_{\e,t}}] \Big)^3 \Big] 
	\longrightarrow
	\Ip{ f^{\otimes 3} }{ \Dio^{\diag'(3)}_t Z_\ic^{\otimes 3} }
	\quad
	\text{as } \e\to 0,
\end{align}
uniformly in $ t $ over compact subsets of $ [0,\infty) $.
\end{proposition}
\begin{proof}
Expand the l.h.s.\ of~\eqref{e.3rdmom} into a sum of products of $ n'=1,2,3 $ moments of $ \ip{f}{Z_{\e,t}} $ as
\begin{align}
	\label{e.p.3rdmom}
	\E\Big[ \Big( \ip{f}{Z_{\e,t}} - \E[\ip{f}{Z_{\e,t}}] \Big)^3 \Big] 
	=
	\E\big[ \ip{f}{Z_{\e,t}}^3 \big] - 3 \E\big[ \ip{f}{Z_{\e,t}}^2 \big] \, \E\big[ \ip{f}{Z_{\e,t}} \big] + 2\big(\E\big[ \ip{f}{Z_{\e,t}} \big]\big)^3.
\end{align}
For the $ n'=1 $ moment, rewriting the \ac{SHE}~\eqref{e.she} in the mild (i.e., Duhamel) form and take expectation gives
\begin{align*}
	\E[ \ip{f}{Z_{\e,t}} ] = \ip{f}{\hktwo*Z_\ic} = \int_{\R^4} \bar{f(x')}\hktwo(t,x'-x)Z_\ic(x)\,\d x \d x',
\end{align*}
where $ * $ denotes convolution in $ x\in\R^2 $.
Note that for $ n'=2 $ the only index $ \diag(2)=\{((1<2))\} $ is the singleton and that
$
	\ip{f^{\otimes n'}}{ \Pt_t Z_\ic^{\otimes n'} } = \ip{f}{\hktwo*Z_\ic}^{n'}.
$
We then have
\begin{align}
	\label{e.p.3rdmom.}
	\lim_{\e\to 0} \E\big[ \ip{f}{Z_{\e,t}}^3 \big] &= \big(\ip{f}{\hktwo*Z_\ic}\big)^{3} + \Ip{ f^{\otimes 3} }{ \Dio^{\diag(3)}_t Z_\ic^{\otimes 3} },
\\
	\label{e.p.3rdmom..}
	\lim_{\e\to 0} \E\big[ \ip{f}{Z_{\e,t}}^2 \big] &= \big(\ip{f}{\hktwo*Z_\ic}\big)^{2} + \Ip{ f^{\otimes 2} }{ \Dio^{((12))}_t Z_\ic^{\otimes 2} }.
\end{align}
Inserting~\eqref{e.p.3rdmom.}--\eqref{e.p.3rdmom..} into~\eqref{e.p.3rdmom} gives
\begin{align}
	\label{e.p.3rdmom...}
	\lim_{\e\to 0}
	\E\Big[ \Big( \ip{f}{Z_{\e,t}} - \E[\ip{f}{Z_{\e,t}}] \Big)^3 \Big] 
	=
	\Ip{ f^{\otimes 3} }{ \Dio^{\diag(3)}_t Z_\ic^{\otimes 3} } - 3 \ip{f}{\hktwo*Z_\ic} \, \Ip{f^{\otimes 2}}{ \Dio^{((12))}_t Z_\ic^{\otimes 2} }.
\end{align}
For $ n'=3 $, degenerate indices in $ \diag(3) $ are the singletons $ ((1<2)), ((1<3)), ((2<3)) $.
This being the case, we see that the last term in~\eqref{e.p.3rdmom...} exactly cancels the contribution of degenerate indices in $ \Ip{f^{\otimes 3}}{ \Dio^{\diag(3)}_t Z_\ic^{\otimes 3} } $.
The desired result follows.
\end{proof}

\subsection{Proof of Corollary~\ref{c.main}}
Here we prove Corollary~\ref{c.main} assuming Theorem~\ref{t.main} (which will be proven in Section~\ref{sect.laplace}).
Our first goal is to show $ \mu_{\e,t}(\d x_1) := Z_{\e}(t,x_1) \d x_1 $, as a random measure on $ \R^2 $, is tight in $ \e $, under the vague topology.
This tightness has been established in \cite{bertini98}, and we repeat the argument here for the sake of being self-contained.
By~\cite[Lemma~14.15]{kallenberg97}, this amounts to showing $ \int_{\R^2} g(x) \mu_{\e,t}(\d x) = \ip{\bar{g}}{Z_{\eps,t}} $ is tight (as a $ \C $-valued random variable),
for each $ g\in \Csp_\cmp(\R^2) $.
Apply Theorem~\ref{t.main} with $ n=2 $, with $ Z_\ic(x_1) \mapsto |Z_\ic(x_1)| \in \Lsp^2(\R^2) $, and with $ f(x_1,x_2) = |g(x_1)g(x_2)| $.
We obtain that $ \E[|\ip{Z_{\e,t}}{g}|^2] $ is uniformly bounded in $ \e $, so $ \int_{\R^2} g(x) \mu_{\e,t}(\d x) $ is tight.

Fixing a limit point $ \mu_{*,t} $ of $ \{\mu_{\e,t}\}_\e $,
we proceed to show~\eqref{e.limit.point.mom}.
Fix a sequence $ \e_k \to 0 $ such that $ \mu_{\e_k,t,Z} \to \mu_{*,t} $ vaguely, as $ k\to\infty $.
The desired result~\eqref{e.limit.point.mom} follows from Theorem~\ref{t.main} if we can upgrade the preceding vague convergence of $ \mu_{\e_k,t,Z} $ to convergence in moments.
To this end we appeal to Theorem~\ref{t.main}. %
Note that $ |Z_\ic(\Cdot)| $ itself is in $ \Lsp^2(\R^2) $. %
Also, for fixed $ f_1,\ldots,f_n \in \Csp_\cmp(\R^2) $, %
the function $ f(x_1,\ldots,x_{2n}) := \prod_{i=1}^n |f_i(x_i)f_{i}(x_{n+i})| $ is in $ \Lsp^2(\R^{2n}) $. %
Applying Theorem~\ref{t.main} with $ n\mapsto 2n $, with $ Z_\ic(x_1) \mapsto |Z_\ic(x_1)| \in \Lsp^2(\R^2) $, and with $ f(x_1,\ldots,x_{2n}) = \prod_{i=1}^n |f_i(x_i)f_{i}(x_{n+i})| $, we obtain that 
\begin{align*}
	\E\Big[ \ip{f}{|Z_{\e,t}|^{\otimes 2n}} \Big]
	=
	\E\Big[\Big|\ip{f_1\otimes\cdots \otimes f_n}{Z_{\e,t}^{\otimes n}}\Big|^2\Big] 
	= 
	\E\Big[\Big| \prod_{i=1}^n \int_{\R^2} \bar{f_i(x_i)} \mu_{\e,t}(\d x_i) \Big|^2 \Big]
\end{align*}
is uniformly bounded in $ \e $.
Hence $ (\prod_{i=1}^n \int_{\R^2} \bar{f_i(x_i)} \mu_{\e,t}(\d x_i)) $ is uniformly integrable in $ \e $ (as $ \C $-valued random variables),
which guarantees the desired convergence in moments.

We now move on to showing \eqref{e.nonGaussian}.
For $ Z_\ic(x_1), f_1(x_1) \geq 0 $, both not identically zero, 
we apply Proposition~\ref{p.3rdmom} to obtain the $ \e\to 0 $ limit of the centered, third moment of $ \int_{\R^2}f_1(x_1)\mu_{\e,t,Z}(\d x_1) $.
As just argued, such a limit is also inherited by $ \mu_{*,t} $, whereby
\begin{align}
	\label{e.3rdmom.limit}
	\E\Big[ \Big( \int_{\R^2} f_1(x_1) \mu_{*,t}(\d x_1) - \E\Big[\int_{\R^2} f_1(x_1) \mu_{*,t}(\d x_1)\Big] \Big)^3 \Big] 
	=
	\Ip{ f_1^{\otimes 3} }{ \Dio^{\diag'(3)}_t Z_\ic^{\otimes 3} }.
\end{align}
As seem from~\eqref{e.diagram}, the operator $ \Dio^{\Vec{(i,j)}} $ has a \emph{strictly} positive integral kernel.
Under current assumption $ Z_\ic(x_1), f_1(x_1) \geq 0 $ and not identically zero,
we see that the r.h.s.\ of~\eqref{e.3rdmom.limit} is strictly positive. 

\section{Resolvent identity}
\label{sect.resolvent1}

In this section we derive the identity \eqref{e.resolvent.e} for the resolvent $ \Ro_{\e,z}=(\Ho_\e-z)^{-1} $ which is the key to our analysis.  

Let $ \Ho_\fr := - \frac12 \sum_{i} \nabla_i^2 $ denote the `free Hamiltonian', and let $\Do_\e: \Lsp^2(\R^{2n}) \to \Lsp^2(\R^{2n})$
\begin{align*}
	\Do_\e u(x) := \sum_{i<j} \delta_\e(x_i-x_j) u(x)
\end{align*}
denote the operator of multiplication by the approximate delta potential, which is a bounded operator for each $ \e>0 $.
The Hamiltonian $ \Ho_\e $ is then an unbounded operator on $ \Lsp^2(\R^{2n}) $ with domain $ \Hsp^2(\R^{2n}) $ (the Sobolev space), i.e.,
\begin{align}
	\label{e.hame}
	\Ho_\e := \Ho_\fr - \beta_\e \Do_\e,
	\quad\quad
	\Dom(\Ho_\e) := \Hsp^2(\R^{2n}) \subset \Lsp^2(\R^{2n}).
\end{align}
The first step is to built a `square root' of $ \Do_\e $.
More precisely, we seek to construct an operator $ \So_{\e ij} $, indexed by a pair $ i<j $, and its adjoint $ \So^*_{\e ij} $ such that 
$ \Do_\e = \sum_{i<j} \So^*_{\e ij} \phi\,\phi \So_{\e ij} $.
To this end,
for each $\eps>0$ and $1\leq i<j\leq n$, consider the linear transformation $ T_{\eps i j}: \R^{2n}\to \R^{2n} $:
\begin{align}
\label{e.T}
 	T_{\eps i j}(x_1,\ldots,x_n)
 	:=
 	(\tfrac{x_i-x_j}{\eps},\tfrac{x_i+x_j}{2},x_{\bar{ij}}),
\end{align}
where $x_{\bar{ij}}\in\R^{2(n-2)}$ denotes the vector obtained by removing the $i,j$-th components from $ x\in\R^{2n} $.
In other words, 
the transformation $T_{\eps ij}$ places the relative distance (on the scale of $\eps$) and the center of mass corresponding to $(x_i,x_j)$ in the first two components, 
while keeping all other components unchanged. 
The transformation $ T_{\eps i j} $ has inverse $ S_{\eps i j}=T_{\eps i j}^{-1}:\R^{2n}\to \R^{2n} $:
\begin{align}\label{e.Se}
	 S_{\eps i j} (y_1,\ldots,y_n)
	 :=
	 (y_3,\ldots,\underbrace{y_2+\tfrac{\eps y_1}{2}}_{i\text{-th}},\ldots, \underbrace{y_2-\tfrac{\eps y_1}{2}}_{j\text{-th}},\ldots, y_n).
\end{align}
Accordingly, we let $ \So_{\e ij} $ and $ \So_{\e ij}^* $ be the induced operators  $ \Lsp^2(\R^{2n})  \to \Lsp^2(\R^{2n})$,
\begin{align}
	\label{e.Soe}
	\big( \So_{\e ij} u \big)(y) := u(S_{\e ij}y),
	\quad
	\big( \So^*_{\e ij} v \big)(x) := \e^{-2} v(T_{\e ij}x).
\end{align}
It is straightforward to check that $ \So_{\e ij}^* $ is the adjoint of $ \So_{\e ij} $,
i.e., the unique operator for which  $\ip{\So^*_{\eps ij} v}{u}=\ip{v}{\So_{\e ij} u} $, $ \forall u,v\in\Lsp^2(\R^{2n}) $.
Since $ S_{\eps i j}, T_{\eps i j} $ are both invertible, 
the operators $ \So_{\eps ij}, \So^{*}_{\eps ij} $ are bounded for each $ \eps>0 $.
$ \Phi $ (defined in \eqref{e.hame.}) is even and non-negative, so
we can set
$
	\phi(x) := \sqrt{\Phi}(x)
$
and view $ (\phi v)(y) := \phi(y_1) v(y_1,\ldots,y_n) $ as a bounded multiplication operator on $ \Lsp^2(\R^{2n}) $.
From~\eqref{e.Soe}, it is straightforward to check
\begin{align}
	\label{e.DoSS}
	\Do_\e = \sum_{i<j} \So^*_{\e ij}\phi\,\phi\So_{\e ij}.
\end{align}


\begin{remark}
We comment on how our setup compares to that of~\cite{dimock04}.
They work in  $ \Lsp^2_\sym(\R^{2n}) $, corresponding to $ n $ Bosons in $ \R^2 $,
the key idea being to decompose the action of the delta potential $ \Do_\e $ on $ \Lsp^2_\sym(\R^{2n}) $
into some intermediate actions from $ \Lsp^2_\sym(\R^{2n}) $ into an `auxiliary space',
consisting of $ n-2 $ Bosons and an `angle particle'.
In our current setting, the auxiliary space is $ \Lsp^2(\R^{2n}) \ni v=v(y_1,y_2,y_3,\ldots,y_n) $.
The components $ y_3,\ldots,y_n $ correspond to the  $ n-2 $ particles,
the component $ y_2 $ corresponds to the angle particle,
while $ y_1 $ is a `residual' component that arises from our space-mollification scheme,
and is not presented under the momentum-cutoff scheme of \cite{dimock04}.
\end{remark}

Given \eqref{e.DoSS},
the next step is to develop an expression for the resolvent $ \Ro_{\e,z}=(\Ho_\e-z)^{-1} $
that is amenable for the $ \e\to 0 $ asymptotic.
In the case of momentum cutoff,
such a resolvent expression is obtained in (Eq (68) of) \cite{dimock04}
by comparing two different ways of inverting a two-by-two (operator-valued) matrix.
Here, we derive the analogous expression (i.e., \eqref{e.resolvent.e})
using a more straightforward procedure --- power-series expansion of (operator-valued) geometric series.
Recall $ \diag(n,m) $ from~\eqref{e.diag.set.m},
recall that $ \normo{\geno} $ denotes the operator norm of $ \geno $,
and recall from~\eqref{e.Green} that $ \Go_z $ denotes the resolvent of the Laplacian.

\begin{lemma}\label{l.resolvent.e}
For all $ \eps\in(0,1) $ and 
$ z\in\C $ such that $ \Re(z)<-\beta_\e (1+\sum_{i<j} \normo{\So_{\e ij}\phi})^2 $, we have
\begin{subequations}
\label{e.resolvent.e}
\begin{align}
	\label{e.resolvent.e.out}
	\Ro_{\e,z}
	:=&
	(\Ho_\e-z\Io)^{-1}
	=
	\Go_{z} 
	+ 
	\sum_{m=1}^\infty \ \sum_{\Vec{(i,j)}\in\diag(n,m)}
	\big( \Go_z \So^{*}_{\e i_1j_1} \phi \big)
\\
	\label{e.resolvent.e.med}
	&\cdot
	\big( \beta_\e^{-1}\Io - \phi\So_{\e 12}\Go_z\So_{\e 12}^{*}\phi \big)^{-1}
	\prod_{k=2}^{m} \Big( \big(\phi\So_{\e i_{k-1}j_{k-1}}\Go_z\So^*_{\e i_kj_k}\phi \big) \, \big( \beta_\e^{-1}\Io - \phi\So_{\e 12}\Go_z\So_{\e 12}^{*}\phi \big)^{-1}  \Big)
\\
	\label{e.resolvent.e.in}
	&\cdot
	\big( \phi\So_{i_{m}j_{m}}\Go_z \big).
\end{align}
\end{subequations}
\end{lemma}%
\begin{remark}
As stated, Lemma~\ref{l.resolvent.e} holds for $ \Re(z) < -C_1(\e,n) $, with a threshold $ C_1(\e,n) $ that depends on $ \e $.
This may not seem  useful as $ \e\to0 $, 
however, as we will show later in Section~\ref{sect.resolcnvg}, 
the r.h.s.\ of \eqref{e.resolvent.e} is actually analytic (in norm) in $ \{z: \Re(z)<-C_2(n)\} $, 
for some threshold $ C_2(n)<\infty $ that is \emph{independent} of $ \eps $. 
It then follows immediately (as argued in Section~\ref{sect.resolcnvg}) that~\eqref{e.resolvent.e} extends to all $ \Re(z)< -C_2(n) $.
\end{remark}
\begin{proof}
To simplify notation, set $ \tilde{\So}_{ij} := \beta_\e^{1/2}\phi \So_{\e ij} $, $ \tilde{\So}^{ij} := (\tilde{\So}_{ij})^* = \beta^{1/2}_\e\So^*_{\e ij}\phi $,
and $ \tilde{\Soo}_{ij}^{k\ell} := \tilde{\So}_{ij}\Go_z\tilde{\So}^{k\ell} $.
In~\eqref{e.resolvent.e.med}, factor $ \beta_\e^{-1} $ from the inverse.
Under the preceding shorthand notation, we rewrite~\eqref{e.resolvent.e} as
\begin{align}
\label{e.resolvent.e.}
	\Ro_{\e,z}
	=
	\Go_{z} 
	+ 
	\sum_{m=1}^\infty \ \sum_{\Vec{(i,j)}\in\diag(n,m)}
	\Go_z \tilde{\So}^{i_1j_1}
	\cdot
	\big( \Io - \tilde{\Soo}_{12}^{12} \big)^{-1}
	\prod_{k=2}^{m}  \tilde{\Soo}_{i_{k-1}j_{k-1}}^{i_kj_k} \, \big( \Io - \tilde{\Soo}_{12}^{12} \big)^{-1}   
	\cdot
	\tilde{\So}_{i_{m}j_{m}} \Go_z .
\end{align}
Our goal is to expand the inverse in~\eqref{e.resolvent.e.},
and then simplify the result to match $ (\Ho_\e-z\Io)^{-1} $.

To expand the inverse in~\eqref{e.resolvent.e.}, we  utilize the geometric series
$
	(\Io - \geno)^{-1} = 
	\Io + \sum_{k=1}^\infty \geno^k$,
valid for $ \normo{\geno}<1 $.
Indeed, $ \normo{\Go_z} \leq 1/(-\Re(z)) $, so under the assumption on the range of $ \Re(z) $ we have
$ \normo{\tilde{\So}_{12}^{12}} <1 $.
Using the geometric series for $ \geno=\tilde{\Soo}_{12}^{12} $,
and inserting the result into~\eqref{e.resolvent.e.} gives
\begin{align}
\label{e.resolvent.e..}
	\Ro_{\e,z}
	=&
	\Go_{z} 
	+ 
	\sum
\Go_z \tilde{\So}^{i_1j_1}
	\underbrace{ \tilde{\Soo}_{12}^{12} \cdots \tilde{\Soo}_{12}^{12} }_{\ell_1}
	\tilde{\Soo}_{i_{1}j_{1}}^{i_2j_2} 
	\underbrace{ \tilde{\Soo}_{12}^{12} \cdots \tilde{\Soo}_{12}^{12} }_{\ell_2}
	\tilde{\Soo}_{i_{2}j_{2}}^{i_3j_3} 
	\cdots
	\tilde{\Soo}_{i_{m-1}j_{m-1}}^{i_mj_m} 
	\underbrace{ \tilde{\Soo}_{12}^{12} \cdots \tilde{\Soo}_{12}^{12} }_{\ell_m}
	\tilde{\So}_{i_{m}j_{m}} \Go_z.
\end{align}
where the sum is over $\ell_1,\ldots,\ell_m \geq 0$, $\Vec{(i,j)}\in\diag(n,m)$,
and $m=1,2,\ldots$.
The sum converges absolutely in operator norm by our assumption on $ z $.
Since $ \Go_z $ acts symmetrically in the $ n $ components, we have $ \tilde{\Soo}_{12}^{12} =\tilde{\Soo}_{ij}^{ij} $, for any pair $ i<j $.
Use this property to rewrite~\eqref{e.resolvent.e..} as
\begin{align}
\label{e.resolvent.e...}
\begin{split}
	\Ro_{\e,z}
	=&
	\Go_{z} 
	+ 
	\sum
\Go_z \tilde{\So}^{i_1j_1}
	\underbrace{ \tilde{\Soo}_{i_1j_1}^{i_1j_1} \cdots \tilde{\Soo}_{i_1j_1}^{i_1j_1} }_{\ell_1}
	\tilde{\Soo}_{i_{1}j_{1}}^{i_2j_2} 
	\underbrace{ \tilde{\Soo}_{i_2j_2}^{i_2j_2} \cdots \tilde{\Soo}_{i_2j_2}^{i_2j_2} }_{\ell_2}
	\tilde{\Soo}_{i_{2}j_{2}}^{i_3j_3} 
	\cdots
	\tilde{\Soo}_{i_{m-1}j_{m-1}}^{i_mj_m} 
	\underbrace{ \tilde{\Soo}_{i_mj_m}^{i_mj_m} \cdots \tilde{\Soo}_{i_mj_m}^{i_mj_m} }_{\ell_m}
	\tilde{\So}_{i_{m}j_{m}} \Go_z.
\end{split}
\end{align}
The summation can be reorganized as
$\sum_{m'=1}^\infty \sum_{i_1<j_1} \cdots \sum_{i_{m'}<j_{m'}}(\ \Cdot \ ) $.
To see this, recall from~\eqref{e.diag.set.m} that $ \Vec{(i,j)}\in\diag(n,m) $
consists of pairs $ (i_k<j_k) $ under the constraint that consecutive pairs are non-repeating, i.e., $ (i_{k-1}<j_{k-1})\neq (i_{k}<j_{k}) $.
The r.h.s.\ of~\eqref{e.resolvent.e...} replenishes all possible repeatings of consecutive pairs,
and hence lifts the constraints imposed by $ \diag(n,m) $.
In the resulting sum, express $ \tilde{\Soo}_{k\ell}^{ij} = \tilde{\So}^{ij} \Go_z \tilde{\So}_{k\ell} $ to get
\begin{align*}
	\Ro_{\e,z}
	=
	\sum_{m=0}^\infty 
	\Go_z 
	\Big( \sum_{i<j} \tilde{\So}^{ij} \tilde{\So}_{ij} \Go_z  \Big)^m.
\end{align*}
From~\eqref{e.DoSS}, we have $ \sum_{i<j}\tilde{\So}^{ij} \tilde{\So}_{ij} = \beta_\e \Do_\e $,
hence
$ 	
	\Ro_{\e,z}
	=
	\Go_z(\Io-\beta_\e \Do_\e\Go_z )^{-1}.
$
Further $ \Go_z = (\Ho_\fr-z\Io)^{-1} $ gives
\begin{align*}
	\Ro_{\e,z}
	=
	(\Ho_\fr-z\Io)^{-1}\big(\Io-\beta_\e \Do_\e (\Ho_\fr-z\Io)^{-1} \big)^{-1}
	=
	\big(\Ho_\fr-z\Io-\beta_\e \Do_\e \big)^{-1}
	=
	\big(\Ho_\e - z \Io \big)^{-1}.
\end{align*}
This completes the proof.
\end{proof}

The resolvent identity~\eqref{e.resolvent.e} is the gateway to the $ \e\to 0 $ limit.
Roughly speaking, we will show that all terms in~\eqref{e.resolvent.e} converge to their limiting counterparts
in the expression of $ \Ro_z $ given in~\eqref{e.resolvent}.
The expression~\eqref{e.resolvent}, however, does not expose such a convergence very well.
This is so because some operators in~\eqref{e.resolvent} map one function space to a different one, 
(e.g., $ \So_{ij} $ maps functions of $ n $ components to $ n-1 $ components),
while \emph{all} operators in the sum over $ m $ in~\eqref{e.resolvent.e}  map $ \Lsp^2(\R^{2n}) $ to $ \Lsp^2(\R^{2n}) $.
We next rewrite~\eqref{e.resolvent} in a way that better compares with~\eqref{e.resolvent.e}.
To this end, consider the operators
\begin{align}
	\label{e.Po}
	&\Po: \Lsp^{2}(\R^{2n}) \to \Lsp^{2}(\R^{2n-2}),
	\quad
	(\Po v)(y_{2-n}) :=\int_{\R^{2}} \phi(y_1) v(y_1,y_{2-n}) \d y_1
\\
	&
	\phi\otimes\Cdot : \Lsp^{2}(\R^{2n-2}) \to \Lsp^{2}(\R^{2n}),
	\quad
	(\phi \otimes v)(y_1,y_{2-n}) := \phi(y_1) v(y_{2-n}).
\end{align}
Given that $ \phi\in\Csp^\infty_\cmp(\R^{2}) $, it is readily checked that $ \Po $ and $ \phi\otimes\Cdot $ are bounded operators.
Note that from $ \phi:=\sqrt{\Phi} $, $ \phi $ has unit norm, i.e., $ \int_{\R^2} \phi^2 \d y = 1 $.
From this we obtain $ \Po(\phi\otimes\geno) = \geno $, for a generic $ \geno: \Lsp^2(\R^{2n})\to\Lsp^2(\R^{2n-2}) $
or $ \geno: \Lsp^2(\R^{2n-2})\to\Lsp^2(\R^{2n-2}) $.
Using this property, we rewrite~\eqref{e.resolvent} as
\begin{subequations}
\label{e.resolvent.}
\begin{align}
	\label{e.resolvent.out}
	\Ro_{z}
	=& ~\Go_{z} +
	\sum_{m=1}^\infty \ \sum_{\Vec{(i,j)}\in\diag(n,m)}
	\big(\Go_z \So^{*}_{i_1j_1}\Po \big)
\\
	\label{e.resolvent.med}
	&
	\cdot 
	\big( \phi\otimes  4\pi(\Jo_{z}-\betaC\Io)^{-1} \Po \big) \, 
	\prod_{s=2}^{m} \Big( \big(\phi\otimes\So_{i_{s-1}j_{s-1}}\Go_z\So^*_{i_sj_s}\Po\big) \, \big( \phi\otimes 4\pi(\Jo_{z}-\betaC\Io)^{-1} \Po \big) \Big) 
\\
	\label{e.resolvent.in}
	&
	\cdot
	\big( \phi\otimes\So_{ij}\Go_z \big).
\end{align}
\end{subequations}
That is, we augment the missing $ y_1 $ dependence (in the operators $ \So_{ij} $, $ \So^*_{ij} $, etc.) 
along the subspace $ \C\phi\subset\Lsp^2(\R^2) $.
Equation~\eqref{e.resolvent.} gives a better expression for comparison with \eqref{e.resolvent.e}.%

For future references, let us setup some terminology for the operators in \eqref{e.resolvent.e} and~\eqref{e.resolvent.}.
We call the operators $ \So_{ij}\Go_z $ or $ \phi\otimes\So_{ij}\Go_z $ in~\eqref{e.resolvent.in}
the \textbf{limiting incoming operators},
and the operators $ \Go_z\So^*_{ij} $ or $ \Go_z\So^*_{ij}\Po $ in~\eqref{e.resolvent.out} the \textbf{limiting outgoing  operators}.
Slightly abusing language, we will use these phrases interchangeably to infer operators \emph{with and without} the action by $ \phi\otimes\Cdot $ or $ \Po $.
Similarly, we call the operators in~\eqref{e.resolvent.e.in} the \textbf{pre-limiting incoming  operators},
and the operators in~\eqref{e.resolvent.e.out} the \textbf{pre-limiting outgoing  operators}.
Next, with $ \Jo_z $ defined in~\eqref{e.Jo} in the following,
we refer to $ (\Jo_z - \betaC\Io) $ and $ (\beta_\e^{-1}\Io - \So_{\e 12}\Go_z\So^*_{\e 12}) $ 
as the \textbf{limiting} and \textbf{pre-limiting} \textbf{diagonal mediating operators}, respectively,
and refer to $ \So_{ij} \Go_z \So^*_{k \ell} $ and $ \So_{\e ij} \Go_z \So^*_{\e k \ell} $, with $ (i<j)\neq(k<\ell) $, as the the \textbf{limiting} and \textbf{pre-limiting} \textbf{off-diagonal mediating operators}.

As we will show in Section~\ref{sect.in.out}, 
each pre-limiting incoming and outgoing operator converges to its limiting counterpart, and,
as will show in Section~\ref{sect.offdiagonal}, 
each off-diagonal mediating operator converges to its limiting counterpart.
Diagonal mediating operators require a more delicate treatment because $ \beta_\e^{-1}\Io $ and $ \So_{\e ij}\Go_z\So^*_{\e ij} $ both diverge on their own,
and we need to cancel the divergence (and also to take an inverse) to obtain a limit.
This procedure, sometimes referred to as \emph{renormalization} in the physics literature, will be carried out in Section~\ref{sect.diagonal}.


\section{Incoming and outgoing operators}
\label{sect.in.out}
In this section we obtain the $ \e\to 0 $ limit of $ \phi\So_{\e ij}\Go_z $ and $ \Go_z\So^*_{\e ij}\phi $
to $\phi\otimes (\So_{ij}\Go_{z})$ and $\Go_{z}\So^*_{ij}\Po $.  The main result is
	stated in Lemma \ref{l.BGconvg}.
	
Recall the linear transformation $ S_{ij} $ and its induced operator $ \So_{ij} $ from~\eqref{e.S}--\eqref{e.So}.
Comparing \eqref{e.Se} and \eqref{e.S},
we see that $ S_{\e ij}(y_1,\ldots,y_n) \to S_{ij}(y_2,\ldots,y_n) $ as $ \e\to 0 $.
Namely, $ S_{ij} $ is the pointwise limit of $ S_{\e ij} $. 
This observation hints that $ \So_{ij} $ should be the limit of $ \So_{\e ij} $,
and the $ \e\to 0 $ limit of the incoming operator $ \phi\So_{\e ij}\Go_z $ 
should be obtained by replacing $ \So_{ij} $ with $ \So_{\e ij} $.
Note that, however, the operator $ \So_{ij} $ is \emph{unbounded},
 because, unlike $ S_{\e ij} $, $ S_{ij} $, maps between spaces of \emph{different} dimensions;
 the $ y_1 $ dependence in $ S_{\e ij}(y_1,\ldots,y_n) $ `vanishes' as $ \e\to 0 $ (c.f., \eqref{e.Se}). 

As the first step of building the limiting operators, we construct the domain of $ \So_{ij} $, along with its adjoint $ \So^{*}_{ij} $.
In the following we will often work in the Fourier domain.
Let $ \hat{f}(q) := \int_{\R^{d}} e^{-\img y\cdot q} f(y) \frac{\d q}{(2\pi)^{d/2}} $ 
denote Fourier transform of functions on $ \R^{d} $;
 the inverse Fourier transform then reads 
$ f(y) = \int_{\R^{d}} e^{\img y\cdot q} \hat{f}(q) \frac{\d q}{(2\pi)^{d/2}} $. 
Let $ \Ssp(\R^{d}) $ denote the space of Schwartz functions, namely the space of $ \Csp^{\infty} $ functions on $ \R^{d} $ with derivatives decaying at super-polynomial rates, c.f., \cite[Definition~7.3]{rudin91}.
In our subsequential analysis, $ d $ is typically $ 2n $ or $ 2(n-1) $.
Consider the (invertible) linear transformation $ \R^{2n}\to\R^{2n} $:
\begin{align}
	\label{e.M}
	\M_{ij}q := (q_3,\ldots,\underbrace{\tfrac12 q_2+q_1}_{i\text{-th}},\ldots,\underbrace{\tfrac12 q_2-q_1}_{j\text{-th}},\ldots,q_n).
\end{align}
For $ q\in\R^{2n} $, we write $ q_{i-j}:=(q_i,\ldots,q_j)\in\R^{2(j-i+1)}$,
and recall that $ q_{\bar{ij}}\in\R^{2n-4} $ is obtained from removing the $ i $-th and $ j $-th components of $ q $.
\begin{lemma}\label{l.So}
\begin{enumerate}[label=(\alph*), leftmargin=20pt]
\item []
\item \label{l.So.}
The operator $\So_{ij}$, given by equation~\eqref{e.So}, is unbounded from $ \Lsp^2(\R^{2n})$ to $\Lsp^2(\R^{2n-2}) $, with
\begin{align}
	\label{e.domSo}
	\Dom(\So_{ij}) := \Big\{ f\in\Lsp^2(\R^{2n}) \, : \, \int_{\R^2} \big|\hat{f}(\M_{ij}(q_1,\Cdot)) \big| \d q_1 \in \Lsp^2(\R^{2n-2}) \Big\}
	\subset
	\Lsp^2(\R^{2n}),
\end{align}
and for $ f\in \Dom(\So_{ij}) $, we have
\begin{align}
	\label{e.S.Four}
	\hat{\So_{ij} f}(q_{2-n})
	= 
	\int_{\R^2} \hat{f}(\M_{ij}q) \frac{d q_1}{2\pi}.
\end{align}
In addition, for all $ a>1 $, we have $ \Hsp^{a}(\R^{2n})\subset\Dom(\So_{ij}) $.
\item \label{l.So*}
The operator
\begin{align}
	\label{e.So*.Four}
	\hat{ \So^*_{ij} g}(p) := \tfrac{1}{2\pi} \hat{g}(p_i+p_j,p_{\bar{ij}})
\end{align}
maps $  \Lsp^2(\R^{2n-2}) \to \cap_{a>1} \Hsp^{-a}(\R^{2n}) $,
and is adjoint to $ \So_{ij} $ in the sense that
\begin{align}
	\label{e.So*}
	\ip{ \So^*_{ij} g }{ f } = \ip{g}{\So_{ij}f},
	\quad
	g \in \Lsp^2(\R^{2n-2}),
	\
	f \in \Hsp^a(\R^{2n}),
	\
	a>1.
\end{align}
\end{enumerate}
\end{lemma}
\begin{proof}
\ref{l.So.}~Let us first show~\eqref{e.S.Four} for $ f\in \Ssp(\R^{2n}) $. 
On the Fourier transform of $ f $, perform the change of variables $ x=S_{1ij}y $, where $ S_{1ij}= S_{\e ij}|_{\e=1} $,
and the substitute $ p=\M_{ij}q $.
From~\eqref{e.Se}, it is readily checked that $ |\det(S_{1ij})| =1 $,
and from \eqref{e.M}, we have $ (S_{1ij}y)\cdot(\M_{ij}q) = y\cdot q $, so
\begin{align}
	\label{e.SoFour2}
	\hat{f}(M_{ij}q) 
	=  
	\int_{\R^{2n}} e^{-\img y\cdot q} f(S_{1ij}y) \frac{\d y}{(2\pi)^{n}}.
\end{align}
Our goal is to calculate the Fourier transform of $ f(S_{ij}\Cdot) $.
Comparing~\eqref{e.S} and \eqref{e.Se} for $ \e=1 $, we see that $ (S_{1ij}y)|_{y_1=0}=S_{ij}(y_{2-n}) $.
It is hence desirable to `remove' the $ y_1 $ variable on the r.h.s.\ of~\eqref{e.SoFour2}.
To this end, apply the identity
\begin{align*}
	\int_{\R^{2n-2}} g(0,y_{2-n}) e^{-\img q_{2-n}\cdot y_{2-n}} \frac{\d y_{2-n}}{(2\pi)^{n-1}}
	=
	\int_{\R^{2}} \hat{g}(q) \frac{\d q_1}{2\pi},
	\quad
	g\in\Ssp(\R^{2n})
\end{align*}
with $ g(\Cdot) = f(S_{1ij}\Cdot) $ to obtain
\begin{align*}
	\int_{\R^2} \hat{f}(M_{ij}q) \frac{\d q_1}{2\pi}
	&=  
	\int_{\R^{2n-2}} e^{-\img y_{2-n}\cdot q_{2-n}} \, f(S_{1ij}y)\big|_{y_1=0} \frac{\d y_{2-n}}{(2\pi)^{n-1}}
=  
	\int_{\R^{2n-2}} e^{-\img y_{2-n}\cdot q_{2-n}} \, f(S_{ij}y_{2-n}) \frac{\d y_{2-n}}{(2\pi)^{n-1}}.
\end{align*}
The last expression is $ \hat{\So_{ij}f}(q_{2-n}) $ by definition.
We hence conclude~\eqref{e.S.Four} for $ f\in\Ssp(\R^{2n}) $.
By approximation, it follows that $ \So_{ij} $ extends to an unbounded operator with domain~\eqref{e.domSo},
and the identity~\eqref{e.S.Four} extends to $ f\in\Dom(\So_{ij}) $.

Fix $ a>1 $, we proceed to show $ \Hsp^a(\R^{2n})\subset\Dom(\So_{ij}) $.
For $ f\in\Hsp^a(\R^{2n}) $, it suffices to bound
\begin{align}
	\label{e.domSo.check}
	\int_{\R^{2n-2}} \Big| \int_{\R^2} \big|\hat{f}(\M_{ij}q) \big| \d q_1 \Big|^2 \d q_{2-n}.
\end{align}
Within the integrals, multiply and divide by $ (\frac12|\M_{ij}q|^2+1)^{\frac{a}{2}} $.
Use $ \frac12|\M_{ij}q|^{2} \geq |q_1|^2 $ (as readily checked from~\eqref{e.M})
and apply the Cauchy--Schwarz inequality over the integral in $q_1$.
We then obtain
\begin{align}
\label{e.domSo.check.}
\begin{split}
	\eqref{e.domSo.check}
	&=
	\int_{\R^{2n-2}} \Big| \int_{\R^2} \frac{1}{(\frac12|\M_{ij}q|^2+1)^{\frac{a}{2}}}(\tfrac12|\M_{ij}q|^2+1)^{\frac{a}{2}}\big|\hat{f}(\M_{ij}q) \big| \d q_1 \Big|^2 \d q_{2-n}
\\
	&\leq
	\int_{\R^2} \Big( \frac{1}{|q_1|^2+1} \Big)^{a} \, \d q_1
	\,
	\norm{f}_{\Hsp^{a}(\R^{2n})}
	\leq 
	\frac{C}{a-1} \norm{f}_{\Hsp^{a}(\R^{2n})}.
\end{split}
\end{align}
This verifies $ \Hsp^a(\R^{2n})\subset\Dom(\So_{ij}) $.

\medskip

\noindent\ref{l.So*}~
That $ \So^*_{ij} $ maps $ \Lsp^2(\R^{2n-2}) $ to $ \cap_{a>1}\Hsp^{-a}(\R^{2n}) $
is checked by similar calculations as in~\eqref{e.domSo.check.}.
To check~\eqref{e.So*}, calculate the inner product $ \ip{ \So^*_{ij} g }{ f } $ in Fourier variables from~\eqref{e.So*.Four}.
Within the resulting integral, perform a change of variable $ p=\M_{ij}q $,
and use $ |\det(\M_{ij})|=1 $ and $ (p_i+p_{j},p_{\bar{ij}})=(\M_{ij}^{-1}p)_{2-n} $ (as readily checked from~\eqref{e.M}).
In the last expression $ (\M_{ij}^{-1}p)_{2-n} $ denotes the last $ (n-1) $ components of the vector $ \M_{ij}^{-1}p\in(\R^{2})^n $.
We then obtain
\begin{align*}
	\ip{ \So^*_{ij} g }{ f }
	=
	\int_{\R^{2n}} \bar{\hat{g}(p_i+p_j,p_{\bar{ij}})} \hat{f}(p) \frac{\d p}{2\pi}
	=
	\int_{\R^{2n}} \bar{\hat{g}(q_{2-n})} \hat{f}(\M_{ij}q) \frac{\d q}{2\pi}.
\end{align*}
From~\eqref{e.S.Four}, we see that the last expression matches $ \ip{ g }{ \So_{ij} f } $.
\end{proof}

Recall that, for each $ \Re(z)<0 $, $ \Go_z(\Lsp^2(\R^{2n})) = \Hsp^2(\R^{2n}) $.
This together with Lemma~\ref{l.So} implies that $ \So_{ij}\Go_z $ is defined on the entire $ \Lsp^{2}(\R^{2n}) $, with image in $ \Lsp^{2}(\R^{2n-2}) $, 
and that $ \Go_z\So^*_{ij} $ is defined on $ \Lsp^{2}(\R^{2n-2}) $, with image in $ \Lsp^{2}(\R^{2n}) $.
Informally,  $ \Go_z $ increases regularity by $ 2 $,
while  $ \So_{ij} $ and $ \So^*_{ij} $ both decrease regularity by $ -(1^+) $, as seen from Lemma~\ref{l.So}.
In total $ \So_{ij}\Go_z $ and $ \Go_z\So^*_{ij} $ have regularity exponent $ 2-(1^+)=1^->0 $.

We now establish a quantitative bound on the operator norm of the limiting operators $ \So_{ij}\Go_z $ and $ \Go_z\So^*_{ij} $.
\begin{lemma}\label{l.Goin}
For $ 1\leq i<j \leq n $ and $ \Re(z)<0 $,
$
	\normo{\So_{ij}\Go_z}=\normo{\Go_{\bar{z}}\So^*_{ij}} \leq C\,(\Re(-z))^{-1/2}.
$
\end{lemma}
\begin{proof}
That $ \normo{\So_{ij}\Go_z}=\normo{\Go_{\bar{z}}\So^*_{ij}} $ follows by~\eqref{e.So*}, so it is enough to bound $ \normo{\So_{ij}\Go_z} $.
Fix $ u\in\Lsp^2(\R^{2n}) $ and apply~\eqref{e.S.Four} for $ f=\Go_zu $ to get
\begin{align}
	\label{e.SG.fourier}
	\hat{ \So_{ij}\Go_z u }(q_{2-n}) = \int_{\R^2} \frac{\hat{u}(\M_{ij}q)}{\frac12|\M_{ij}q|^2-z} \frac{\d q_1}{2\pi}.
\end{align}
Calculate the norm of $ \So_{ij}\Go_z u $ from~\eqref{e.SG.fourier} gives
\begin{align*}
	\| \So_{ij}\Go_z u \|^2 
	= 
	\int_{\R^{2n-2}} \Big| \int_{\R^2} \frac{\hat{u}(\M_{ij}q)}{\frac12|\M_{ij}q|^2-z} \frac{\d q_1}{2\pi} \Big|^2 \d q_{2-n}.
\end{align*}
Apply the Cauchy--Schwarz inequality over the $ q_1 $ integration,
and within the result use $ \frac12|\M_{ij}q|^{2} \geq |q_1|^2 $ (as readily checked from~\eqref{e.M}) and $ \Re(z)<0 $.
We get 
\begin{align*}
	\| \So_{ij}\Go_z u \|^2 
	\leq
	\Big( \int_{\R^2} \frac{1}{(|q_1|^2+\Re(-z))^2} \frac{\d q_1}{(2\pi)^2} \Big)\, \norm{u}^2.
\end{align*}
The last integral over $ q_1 $ can be evaluated in polar coordinate to be $ \frac1{4\pi}\Re(-z) $ .
This completes the proof.
\end{proof}

Having built the limiting operator, our next step is to show the convergence.
In the course of doing so, we will often use a partial Fourier transform in the last $ n-1 $ components:
\begin{align}\label{pft}
	\pFou{f}(y_1,q_{2-n}) 
	:= 
	\int_{\R^{2n-2}} e^{-\img (y_2,\ldots,y_d)\cdot (q_2,\ldots,q_n)} f(y_1,\ldots,y_n) \prod_{i=2}^n \frac{\d y_i}{2\pi}.
\end{align}
Recall $ \So_{\e ij} $ from~\eqref{e.Soe}.
To prepare for the proof of the convergence, we establish an expression of $ \So_{\e ij}u $ in partial Fourier variables.

\begin{lemma}\label{l.Soe.Four}
For every $ 1\leq i < j \leq n $ and $ u\in\Ssp(\R^{2n}) $, we have
\begin{align}
	\label{e.Se.Four}
	\pFou{\So_{\e ij} u}(y_1,q_{2-n})
	=
	\int_{\R^2} e^{\img \e q_1\cdot y_1}\hat{u}(\M_{ij}q) \frac{d q_1}{2\pi}.
\end{align}
\end{lemma}

\begin{proof}
A partial Fourier transform can be obtained by inverting a full transform in the first component:
\begin{align}
	\label{e.pinverse}
	\pFou{\So_{\e ij} f}(y_1,q_{2-n}) = \int_{\R^{2}} \hat{\So_{\e ij} f}(q) e^{\img y_1\cdot q_1} \frac{\d q_1}{2\pi}. 
\end{align}
We write the full Fourier transform as $ \hat{\So_{\e ij} f}(q) = \int_{\R^{2n}} e^{-\img y\cdot q} f(S_{\e ij}y) \frac{\d y}{(2\pi)^n} $.
We wish to perform a change of variable $ x=S_{\e ij}y $.
Doing so requires understanding how $ (y\cdot q) $ transform accordingly.
Defining
\begin{align*}
	\M_{\e ij}q := (q_3,\ldots,\underbrace{\tfrac12 q_2+\e^{-1}q_1}_{i\text{-th}},\ldots,\underbrace{\tfrac12 q_2-\e^{-1}q_1}_{j\text{-th}},\ldots,q_n),
\end{align*}
it is readily checked that $ y\cdot q = (\M_{\e ij} q)\cdot(S_{\e ij}y) $.
Given this, we perform the change of variable $ x=S_{\e ij}y $. 
With $ |\det(S_{\e ij})|=\e^2 $, we now have
\begin{align}
	\label{e.fourier1}
	\hat{\So_{\e ij} f}(q) = \e^{-2} \int_{\R^{2n}} e^{-\img (\M_{\e ij} q)\cdot x} f(x) \frac{\d x}{(2\pi)^n} = \e^{-2} \hat{f}(\M_{\e ij} q).
\end{align}
Inserting~\eqref{e.fourier1} into the r.h.s.\ of~\eqref{e.pinverse},
and performing a change of variable $ q_1 \mapsto \e q_1 $, under which $M_{\eps ij }q \mapsto M_{ij}q$,
we conclude the desired result \eqref{e.Se.Four}.
\end{proof}

We now show the convergence. Recall $ \Po $ from~\eqref{e.Po}.
\begin{lemma}\label{l.BGconvg}
For each $ i<j $ and $ \Re(z)<0 $, we have
\begin{align*}%
	\NOrmo{ \phi\So_{\e ij}\Go_z - \phi\otimes (\So_{ij}\Go_{z}) }
	+
	\NOrmo{ \Go_z\So^*_{\e ij}\phi - \Go_{z}\So^*_{ij}\Po }
	\leq
	C\,%
	\e^\frac12 \, (-\Re(z))^{-1/4}
	\longrightarrow 0,
	\quad \text{as } \e \to 0.
\end{align*}
\end{lemma}
\begin{proof}
It suffices to consider %
$ \phi\So_{\e ij}\Go_z $ since $ \Go_z\So^*_{\e ij}\phi = (\phi\So_{\e ij}\Go_{\bar{z}})^* $ 
and $\Go_{z} \So^*_{ij}\Po = (\phi\otimes(\So_{ij}\Go_{\bar{z}}))^* $. 
 Fix $ u\in\Ssp(\R^{2n}) $, and, to simplify notation, let $ u' := (\phi\So_{\e ij} \Go_z -\phi\otimes(\So_{ij}\Go_{z})) u $.
We use~\eqref{e.SG.fourier} and \eqref{e.Se.Four} to calculate the partial Fourier transform of $ u' $ as
\begin{align*}
	\pFou{u'}(y_1,q_{2-n})
	=
	\phi(y_1) \int_{\R^2} \frac{e^{\img \e y_1\cdot q_1}-1}{\frac12|\M_{ij}q|^2-z} \hat{u}(M_{ij}q) \frac{\d q_1}{2\pi}.
\end{align*}
From this we calculate the norm of $ u' $ as
\begin{align*}
	\norm{u'}^2 = \int_{\R^{2n}} |\pFou{u'}(y_1,q_{2-n})|^2 \d y_1 \d q_{2-n}
	=
	\int_{\R^{2n}} \Big| \phi(y_1)\int_{\R^2} 
	\frac{ e^{\img \e y_1\cdot q_1}-1 }{ \frac12|\M_{ij}q|^2-z } \hat{u}(M_{ij}q) \frac{\d q_1}{2\pi} \Big|^2 
	\d y_1 \d q_{2-n}.
\end{align*}
Recall that, by assumption, $ \phi \in \Csp^\infty_\cmp(\R^{2}) $ is fixed, so $ |\phi(y_1)| \leq C \1_{\{|y_1|\leq C\}} $.
For $ |y_1|\leq C $ we have $ |e^{\img \e y_1\cdot q_1}-1| \leq C \, ((\e|q_1|)\wedge 1) $.
Using this and $ |\M_{ij}q|^2 \geq 2 |q_1|^2 $ (as verified from~\eqref{e.M}), we have
\begin{align*}
	\norm{u'}^2
	\leq
	C
	\int_{\R^{2n-2}} \Big( \int_{\R^2} \frac{ (\e|q_1|)\wedge 1 }{ |q_1|^2-\Re(z) } |\hat{u}(M_{ij}q)| \frac{\d q_1}{2\pi} \Big)^2 \d q_{2-n}
	\leq
	C
	\norm{u}^2 \, \int_{\R^2} \Big( \frac{ (\e|q_1|)\wedge 1 }{ |q_1|^2-\Re(z) } \Big)^{2} \, \d q_1.	
\end{align*}
Set $ -\Re(z)=a>0 $ to simplify notation.
We perform a change of variable $ q_1\mapsto \sqrt{a}q_1 $ in the last integral to get
$ \frac{1}{a}\int_{\R^2} \frac{(\eps\sqrt{a}|q_1|)^2\wedge 1}{(|q_1|^2+1)^2} \, \d q_1$.
Decompose it according to $ |q_1| < \e^{1/2}a^{1/4} $ and $ |q_1| > \e^{1/2}a^{1/4} $.
For the former use $ \frac{(\eps\sqrt{a}|q_1|)^2\wedge 1}{(|q_1|^2+1)^2} \leq 1 $,
and for the latter use $ (\eps\sqrt{a}|q_1|)^2\wedge 1 \leq (\eps\sqrt{a}|q_1|)^2 $.
It is readily checked that the integrals are both bounded by $ C\eps a^{-1/2} $.
\end{proof}

\section{Off-diagonal mediating operators}
\label{sect.offdiagonal}
To get a rough idea of how the mediating operators (those in~\eqref{e.resolvent.e.med}) should behave as $ \e\to 0 $,
we perform a regularity exponent count similar to the discussion just before Lemma~\ref{l.Goin}.
Recall that $ \Go_z $ increases regularity by  $ 2 $,
while  $ \So_{ij} $ and $ \So^*_{k\ell} $ \emph{decrease} regularity by $ -(1^+) $.
\emph{Formally} the regularity of $ \So_{ij} \Go_z \So^*_{k\ell} $ adds up to $ 2-(1^+)-(1^+)=0^-<0 $.
This being the case, one might expect $ \So_{\e ij} \Go_z \So^*_{\e k\ell} $ to diverge, in a somewhat marginal way, as $ \e\to 0 $.%

As we will show in the next section, the diagonal operator $ \So_{\e 12} \Go_z \So^*_{\e 12} $ diverges logarithmically in $ \e $. 
This divergence, after a suitable manipulation, cancels the relevant, leading order divergence in $ \beta_\e^{-1}\Io $
(recall from~\eqref{e.betaeps} that $ \beta_\e^{-1}\to\infty $).
On the other hand, for each $ (i<j)\neq(k<\ell) $, the off-diagonal operator $ \So_{\e ij} \Go_z \So^*_{\e k\ell} $ converges.
This is not an obvious fact, cannot be teased out from the preceding regularity counting,
and is ultimately due to an inequality derived in \cite[Equation~(3.2)]{dell94}.
We treat the off-diagonal terms in this section.

We begin by building the limiting operator.

\begin{lemma}\label{l.off.diagonal}
Fix $ (i<j)\neq (k<\ell) $ and $ \Re(z)<0 $.
We have that $ \Go_z\So^{*}_{k\ell}(\Lsp^2(\R^{2n-2})) \subset \Dom(\So_{ij}) $, 
so $ \So_{ij}\Go_z\So^*_{k\ell} $ maps $ \Lsp^2(\R^{2n-2})$ to $\Lsp^2(\R^{2n-2}) $.
Furthermore, $ \normo{ \So_{ij}\Go_z\So^*_{k\ell} } \leq C $ and
\begin{align}
	\label{e.GbiQ.Four}
	 \Ip{g}{ \So_{ij}\Go_z\So^*_{k\ell} f}
	&=
	\int_{\R^{2n}} \bar{\hat{g}(p_i+p_j,p_{\bar{ij}})}
	\frac{ 1 }{\frac12|p|^2-z}
	\hat{f}(p_{k}+p_\ell,p_{\bar{k\ell}}) 
	\ \frac{\d p}{(2\pi)^2},
\end{align}
for $ f,g \in \Lsp^2(\R^{2n-2}) $.
\end{lemma}
\begin{proof}
The inequalities derived in \cite[Equations~(3.1), (3.3), (3.4), (3.6)]{dell94} translate, under our notation, into
\begin{align}
	\label{e.dell}
	\sup_{\alpha> 0} \int_{\R^{2n}} \frac{|\hat{g}(p_i+p_j,p_{\bar{ij}})|\,|\hat{f}(p_k+p_\ell,p_{\bar{k\ell}})|}{|p|^2+\alpha} \d p
	\leq
	C\, 
	\norm{g} 
	\,
	\norm{f}, 
\end{align}
for all $ (i<j)\neq(k<\ell) $ and $ f,g \in \Lsp^2(\R^{2n-2}) $. Also, from \eqref{e.So*.Four} we have
\begin{align}
	\label{e.S*ijS*kl}
	\hat{\So^*_{ij}g}(p) = \tfrac{1}{2\pi} \hat{g}(p_i+p_j,p_{\bar{ij}}),
	\qquad
	\hat{\So^*_{k\ell}f}(p) = \tfrac{1}{2\pi} \hat{f}(p_{k}+p_\ell,p_{\bar{k\ell}}).
\end{align}
A priori, we only have $ \Go_z\So^*_{k\ell}f\in\Lsp^2(\R^{2n}) $ from Lemma~\ref{l.So}.
Given \eqref{e.dell}--\eqref{e.S*ijS*kl} together with $ \Re(z)<0 $, we further obtain
\begin{align}
	\label{e.dell.conseq}
	\int_{\R^{2n}} \Big| \hat{g}(p_i+p_j,p_{\bar{ij}}) \frac{1}{\frac12|p|^2-z}\hat{\So_{k\ell}^*f}(p) \Big| \, \d p
	=
	\int_{\R^{2n}} \Big| \hat{g}(q_{2-n}) \frac{1}{\frac12|\M_{ij}q|^2-z}\hat{\So_{k\ell}^*f}(M_{ij}q) \Big| \, \d q
	\leq
	C\, 
	\norm{g}
	\,
	\norm{f},
\end{align}
where, in deriving the equality, we apply a change of variable $ q=M_{ij}^{-1}p $,
together with $ (p_i+p_j,p_{\bar{ij}})=(\M_{ij}^{-1}p)_{2-n} $
and $ |\det(M_{ij})|=1 $ (as readily verified from~\eqref{e.M}).
Referring to the definition~\eqref{e.domSo} of $ \Dom(\So_{ij}) $,
since~\eqref{e.dell.conseq} holds for all $ g\in\Lsp^2(\R^{2n-2}) $,
we conclude $ \Go_z\So^*_{k\ell}f \in \Dom(\So_{ij}) $
and further  that 
$ |\ip{ g }{ \So_{ij}\Go_z\So^*_{k\ell}f }| $
$=$
$ |\ip{ \So^*_{ij}g }{ \Go_z\So^*_{k\ell}f }|$
$ \leq $
$C \norm{g}\,\norm{f} $.
The desired identity \eqref{e.GbiQ.Four} now follows from~\eqref{e.S*ijS*kl}.
\end{proof}

We next derive the $ \e>0 $ analog of~\eqref{e.GbiQ.Four}.
Recall that $ \pFou{v}(y_1,q_{2-n}) $ denotes partial Fourier transform in the last $ n-1 $ components.

\begin{lemma}\label{l.Gobi.bilinear}
For (not necessarily distinct) $ (i<j), (k<\ell) $, $ \Re(z)<0 $, and $ v,w\in\Ssp(\R^{2n}) $,
\begin{subequations}
\begin{align}
	\label{e.GbieQ.Four}
	&\Ip{w}{ \So_{\e ij}\Go_z\So^*_{\e k\ell} v}
	=
	\int_{\R^{2n}} \bar{\hat{w}(\tfrac{\e}{2}(p_i-p_j),p_i+p_j,p_{\bar{ij}})}
	\frac{ 1 }{\frac12|p|^2-z}
	\hat{v}(\tfrac{\e}{2}(p_k-p_\ell),p_{k}+p_\ell,p_{\bar{k\ell}}) 
	\ \d p
\\
	\label{e.GbieQ.pFour}
	&\quad=
	\int_{\R^{2+2+2n}} \bar{\pFou{w}(y'_1,p_i+p_j,p_{\bar{ij}})}
	\frac{ e^{\frac12\img\e ((p_i-p_j)\cdot y'_1-(p_k-p_\ell)\cdot y_1)} }{\frac12|p|^2-z}
	\pFou{v}(y_1,p_{k}+p_\ell,p_{\bar{k\ell}}) 
	\ \frac{\d y_1 dy'_1 \d p}{(2\pi)^2}.
\end{align}
\end{subequations}
\end{lemma}
\begin{proof}
Fixing $ v,w\in\Ssp(\R^{2n}) $,
we write $ \ip{w}{ \So_{\e ij}\Go_z\So^*_{\e k\ell}v} = \ip{\So^*_{\e ij}w}{ \Go_z\So^*_{\e k\ell}v} $.
Our goal is to express the last quantity in Fourier variables,
which amounts to expressing $ \So^*_{\e k\ell}v $ and $ \So^*_{\e ij}w $ in Fourier variables.
Recall (from~\eqref{e.Soe}) that $ \So^*_{\e ij} $ acts on $ \Lsp(\R^{2n}) $ by 
$ v(\Cdot)\mapsto \e^{-2}v(T_{\e ij}\Cdot) $, where $ T_{\e ij} $ is the invertible linear transformation defined in~\eqref{e.T}.
Write
\begin{align*}
	\hat{\So^*_{\e ij}w} (p) 
	= 
	\int_{\R^{2n}} e^{-\img p\cdot x} \e^{-2} w(T_{\e ij}x) \frac{\d x}{(2\pi)^n}.
\end{align*}
We wish to perform a change of variable $ T_{\e ij}x=y $.
Doing so requires understanding how $ (p\cdot x) $ transform accordingly.
Defining
$
	\tilde{\M}_{\e ij}p := (\tfrac{\e}{2}(p_i-p_j),p_i+p_j,p_{\bar{ij}}),
$
it is readily checked that $ p \cdot x = \tilde{\M}_{\e ij} p \cdot(T_{\e ij}x) $.
Given this, we perform the change of variable $ T_{\e ij}x=y $. 
With $ |\det(T_{\e ij})|=\e^{-2} $, we now have
\begin{align*}
	\hat{\So^*_{\e ij}w} (p) 
	= 
	\int_{\R^{2n}} e^{-\img (\tilde{\M}_{\e ij} p)\cdot y}  w(y) \frac{\d y}{(2\pi)^n}
	=
	\hat{w}(\tilde{\M}_{\e ij}p)
	=
	\hat{w}(\tfrac{\e}{2}(p_i-p_j),p_i+p_j,p_{\bar{ij}}),
\end{align*}
and similarly $ \hat{\So^*_{\e k\ell }v} (p) = \hat{v}(\frac{\e}{2}(p_k-p_\ell),p_k+p_\ell,p_{\bar{k\ell}}) $.
From these expressions of $ \So^*_{\e k\ell}v $ and $ \So^*_{\e ij}w $ 
we conclude~\eqref{e.GbieQ.Four}.
The identity~\eqref{e.GbieQ.pFour} follows from~\eqref{e.GbieQ.Four}
by writing $ \pFou{v}(y_1,p_{2-n}) = \int_{\R^2} e^{\img y_1\cdot p_1} \hat{v}(p) \frac{\d p_1}{2\pi} $
(and similarly for $ \pFou{w} $).
\end{proof}

A useful consequence of Lemma~\ref{l.Gobi.bilinear} is the following norm bound.
\begin{lemma}\label{l.off.diagonal.bdd}
For  distinct $ (i<j) \neq (k<\ell) $, $ \Re(z)<0 $, and $ \e\in(0,1) $,
$
	\normo{\phi\So_{\e ij}\Go_z\So^*_{\e k\ell}\phi} \leq C.
$
\end{lemma}
\begin{proof}
In~\eqref{e.GbieQ.pFour}, 
apply~\eqref{e.dell} with $ f(\Cdot) =\phi(y_1) \pFou{v}(y_1,\Cdot) $ and $ g(\Cdot) = \phi(y'_1)\pFou{w}(y'_1,\Cdot) $.
and integrate the result over $ y_1 ,y'_1 $.
We have
\begin{align*}
	|\ip{\phi w}{\So_{\e ij}\Go_z\So^*_{\e k\ell} ( \phi v)}|	
	\leq
	C\,
	\int_{\R^{2}}
	\norm{v(y_1,\Cdot)}\phi(y_1) \d y_1
	\,
	\int_{\R^{2}}
	\norm{w(y'_1,\Cdot)} \phi(y'_1) \d y'_1.
\end{align*}
The last expression, upon an application of the Cauchy--Swchwarz inequality in $ y_1 $ and in $ y'_1 $,
is bounded by $ C\norm{v} \, \norm{w} $.
From this we conclude $ \normo{\phi\So_{\e ij}\Go_z\So^*_{\e k\ell}\phi} \leq C $.
\end{proof}

We are now ready to establish the convergence of the operator $ \phi \, \So_{\e ij}\Go_z\So^*_{\e k\ell} \,\phi $ for distinct pairs.
Recall $ \Po $ from~\eqref{e.Po}.
\begin{lemma}\label{l.Gobi.cnvg}
For each $ (i<j)\neq(k<\ell) $, and $ \Re(z)<0 $,
we have $ \phi \, \So_{\e ij}\Go_z\So^*_{\e k\ell} \,\phi \to \phi\otimes (\So_{ij}\Go_z\So^*_{k\ell} \Po) $ strongly as $ \e\to 0 $.
\end{lemma}
\begin{proof}
Our goal is to show $ \phi\,\So_{\e ij}\Go_z\So^*_{\e k\ell}\phi v \to \phi\otimes \So_{ij}\Go_z\So^*_{k\ell}\Po v  $,
for each $ v\in\Lsp^2(\R^{2n}) $.
As shown in Lemmas~\ref{l.off.diagonal} and \ref{l.off.diagonal.bdd},
the operators $ (\So_{\e ij}\Go_z\So^*_{\e k\ell}) $ and $ (\So_{ij}\Go_z\So^*_{k\ell}) $ are norm-bounded, uniformly in $ \e $.
Hence it suffices to consider $ v\in \Ssp(\R^{2n}) $, the Schwartz space.
To simplify notation, set
$
	u_{\e}:= (\phi\,\So_{\e ij}\Go_z\So^*_{\e k\ell}\phi) v
$
and
$
	u:=(\phi\otimes \So_{ij}\Go_z\So^*_{k\ell}\Po) v.
$
The strategy of the proof is to express $ \norm{u_\e-u}^2 $ as an integral, and use the dominated convergence theorem.

The first step is to obtain expressions for the partial Fourier transforms of $ u_\e=(\phi\,\So_{\e ij}\Go_z\So^*_{\e k\ell}\phi) v $ 
and $ u=(\phi\otimes \So_{ij}\Go_z\So^*_{k\ell}\Po) v $.
To this end, fix $ v,w\in \Ssp(\R^{2n}) $,
in~\eqref{e.GbiQ.Four}, set $ (f(\Cdot),g(\Cdot)) = (\phi(y_1) v(y_1,\Cdot),\phi(y'_1)w(y'_1,\Cdot)) $,
and integrate over $ y_1,y'_1 $.
Note that $ \hat{f}(p_{2-n}) = \phi(y_1) \pFou{v}(y_1,p_{2-n}) $ (and similar for $ g $).
We have
\begin{align}
	\tag{\ref*{e.GbiQ.Four}'}
	\label{e.GbiQ.Four.}
	\Ip{ w }{ u } 
	&=
	\int_{\R^{2+2+2n}} \bar{\pFou{w}(y'_1,p_i+p_j,p_{\bar{ij}})} \phi(y'_1)
	\frac{ 1 }{\frac12|p|^2-z}
	\phi(y_1)
	\pFou{v}(y_1,p_k+p_\ell,p_{\bar{k \ell}}) 
	\ \frac{\d y_1 \d y'_1\d p}{(2\pi)^2}.
\end{align}
Similarly, in~\eqref{e.GbieQ.pFour}, substitute $ (v,w)=(\phi v, \phi w) $ to get
\begin{align}
\tag{\ref*{e.GbieQ.pFour}'}
\label{e.GbieQ.pFour.}
	\Ip{ w }{ u_\e } 
	=
	\int_{\R^{2+2+2n}}& \bar{\pFou{w}(y'_1,p_i+p_j,p_{\bar{ij}})} 
\phi(y'_1)
	\frac{ e^{\frac12\img\e ((p_i-p_j)\cdot y'_1-(p_k-p_\ell)\cdot y_1)} }{\frac12|p|^2-z}
	\phi(y_1)
	\pFou{v}(y_1,p_k+p_\ell,p_{\bar{k \ell}}) 
	\ \frac{\d y_1 \d y'_1\d p}{(2\pi)^2}.
\end{align}
%
Equations~\eqref{e.GbiQ.Four.} and \eqref{e.GbieQ.pFour.} express
the inner product (against a generic $ w $) of $ u_\e $ and $ u$ in partial Fourier variables.
From these expressions we can read off $ \pFou{u_\e}(y'_1,q_{2-n}) $ and $ \pFou{u}(y'_1,q_{2-n}) $.
Specifically, we perform a change of variable $ q=\M^{-1}_{ij}p=(\tfrac{1}{2}(p_i-p_j),p_i+p_j,p_{\bar{ij}}) $ in \eqref{e.GbiQ.Four.} and \eqref{e.GbieQ.pFour.},
so that $ \pFou{w} $  takes variables $ (y'_1,q_{2-n}) $ instead of $ (y'_1,p_i+p_j,p_{\bar{ij}}) $.
From the result we read off
\begin{align}
	\label{e.long.}
	\pFou{u} \big( y'_1,q_{2-n} \big)
	=
	\int_{\R^{4}}  f_{z,v} \, \d y_1 \d q_1,
\quad
	\pFou{u_\e} \big( y'_1,q_{2-n} \big)
	=
	\int_{\R^{4}}E_\e\, f_{z,v} \, \d y_1 \d q_1.
\end{align}
Here $ E_\e $ and $ f_{z,v} $ are (rather complicated-looking) functions of $ q,y_1,y'_1 $,
given in the following.
The precisely functional forms of $ f_{z,v} $ and $ E_\e $ will be irrelevant.
Instead, we will explicitly signify 
what properties of these functions we are using whenever doing so.
We have $ E_\e := e^{\img\e q_1\cdot y'_1-\img\e[\M^{-1}_{k\ell}\M_{ij}q]_1\cdot y_1} $ and
\begin{align*}
	f_{z,v}
	:=
	\phi(y'_1)\frac{ 1 }{\frac12|\M_{ij}q|^2-z} \phi(y_1)
	\pFou{v}(y_1,[\M_{k\ell}^{-1}M_{ij}q]_{2-n})
	\ \frac{1}{(2\pi)^2}.
\end{align*}

Additionally, we will need an auxiliary function $ v'\in\Lsp^2(\R^{2n}) $ such that $ \pFou{v'}(y_1,\tilp) = |\pFou{v}(y_1,\tilp)| $.
Such a function $ v'=v'(y) $ is obtained by taking inverse Fourier of $ |\pFou{v}(y_1,q_{2-n})| $ in $ q_{2-n} $.
Note that $ \norm{v'}=\norm{v}<\infty $.
Set $ a:=-\Re(z)>0 $ and $ u' := (\phi\otimes \So_{ij}\Go_{-a}\So_{k\ell}^*\Po) v' $. 
We have
\begin{align}
	\label{e.short}
	\pFou{u'} \big( y'_1,q_{2-n} \big)
	=
	\int_{\R^{4}} f_{-a,v'} \, \d y_1 \d q_1,
	\qquad
	f_{-a,v'} \geq |f_{z,v}| \geq 0.
\end{align}

Now, use~\eqref{e.long.} and~\eqref{e.short} to write
\begin{align}
	\label{e.cmp1}
	\norm{ u_\e - u }^2
	&\leq
	\int_{\R^{2n}} \Big( \int_{\R^{4}} |f_{z,v}| \, |E_\e-1| \d y_1 \d q_1 \Big)^2 \d y'_1 \d q_{2-n},
\\
	\label{e.cmp2}
	\norm{ u' }^2
	&=
	\int_{\R^{2n}} \Big( \int_{\R^{4}} f_{-a,v'} \d y_1 \d q_1 \Big)^2 \d y'_1 \d q_{2-n}.
\end{align}
View~\eqref{e.cmp1}--\eqref{e.cmp2} as integrals over $ \R^{8+2n} $, i.e.,
\begin{align*}
	\text{r.h.s.\ of \eqref{e.cmp1}}
	:=
	\int_{\R^{8+2n}} g_\e \, \d(\ldots),
\qquad
	\text{r.h.s.\ of \eqref{e.cmp2}}
	:=
	\int_{\R^{8+2n}} g \, \d(\ldots).
\end{align*}
We now wish to apply the dominated convergence theorem on $ g_\e $ and $ g $.
To check the relevant conditions, note that:
since $ |E_\e-1| \leq 1 $ and $ |f_{z,v}|\leq f_{-a,v'} $, we have $ 0 \leq g_\e \leq g $;
since $ |E_\e-1| \to 0 $ pointwisely on $ \R^{8+2n} $, we have $ g_\e \to 0 $ pointwise on $ \R^{8+2n} $;
the integral of $ g $ over $ \R^{8+2n} $ evaluates to $ \norm{u'}^2 = \norm{(\phi\otimes \So_{ij}\Go_z\So^*_{k\ell} \Po) v'}^2 $, 
which is \emph{finite} since the operators $ \So_{ij}\Go_z\So^*_{k\ell} $, $ (\phi\otimes\,\Cdot\,) $, and $ \Omega_\phi$ are bounded.
The desired result $ \int_{\R^{8+2n}} g_\e \, \d(\ldots) = \norm{ w_\e - w }^2 \to 0 $ follows.
\end{proof}

\section{Diagonal mediating operators}
\label{sect.diagonal}
The main task here is to analyze the asymptotic behavior of the diagonal part $ \phi \, \So_{\e 12}\Go_z\So^*_{\e 12} \phi $,
which diverges logarithmically.
We begin by deriving an expression for $ \ip{w}{ \phi \, \So_{\e 12}\Go_z\So^*_{\e 12} \phi \,v} $ that exposes such $ \e \to 0 $ behavior.
Let $ \Gtwo_z(x):=( -\frac12  \nabla^2 - z\Io )^{-1}(0,x) $, $ x\in\R^2 $, denote Green's function in two dimensions.
Recall that $ |p|^2_{2-n} := \frac{1}{2}|p_2|^2 + |p_3|^2 +\ldots+|p_n|^2 $.

\begin{lemma}\label{l.G1212}
For $ v,w\in\Lsp^2(\R^{2n}) $, we have
\begin{align}
\label{e.G1212}
	&\ip{w}{ \phi \, \So_{\e 12}\Go_z\So^*_{\e 12} \phi v}
=
	\int_{\R^{2n}} \bar{\pFou{w}(y'_1,p_{2-n})} 
	\, 
	\phi(y'_1) \,\tfrac12 \Gtwo_{\e^2(\frac12z-\frac14|p|^2_{2-n})}\big(y'_1-y_1\big) 
	\, \phi(y_1) \, \pFou{v}(y_1,p_{2-n})
	\ 
	\d y_1 \d y'_1 \d p_{2-n}.
\end{align}
\end{lemma}
\begin{proof}
Apply Lemma~\ref{l.Gobi.bilinear} for $ (i<j)=(k<\ell)=(1<2) $
and for $ (v,w)\mapsto (\phi v, \phi w) $,
and perform a change of variable $ (\frac{p_1-p_2}{2},p_1+p_2) \mapsto (p_1,p_2) $ in the result.
We obtain
\begin{align}
	\label{e.G1212.}
	\Ip{w}{\phi\So_{\e 12}\Go_z\So^*_{\e 12}\phi v}
	=
	\int_{\R^{2+2+2n}} \bar{\pFou{w}(y'_1,p_{2-n})}
	\phi(y'_1) \frac{ e^{
	\img\e p_1\cdot (y'_1-y_1)} }{|p_1|^2+\frac12|p|^2_{2-n}-z}
	\phi(y_1) \pFou{v}(y_1,p_{2-n})
	\ \frac{\d y_1 dy'_1 \d p}{(2\pi)^2},
\end{align}
and we recognize $\int_{\R^2} \frac{ 
	e^{\img p_1\cdot x_1}}{\frac12 |p_1|^2-z} \frac{\d p_1}{(2\pi)^2}$ as the Fourier transform of the two-dimensional Green's function $ \Gtwo_z $.
\end{proof}

Lemma~\ref{l.G1212} suggests analyzing the behavior of $ \Gtwo_z(x) $ for small $|z|$:
\begin{lemma}\label{l.Gtwo}
Take the branch cut of the complex-variable functions $ \sqrt{z} $ and $ (\log z) $ to be $ (-\infty,0] $,
let $ \cEM $ denote the Euler--Mascheroni constant.
For all $ x \neq 0 $ and $ z \in \C\setminus [0,\infty) $, we have
\begin{align}
	\label{e.Gtwo}
	\Gtwo_{z}(x)
	=\tfrac1{\pi} K_0(\sqrt{-2z}|x|)
	=
	-\tfrac{1}{\pi} \log\tfrac{\sqrt{-z}|x|}{\sqrt{2}}
	-\tfrac{1}{\pi} \cEM
	+
	\smallt(\sqrt{-z}x),
\end{align}
for some $ \smallt(\Cdot) $ that grows linearly near the origin, i.e., $ \sup_{|z|\leq a} (|z|^{-1}\,|\smallt(z)|) \leq C(a) $, for all $ a<\infty $.
\end{lemma}

\noindent
The proof follows from classical special function theory.
We present it here for the convenience of the readers.
\begin{proof}
Write the equation $ (-\frac12\nabla^2 - z) \Gtwo_z(x) =0 $, $ x\neq 0 $, in radial coordinate,
compare the result to the modified Bessel equation~\cite[9.6.1]{abramowitz65},
and note that $ \Gtwo_z(x) $ vanishes at $ |x|\to\infty $.
We see that $ \Gtwo_z(x) = c K_0(\sqrt{-2z}|x|) $, for some constant $ c $,
where $ K_\nu $ denotes the modified Bessel function of second kind.
To fix $ c $,
compare the know expansion of $ K_0(r) $ around $ r=0 $ \cite[9.6.54]{abramowitz65} (noting that $ I_0(0)= 1 $ therein),
and use $ - \pi r \frac{\d~}{\d r}G_{z}(|r|) = 1 $ (because $ (-\frac12\nabla^2G_z(x) -z)=\delta(x) $)  for $ r\to 0 $.
We find $ c= \frac1{\pi} $.
The second equality follows from~\cite[9.6.54]{abramowitz65}.
\end{proof}

For subsequent analysis, it is convenient to decompose $ \Lsp^2(\R^{2n}) $
into a `projection onto $ \phi $' and its orthogonal compliment.
More precisely, recall $ \Po $ from~\eqref{e.Po}, and that $\int \phi^2=1$, we define the projection
\begin{align}
	\label{e.Pro}
	\Pro_\phi := \phi \otimes \Po: \Lsp^2(\R^{2n}) \to \Lsp^2(\R^{2n}),
	\quad
	(\phi\otimes\Po v)(y) := \phi(y_1) \int_{\R^2} \phi(y'_1)v(y'_1,y_{2-n}) \, \d y'_1.
\end{align}

Returning to the discussion about the $ \e\to 0 $ behavior of $ \phi \, \So_{\e 12}\Go_z\So^*_{\e 12} \phi $,
inserting \eqref{e.Gtwo} into \eqref{e.G1212}, we see that $ (\phi\So_{\e 12}\Go_z\So^*_{\e 12}\phi) $ has a divergent part
$
	(\frac{1}{2\pi} |\log\e| ) \Pro_\phi.
$
The coefficient $ (\frac{1}{2\pi}|\log\e|) $ matches the leading order of $ \beta_\e^{-1} $ (see~\eqref{e.betaeps}),
so $ (\frac{1}{2\pi} |\log\e| ) \Pro_\phi $ cancels the divergence $ \beta^{-1}_\e \Io $ on the subspace $ \Img(\Pro_\phi) $,
but still leaves the remaining part $ \beta^{-1}_\e\Io|_{\Img(\Pro_\phi)^\perp} = \beta^{-1}_\e(\Io-\Pro_\phi) $ divergent.
However, recall that $ (\beta^{-1}_\e\Io -\phi \, \So_{\e 12}\Go_z\So^*_{\e 12} \phi) $ 
appears as an \emph{inverse} in the resolvent identity~\eqref{e.resolvent.e}.
Upon taking inverse, the divergent part on $ \Img(\Pro_\phi)^\perp $ becomes a vanishing term.%

We now begin to show the convergence of $ (\beta_\e^{-1}\Io-\phi \, \So_{\e 12}\Go_z\So^*_{\e 12} \phi )^{-1} $.
Doing so requires a technical lemma.
To setup the lemma, consider
a collection of bounded operators $ \{ \auxo_{\e,p} : \Lsp^2(\R^{2})\to \Lsp^2(\R^{2})\}$,
indexed by $ \eps\in(0,1) $ and $ p\in\R^{2n-2} $, such that for each $\eps>0$, 
$	\sup_{p\in\R^{2n-2}} \normo{\auxo_{\e,p}} <\infty. $
Note that here, unlike in the preceding, here $ p=(p_2,\ldots,p_n)\in\R^{2n-2} $ denotes a vector of $ n-1 $ components. 
For each $ \e\in(0,1) $, construct a bounded operator $ \auxo_\e $ as
\begin{align*}
	\auxo_\e: \Lsp^2(\R^{2n}) \to \Lsp^2(\R^{2n}),
	\quad
	\pFou{\auxo_\e u}(\Cdot,p) := \auxo_{\eps,p} \pFou{u}(\Cdot,p).
\end{align*}
Roughly speaking, we are interested in an operator $ \auxo_\e $ that acts on $ y_1\in\R^{2} $ in a way that depends on the partial Fourier components $ p=(p_2,\ldots,p_n)\in\R^{2n-2} $.
The operator $ \auxo_{\e,p} $ records the action of $ \auxo_{\e} $ on $ y_1 $ per \emph{fixed} $ p\in\R^{2n-2} $.
We are interested in obtaining the inverse $ \auxo_\e^{-1} $ and its strong convergence (as $ \e\downarrow 0 $).
The following lemma gives the suitable criteria in terms of each $ \auxo_{\e,p} $.

\begin{lemma}\label{l.auxi}
Let $\{ \auxo_{\e,p} \} $ and $ \auxo_\e $ be as in the preceding.
If each $ \auxo_{\eps,p} $ is invertible with
\begin{align*}
	\sup\big\{ \normo{ \auxo_{\eps,p}^{-1} } : \eps\in(0,1),p\in\R^{2n-2} \big\} := b <\infty,
\end{align*}
and if each $ \auxo_{\eps,p}^{-1} $ permits a norm limit, i.e.,
there exists $ \auxo'_{p}: \Lsp^2(\R^2) \to \Lsp^2(\R^2) $ such that
\begin{align*}
	\auxo_{\eps,p}^{-1} \longrightarrow \auxo'_{p}	
	\text{ in norm as } \eps\to0,
	\quad
	\text{ for each fixed } p\in\R^{2n-2},
\end{align*}
then $ \auxo_{\eps} $ is invertible with $ \sup_{\eps\in(0,1)} \normo{ \auxo_{\eps}^{-1} } \leq b <\infty $,
\begin{align*}
	\auxo_\eps^{-1} &\longrightarrow \auxo', \quad \text{strongly, as } \e\to 0,
\end{align*}
and $ \normo{\auxo'} \leq b <\infty $,
where the operator $ \auxo': \Lsp^2(\R^{2n})\to\Lsp^2(\R^{2n}) $
is built from the limit of each $ \auxo_{\e,p}^{-1} $ as $ \pFou{\auxo'u}(\Cdot,p) := \auxo'_{p} \pFou{u}(\Cdot,p) $.
\end{lemma}

\begin{proof}
We begin by constructing the inverse of $ \auxo_\e $.
By assumption each $ \auxo_{\eps,p} $ has inverse $ \auxo^{-1}_{\eps,p} $,
from which we define
$
	\pFou{\auxo'_{\e} u}(\Cdot,p) := \auxo^{-1}_{\eps,p} \pFou{u}(\Cdot,p).
$
It is readily checked that $ \normo{\auxo'_{\e}} \leq \sup_{\e,p}\|\auxo_{\eps,p}^{-1}\|\leq b $,
and the operator $ \auxo'_{\e} $ actually gives the inverse of $ \auxo_\e $, 
i.e., $ \auxo'_\e\auxo_\e = \auxo_\e\auxo'_\e = \Io $.
Note that, for each $ p\in\R^{2n-2} $, the operator $ \auxo'_{p} $ inherits a bound from $ \auxo^{-1}_{\e,p} $,
i.e., $ \sup_{p} \normo{\auxo'_{p}} \leq \sup_{\e,p} \normo{ \auxo^{-1}_{\e,p}} \leq b $.
Together with the definition of $\auxo'$ we also have $ \normo{\auxo'} \leq b $.

It remains to check the strong convergence. 
For each $ u\in\Lsp^2(\R^{2n}) $ we have
\begin{align*}
	\norm{ \auxo_\eps^{-1}u- \auxo'u }^2
	&=
	\int_{\R^{2n-2}} \Big( \int_{\R^2} |\auxo_{\eps,p}^{-1}\pFou{u}(y_1,p) - \auxo'_{p}\pFou{u}(y_1,p)|^2 \, \d y_1 \Big) \d p
\\
	&
	\leq
	\int_{\R^{2n-2}} \left(\normo{ \auxo_{\eps,p}^{-1} -\auxo'_{p} }^2 \,\int_{\R^2} |\pFou{u}(y_1,p)|^2 \, \d y_1\right) \d p.
\end{align*}
The integrand within the last integral converges to zero pointwisely,
and is dominated by $ 4b^2|\pFou{u}(y_1,p)|^2 $, which is integrable over $ \R^{2n} $.
Hence by the dominated convergence theorem $ \norm{ \auxo_\eps^{-1}u- \auxo'u }^2 \to 0 $.
\end{proof}

With Lemma~\ref{l.auxi},
we next establish the norm boundedness and strong convergence of $ (\beta^{-1}_\e\Io -\phi \, \So_{\e 12}\Go_z\So^*_{\e 12} \phi)^{-1} $ in two steps,
first for \emph{fixed} $ p \in\R^{2n-2} $.
Slightly abusing notation, 
in the following lemma, we also treat $ \Pro_\phi $ (defined in~\eqref{e.Pro})
as its analog on $ \Lsp^2(\R^2) $, namely the projection operator $ \Pro_\phi f(y_1) := \phi(y_1) \int_{\R^2} \phi(y'_1) f(y'_1) \, \d y'_1 $.
\begin{lemma}\label{l.diagonal.fixed}
For each $ p \in\R^{2n-2} $, define an operator $ \auxo_{\e,p}:\Lsp^2(\R^2) \to \Lsp^2(\R^2) $,
\begin{align}
	\label{e.auxo.choice}
	\auxo_{\e,p} f(y_1) 
	:= 
	\beta_\e^{-1} f(y_1) - \phi(y_1) \int_{\R^2} \tfrac{1}{2} \Gtwo_{\eps^2(\frac12z-\frac14|p|^2_{2-n})}(y_1-y'_1) \phi(y'_1) f(y'_1) \, \d y'_1.
\end{align}
Then, there exist constants $ C_1<\infty,C_2(\betaf)>0 $ such that,
for all $ \Re(z)<-e^{\betaC+C_1} $ and $ \e\in(0,1/C_2(\betaf)) $,
\begin{align*}
	&
	\Normo{ ( \auxo_{\e,p})^{-1} } 
	\leq 
	C \, (\log(-\Re(z))-\betaC)^{-1},
\\	
	&
	\big( \auxo_{\e,p} \big)^{-1}
	\longrightarrow
	\frac{ 4\pi }{ \log(\frac12|p|^2_{2-n}-z) - \betaC} \Pro_\phi,
	\
	\text{ in norm as } \e\to 0, \text{ for each fixed } p\in\R^{2n-2}.
\end{align*}
\end{lemma}
\begin{proof}
Through out the proof, we say a statement holds for $ -\Re(z) $ large enough,
if the statement holds for all $ -\Re(z)> e^{\betaC+C} $, for some fixed constant $C<\infty $,
and we say a statement holds for all $ \e $ small enough,
if the statement holds for all $ \e < 1/C(\betaf) $, for some constant $ C(\betaf)<\infty $ that depends only on $ \betaf $.

Our first goal is to show $ \auxo_{\e,p} $ is invertible and establish bounds on $ \normo{\auxo_{\e,p}^{-1}} $.
We do this in two separate cases: \textit{i}) $ |\frac12|p|^2_{2-n}-z| \leq \e^{-2} $ 
and \textit{ii}) $ |\frac12|p|^2_{2-n}-z| > \e^{-2} $.

\medskip
\textit{i})
The first step here is to derive a suitable expansion of $ \auxo_{\e,p} $.
Recall that, we have abused notation to write $ \Pro_\phi $ (defined in~\eqref{e.Pro})
for the projection operator $ \Pro_\phi f(y_1) := \phi(y_1) \int_{\R^2} \phi(y'_1) f(y'_1) \, \d y'_1 $.
Applying Lemma~\ref{l.Gtwo} yields
\begin{align}
	\label{e.I-BGB.}
	\auxo_{\e,p}
	=
	\beta_\eps^{-1}\Io 
	+
	\Big( -\tfrac{1}{2\pi}|\log\e| + \tfrac{1}{4\pi} \log(\tfrac12|p|^2_{2-n}-z) - \tfrac1{2\pi}\log 2 + \tfrac1{2\pi}\cEM  \Big) \Pro_\phi	
	+	
	\logo_\phi	
	-
	\smallo_{\e,z,p},
\end{align}
where $ \logo_\phi $ and $ \smallo_{\e,z,p} $ are integral operators $ \Lsp^{2}(\R^{2})\to\Lsp^2(\R^{2}) $ defined as
\begin{align}
	\label{e.logo}
	(\logo_\phi f)(y_1) &:= \frac{1}{2\pi} \phi(y_1) \int_{\R^2} \log|y_1-y'_1| \phi(y'_1) f(y'_1) \, \d y'_1,
\\
	\label{e.smallo}
	(\smallo_{\e,z,p} f)(y_1) &:= \frac12 \phi(y_1) \int_{\R^2} \smallt\Big(\tfrac12|y_1-y'_1|\e\sqrt{\tfrac12|p|^2_{2-n}-z}\Big) \phi(y'_1) f(y'_1)
	\, \d y'_1,
\end{align}
and the function $ \smallt(\Cdot) $ is the remainder term in Lemma~\ref{l.Gtwo}.
Let $ \Pro_\perp := \Io - \Pro_\phi $ denote the orthogonal projection onto $ (\C\phi)^\perp $ in $ \Lsp^2(\R^2) $
and recall $ \beta_\e $ from~\eqref{e.betaeps}.
In~\eqref{e.I-BGB.}, decomposing $ \beta_\e^{-1} \Io = \beta_\e^{-1} \Pro_\perp + \frac{1}{2\pi}(|\log \e|-\betafe) \Pro_\phi $,
where $
	\betafe := |\log\e| - |\log\e|(1+\tfrac{\betaf}{|\log\e|})^{-1}$, 
we rearrange terms to get
\begin{align}
	\label{e.I-BGB.d}
	\auxo_{\e,p}
	=
	\beta_\eps^{-1}\Pro_\perp
	+
	\tfrac{1}{4\pi} \big( \log(\tfrac12|p|^2_{2-n}-z) -\betaCe' \big) \Pro_\phi 
	+	
	\logo_\phi	
	-
	\smallo_{\e,z,p},
\end{align}
where $
	\betaCe':=  2(\log 2+\betafe-\cEM)$.
We next take the inverse of $ \auxo_{\e,p} $ from \eqref{e.I-BGB.d}, 
utilizing
\begin{align}
	\label{e.inverse.id}
	\big( \geno-\tilde{\geno} \big)^{-1} &= \sum_{m=0}^\infty \geno^{-1} \big( \tilde{\geno}\geno^{-1} \big)^m,
\qquad
	\normo{( \geno-\tilde{\geno})^{-1}} \leq \normo{\geno^{-1}}/(1-\normo{\geno^{-1}} \normo{\tilde{\geno}} ),
\end{align}
valid for operators $ \geno, \tilde{\geno} $ such that $ \geno $ is invertible with $ \normo{\geno^{-1}} \normo{\tilde{\geno}} < 1 $.
Our choice will be
$
	\geno
	:= 
	\beta_\eps^{-1}\Pro_\perp
	+
	\tfrac{1}{4\pi} \big( \log(\tfrac12|p|^2_{2-n}-z) - \betaCe' \big) \Pro_\phi$ and 
$ \tilde{\geno} :=-\logo_\phi+\smallo_{\e,z,p} $.

From~\eqref{e.logo}, we have $ \normo{\logo_\phi} <\infty $.
Under our current assumption $ |\frac12|p|^2_{2-n}-z| \leq \e^{-2} $,
from~\eqref{e.smallo} and the property of $ \smallt(\Cdot) $ stated in Lemma~\ref{l.Gtwo},
we have $ \normo{\smallo_{\e,z,p}} \leq C <\infty $.
Hence
\begin{align}
	\label{e.genostartez.bd}
	 \Normo{-\logo_\phi+\smallo_{\e,z,p}} \leq C.
\end{align}
With $ \Pro_\perp $ and $ \Pro_\phi $ being projection operators orthogonal to each other,
we calculate
\begin{align}
	\label{e.geno.star}
	\left(\beta_\eps^{-1}\Pro_\perp
	+
	\tfrac{1}{4\pi} \big( \log(\tfrac12|p|^2_{2-n}-z) - \betaCe' \big) \Pro_\phi\right)^{-1} 
	= 
	\beta_\eps\Pro_\perp
	+
	4\pi \big( \log(\tfrac12|p|^2_{2-n}-z) -\betaCe' \big)^{-1} \Pro_\phi.
\end{align}
The operator norm of this inverse is thus bounded by $\max\{ \beta_\e \, , \, \frac{4\pi}{ \log(-\Re(z))-\betaCe'} \}$. 
Since $\betaCe' \to \betaC + 2\betaphi$ and  $ \beta_\e \to 0 $, this allows us to get a convergent series  	\eqref{e.inverse.id} for  $ -\Re(z) $ large enough and $ \e $ small enough,
with 
$ 
	\normo{ \auxo^{-1}_{\e,p} }
	\leq
	C (\log(-\Re(z))-\betaC)^{-1}$.

\medskip
\textit{ii})  Now we consider the case $ |\frac12|p|^2_{2-n}-z| > \e^{-2} $.
We apply~\eqref{e.inverse.id} again to ~\eqref{e.auxo.choice} with $ \geno = \beta^{-1}_\e \Io $.
To check the relevant condition, we write the operator $ \auxo_{\e,p} $ (in \eqref{e.auxo.choice}) in a coordinate-free form as
$
	\auxo_{\e,p}
	= 
	\beta_\e^{-1} \Io - \phi \tfrac{1}{2} \Go^{(n=1)}_{\eps^2(\frac12z-\frac14|p|^2_{2-n})} \phi$,
where $ \Go^{(n=1)}_z $ denotes the two-dimensional Laplace resolvent.
Recall that $ \Re(z) < - e^{-\betaC+C_1}<0 $,
so $ \Re(\frac12 z-\frac14|p|^2_{2-n}) <0 $, which gives
$
	\normo{ \Go^{(n=1)}_{\eps^2(\frac12z-\frac14|p|^2_{2-n})}  }
	=
	 |\eps^2(\frac12z-\frac14|p|^2_{2-n})|^{-1 }
$.
Under the current assumption $ |\frac12|p|^2_{2-n}-z|>\e^{-2} $,  this is bounded by $ 2 $, so
\begin{align*}
	\Norm{ \phi \tfrac{1}{2} \Go^{(n=1)}_{\eps^2(\frac12z-\frac14|p|^2_{2-n})} \phi } \leq C.
\end{align*}
Since $ \beta^{-1}_\e \to \infty $, \eqref{e.inverse.id} applied to \eqref{e.auxo.choice} with $ \geno = \beta^{-1}_\e \Io $,
show that $  \auxo^{-1}_{\e,p} $ exists with $ 	\normo{ \auxo^{-1}_{\e,p} } \leq C \, (\log \e)^{-1} $,
for all $ \e $ small enough.

Having obtained $ \auxo_{\e,p}^{-1} $ and its bound, we next show the norm convergence.
The condition $|\frac12|p|^2_{2-n}-z| \leq \e^{-2} $ holds for all $ \e \leq C(p) $,
whence we have from~\eqref{e.inverse.id} that 
\begin{align}
	\label{e.I-BGB.d.}
	\auxo^{-1}_{\e,p}
	=
	\Big(
		\beta_\eps\Pro_\perp
		+
		\frac{4\pi}{  \log(\tfrac12|p|^2_{2-n}-z) -\betaCe' } \Pro_\phi
	\Big) 
	\sum_{m=0}^\infty 
	\Big(
		(-\logo_\phi+\smallo_{\e,z,p})\Big(\beta_\eps\Pro_\perp
		+
	\frac{4\pi}{ \log(\tfrac12|p|^2_{2-n}-z) -\betaCe' } \Pro_\phi\Big) 
	\Big)^{m}.
\end{align}
We now take termwise limit in~\eqref{e.I-BGB.d.}. 
Referring to~\eqref{e.smallo}, with $ p\in\R^{2n-2} $ being \emph{fixed},
the linear growth property of $ \smallt(\Cdot) $ in Lemma~\ref{l.Gtwo} gives that
$ \smallo_{\e,z,p} $ converges to $ 0 $ in norm.
Since $ \beta_\e \to 0 $, 
\begin{align*}
	\beta_\eps\Pro_\perp
	+
	\frac{4\pi}{  \log(\tfrac12|p|^2_{2-n}-z) -\betaCe' } \Pro_\phi \longrightarrow \frac{ 4\pi }{ \log(\frac12|p|^2_{2-n}-z) -\betaC-2\betaphi} \Pro_\phi,
	\quad
	\text{in norm}.
\end{align*}
Further, the bound~\eqref{e.genostartez.bd} guarantees that,
for all $ -\Re(z) $ large enough, the series \eqref{e.I-BGB.d.} converges absolutely in norm, uniformly for all $ \e $ small enough.
From this we conclude $ \auxo^{-1}_{\e,p} \to \auxo'_{p} $ in norm, where 
\begin{align}
	\label{e.auxo'tilp}
	\auxo'_{p}
	:=
	\frac{ 4\pi }{ \log(\frac12|p|^2_{2-n}-z) -\betaC'-2\betaphi}  
	\sum_{m=0}^\infty \Pro_\phi \Big(\frac{ 4\pi }{ \log(\frac12|p|^2_{2-n}-z) -\betaC-2\betaphi} (-\logo_\phi)\Pro_\phi \Big)^{m}.
\end{align}
This expression can be further simplified 
using $ \Pro_\phi^m=\Pro_\phi $ and $ \Pro_\phi \logo_\phi \Pro_\phi = \frac{\betaphi}{2\pi} \Pro_\phi $,
\begin{align*}
	\auxo^{'}_{p}
	=
	\frac{ 4\pi }{ \log(\frac12|p|^2_{2-n}-z) -\betaC-2\betaphi}  
	\sum_{m=0}^\infty \Pro_\phi \Big(\frac{ -2\betaphi }{ \log(\frac12|p|^2_{2-n}-z) -\betaC-2\betaphi} \Pro_\phi \Big)^{m}
	=
	\frac{ 4\pi }{ \log(\frac12|p|^2_{2-n}-z) - \betaC} \Pro_\phi.	
\end{align*}
This completes the proof.
\end{proof}

Recall $ \Jo_z $ from~\eqref{e.Jo}.
Combining Lemmas~\ref{l.auxi}--\ref{l.diagonal.fixed} immediately gives the main result of this section:
\begin{lemma}\label{l.diagonal}
There exist constants $ C_1<\infty, C_2(\betaf)>0 $ such that, 
for all $ \Re(z)<-e^{\betaC+C_1} $,
and for all $ \e\in(0,1/C_2(\betaf)) $,
the inverse $ (\beta^{-1}_\e\Io -\phi \, \So_{\e 12}\Go_z\So^*_{\e 12}\phi)^{-1}: \Lsp^2(\R^{2n})\to\Lsp^2(\R^{2n}) $ exists, 
with
\begin{align}
	\label{e.I-BGB.bdd}
	&\Normo{ \big( \beta^{-1}_\e\Io -\phi \, \So_{\e 12}\Go_z\So^*_{\e 12} \phi \big)^{-1} } \leq C \, (\log(-\Re(z))-\betaC)^{-1},
\\	
	\label{e.I-BGB.cnvg}
	&
	\big( \beta^{-1}_\e\Io -\phi \, \So_{\e 12}\Go_z\So^*_{\e 12} \phi \big)^{-1}
	\longrightarrow
	4\pi \phi\otimes \big( (\Jo_{z} - \betaC\Io)^{-1}\Po \big),
	\
	\text{strongly, as } \e\to 0.
\end{align}
\end{lemma}

\section{Convergence of the resolvent}
\label{sect.resolcnvg}

In this section we collect the results of 
Sections~\ref{sect.resolvent1}--\ref{sect.diagonal}
to
prove  Proposition~\ref{p.resolvent}\ref{p.resolvent.}--\ref{p.resolvent.sym} and Theorem~\ref{t.resolcnvg}\ref{t.resolcnvg.e}--\ref{t.resolcnvg.cnvg} and the convergence part of Theorem~\ref{t.main}\ref{t.main.dcnvg}.

Proposition~\ref{p.resolvent}\ref{p.resolvent.} and Theorem~\ref{t.resolcnvg}\ref{t.resolcnvg.e}
follow from the bounds obtained in Lemmas~\ref{l.Goin}--\ref{l.BGconvg}, \ref{l.off.diagonal.bdd}, and~\ref{l.diagonal}.
We now turn to Theorem~\ref{t.resolcnvg}\ref{t.resolcnvg.cnvg}.
Recall that Lemma~\ref{l.resolvent.e}, as stated, applies only  for  $ \Re(z)<-C(n,\e) $, for some threshold $ C(n,\e) $ that depends on $ \e $.
Here we argue that the threshold can be improved to be independent of $ \e $.
%
Using the bounds from Lemmas~\ref{l.Goin}--\ref{l.BGconvg}, \ref{l.off.diagonal.bdd}, and \ref{l.diagonal} on the r.h.s.\ of \eqref{e.resolvent.e},
we see that $ (\Ro_{\e,z}-\Go_z) $ defines an analytic function (in operator norm) in $ B:=\{\Re(z)<-e^{Cn^2+\betaC}\} $.
On the other hand, we also know that $ (\Ro_{\e,z}-\Go_z) $ is analytic in $ z $ off $ \sigma(\Ho_{\e})\cup [0,\infty) $, where $ \sigma(\Ho_{\e})\subset \R $ denotes the spectrum of $ \Ho_{\e} $.
Consequently, both sides must match on $ B\setminus\sigma(\Ho_{\e}) $.
We now argue $ B\cap\sigma(\Ho_{\e}) = \varnothing $, so the matching actually holds on the entire $ B $.
Assuming the contrary, we fix $ z_0\in B\cap \sigma(\Ho_{\e}) $, take a sequence $ z_k\in B $ and approaches $ z_k \to z_0 $ along the vertical axis.
Along this sequence $ (\Ro_{\e,z_k}-\Go_{z_k}) $ is bounded,  contradicting $ z_0\in\sigma(\Ho_{\e}) $.

We now show the convergence of the resolvent, i.e.~\eqref{e.resolvent.e} to~\eqref{e.resolvent.}.
As argued previously, both series \eqref{e.resolvent.e} and \eqref{e.resolvent.} converge absolutely in operator norm, uniformly over $ \e $.
It hence suffices to show termwise convergence.
By Lemmas~\ref{l.BGconvg}, \ref{l.Gobi.cnvg}, and \ref{l.diagonal},
each factor in~\eqref{e.resolvent.e.out}--\eqref{e.resolvent.e.in} converges to its limiting counterparts 
in~\eqref{e.resolvent.out}--\eqref{e.resolvent.in}, strongly or in norm.
Using this in conjunction%
with the elementary, readily checked fact
\begin{align*}
	\geno_\e\geno'_\e \rightarrow \geno \geno' \text{ strongly if }
	\geno_\e, \geno'_\e \text{ are uniformly bounded and } \geno_\e \to \geno, \geno'_\e \to \geno' \text{ strongly,}
\end{align*}
we conclude the desired convergence of the resolvent, Theorem~\ref{t.resolcnvg}\ref{t.resolcnvg.cnvg}.

Next we prove Proposition~\ref{p.resolvent}\ref{p.resolvent.sym}.
First, given the bounds from Lemmas~\ref{l.Goin}--\ref{l.BGconvg}, \ref{l.off.diagonal.bdd}, and~\ref{l.diagonal},
we see that $ \Ro^\sym_z $ in~\eqref{e.resolvent.sym} defines a bounded operator on $ \Lsp^2(\R^{2n}) $ for all $ \Re(z)<-e^{\betaC+n^2C} $.
Our goal is to match $ \Ro^\sym_z $ to $ \Ro_z $ on $ \Lsp^2_\sym(\R^{2n}) $, for these values of $ z $.
Apply~\eqref{e.inverse.id} with $ \geno = \frac{1}{4\pi}(\Jo_z-\betaC\Io) $ and with $ \tilde{\geno} = \frac{2}{n(n-1)} \dsum \So_{ij}\Go_z\So^*_{k\ell} $
for the prescribed values of $ z $ (so that the condition for \eqref{e.inverse.id} to apply checks).
We obtain
\begin{align}
\label{e.resolvent.expansion**}
	\Ro^\sym_{z}
	=&
	\Go_{z} 
	+ 
	\sum
	\Go_z\So^*_{i_1j_1} \, 
	\Big( 4\pi (\Jo_z-\betaC\Io)^{-1} \, \prod_{s=2}^{m} \Big( \tfrac{2}{n(n-1)} \So_{k_sk_s}\Go_z\So^*_{i_sj_s} \, 4\pi (\Jo_z-\betaC\Io)^{-1} \Big) \Big) 
	\, \tfrac{2}{n(n-1)}\So_{k_{m+1}k_{m+1}}\Go_z,
\end{align}
where the sum is over all pairs $ (i_1<j_1) $, $ (k_2<\ell_2)\neq(i_2<j_2) $, \ldots, $ (k_{m}<\ell_{m})\neq(i_{m}<j_{m}) $, $ (k_{m+1}<k_{m+1}) $, and all $m$.

At this point we need to use the symmetry of $ \Lsp^2_\sym(\R^{2n}) $. Let %
\begin{align*}
	\Lsp^2_{\sym'}(\R^{2n-2})
	:=
	\big\{ v\in\Lsp^2(\R^{2n-2}): v(y_2,y_{\sigma(3)},\ldots,y_{\sigma(n)})= v(y_2,y_3,\ldots,y_n) \ \sigma \in \mathbb{S}_{n-2} \big\}
\end{align*}
denote the space of functions on $ \R^{2n-2} $ that are symmetric in the last $ (n-2) $ components.
It is readily checked that the incoming operator (i.e., $ \So_{k_{m+1}k_{m+1}}\Go_z $)
maps $ \Lsp^2_\sym(\R^{2n}) $ into $ \Lsp^2_{\sym'}(\R^{2n-2}) $,
that the mediating operators (i.e., $ \So_{k_sk_s}\Go_z\So^*_{i_sj_s} $ and $ 4\pi (\Jo_z-\betaC\Io)^{-1} $)
map $ \Lsp^2_{\sym'}(\R^{2n-2}) $ to $ \Lsp^2_{\sym'}(\R^{2n-2}) $.
Further, given that $ \Go_z $ acts symmetrically in the $ n $ components, we have
\begin{align}
	\label{e.GS.symmetry}
	&\So_{ij}\Go_z\big|_{\Lsp^2_\sym(\R^{2n})} = \So_{i'j'}\Go_z \big|_{\Lsp^2_\sym(\R^{2n})},
	\qquad
	\forall \, (i<j), \, (i'<j').
\end{align}
Also, from~\eqref{e.GbiQ.Four} we have
\begin{align}
	&\So_{k\ell}\Go_z\So^*_{ij}\big|_{\Lsp^2_{\sym'}(\R^{2n-2})} = \So_{\sigma(k)\sigma(\ell)}\Go_z\So^*_{\sigma(i)\sigma(j)}\big|_{\Lsp^2_{\sym'}(\R^{2n-2})},
	\label{e.SGS.symmetry}
	\qquad
	\forall \, (i<j) \neq (k<\ell),
	\
	\sigma \in \mathbb{S}_n.
\end{align}%

In~\eqref{e.resolvent.expansion**}, use \eqref{e.SGS.symmetry} to rearrange the sum over $ (k_2<\ell_2)\neq(i_2<j_2) $ as 
\begin{align*}
	\frac{2}{n(n-1)} \sum_{(k_2<\ell_2)\neq(i_2<j_2) } \So_{k_2\ell_2}\Go_z\So^*_{i_2j_2} \big|_{\Lsp^2_{\sym'}(\R^{2n-2})}
	=
	\sum_{i_2<j_2} \So_{i_1j_1}\Go_z\So^*_{i_2j_2} \big|_{\Lsp^2_{\sym'}(\R^{2n-2})} \ind_\set{(i_2<j_2) \neq (i_1<i_1)}.
\end{align*}
That is, we use \eqref{e.SGS.symmetry} for some $ \sigma\in\mathbb{S}_n $ such that $ (\sigma(k_2)<\sigma(k_2)) = (i_1<j_1) $.
Doing so reduces the sum over double pairs $ (k_2<\ell_2)\neq (i_2<j_2) $ into a sum over a single pair $ (i_2<j_2) $ with $ (i_2<j_2) \neq (i_1<j_1) $,
and the counting in this reduction cancels the prefactor $ 2/(n(n-1)) $.
Continue this procedure inductively from $ s=2 $ through $ s=m $,
and then, at the $ m+1 $ step, similarly use~\eqref{e.GS.symmetry} to write
\begin{align*}
	\frac{2}{n(n-1)} \sum_{k_{m+1}<\ell_{m+1} } \So_{k_{m+1}\ell_{m+1}}\Go_z \big|_{\Lsp^2_\sym(\R^{2n})} = \So_{i_{m}j_{m}}\Go_z \big|_{\Lsp^2_\sym(\R^{2n})}.
\end{align*}
We then conclude Proposition~\ref{p.resolvent}\ref{p.resolvent.sym}, 
\begin{equation} \Ro^\sym_z|_{\Lsp^2_\sym(\R^{2n})} = \Ro_z|_{\Lsp^2_\sym(\R^{2n})} .
\end{equation}

We now turn to the  convergence of the fixed time correlation functions in Theorem~\ref{t.main}\ref{t.main.dcnvg}.  Given Theorem~\ref{t.resolcnvg},
applying the Trotter--Kato Theorem, c.f., \cite[Theorem~VIII.22]{reed72}, 
we know that there exists an (unbounded) self-adjoint operator $ \Ho $ on $ \Lsp^2(\R^{2n}) $,
such that $ \Ro_z $ (in~\eqref{e.resolvent}) is the resolvent for $ \Ho $, i.e., $ \Ro_z =(\Ho-z\Io)^{-1} $, for all $ \Im(z)\neq 0 $.
Theorem~\ref{t.resolcnvg} also guarantees that the spectra of $ \Ho_\e $ and $ \Ho $ are bounded below, uniformly in $ \e $.
More precisely, $ \sigma(\Ho_\e), \sigma(\Ho) \subset (-C_1(n,\betaC),\infty) $, for all $ \e\in(0,1/C_2(\betaf)) $, for some $ C_1(n,\betaC)<\infty $ and $ C_2(\betaf)>0 $.
Fix $ t\in\R_+ $. We now apply \cite[Theorem~VIII.20]{reed72}, which says that if self-adjoint operators $ \Ho_\e \to \Ho $ in the strong resolvent sense, and $f$ is bounded and continuous on $\R$ then $f(\Ho_\e) \to f(\Ho )$ strongly.  
We use $ f(\lambda) = e^{(-t\lambda)\wedge C_1(n,\betaC)} $,
which is bounded and continuous, and from what we have proved, $ f(\Ho_\e)=e^{-t\Ho_\e} $ and $ f(\Ho)=e^{-t\Ho} $.
Hence
\begin{align}
	\label{e.sg.strong}
	e^{-t\Ho_\e} \longrightarrow e^{-t\Ho}
	\quad
	\text{strongly on } \Lsp^2(\R^{2n}), \text{ for each fixed } t\in\R_+.
\end{align}
For Theorem \ref{t.main}\ref{t.main.dcnvg},
we wish to upgrade this convergence to be \emph{uniform} over finite intervals in $ t $.
Given the lower bound on the spectra, we have the uniform (in $ \e $) norm continuity:
\begin{align*}
	\Normo{ e^{-t\Ho_\e} -e^{-s\Ho_\e} } + \Normo{ e^{-t\Ho} -e^{-s\Ho} }
	\leq
	C_2(n,\betaC)|t-s|e^{C_2(n,\beta)(t\vee s)},
\end{align*}
for all $ \e\in(0,1/C_2(\betaf)) $ and $ s,t\in[0,\infty) $.
This together with~\eqref{e.sg.strong} gives
\begin{align*}
	\lim_{\e\to 0}\sup_{t\in[0,\tau]} \Norm{ e^{-t\Ho_\e} u - e^{-t\Ho} u } = 0, 
	\quad
	u \in \Lsp^2(\R^{2n}),
	\
	\tau<\infty.
\end{align*}
Comparing this with~\eqref{e.mom.semigroup}, we now have, for each fixed $ g\in\Lsp^2(\R^{2n}) $,
\begin{align}
	\label{e.semigroup.cnvg}
	\E\big[ \ip{ Z_{\e,t}^{\otimes n}}{g} \big] 
	\longrightarrow
	\Ip{ Z_\ic^{\otimes n} }{ e^{-t\Ho} g },
	\quad
	\text{uniformly over finite intervals in }t.
\end{align}
What is missing for the proof of Theorem~\ref{t.main}  is the identification of the semigroup $ e^{-t\Ho} $ with the explicit operators defined in \eqref{e.Diagram.op}, \eqref{e.diagram.op}.  This is the subject of the next section.

\section{Identification of the limiting semigroup}
\label{sect.laplace}

The remaining task is to match $ e^{-t\Ho} $ to the operator $ \Pt_t+\Dio^{\diag(n)}_t $ on r.h.s.\ of~\eqref{e.mom.cnvg}.
To rigorously perform the heuristics in Remark \ref{remheur}, 
it is more convenient to operate in the forward Laplace transform, i.e., going from $ t $ to $ z $.
Doing so requires
establishing bounds on the relevant operators in \eqref{e.diagram.op}, 
and verifying the semigroup property of $ \Pt_t + \Dio_t^{\diag(n)} $, defined in~\eqref{e.Diagram.op}.
The bounds will be established in Section~\ref{sect.t,op.bd},
and, as the major step toward verifying the semigroup property, we establish an identity in Section~\ref{sect.semigroup}.


\subsection{Bounds and Laplace transforms}
\label{sect.t,op.bd}
We begin with the incoming and outgoing operators.
We now establish a quantitative bound on the norms of $\So_{ij}\Pt_t$ and $\Pt_t\So_{ij}^*$ , and  match them to the corresponding Laplace transform.

\begin{lemma}\label{l.Ptin.out}
\begin{enumerate}[label=(\alph*)]
\item[]
\item\label{l.Ptin.out.} 
For each pair $ i<j $ and $ t\in\R_+ $, $ \So_{ij}\Pt_t: \Lsp^2(\R^{2n})\to\Lsp^2(\R^{2n-2}) $ 
and $ \Pt_t\So^*_{ij}: \Lsp^2(\R^{2n-2})\to\Lsp^2(\R^{2n}) $
are bounded with
\begin{align*}
	\normo{\So_{ij}\Pt_{t}}+\normo{\Pt_{t}\So^*_{ij}} \leq C t^{-1/2}.
\end{align*}
\item \label{l.Ptin.out.Lap}
For each pair $ i<j $, $ \Re(z)<0 $, $ u\in\Lsp^2(\R^{2n}) $, and $ v\in\Lsp^2(\R^{2n-2}) $, 
\begin{align*}
	\int_{\R_{+}} e^{tz} \ip{v}{\So_{ij}\Pt_t u} \, \d t 
	=
	\int_{\R_{+}\times\R^{4n-2}} e^{tz} \bar{v(y)} \hk(t,S_{ij}y-x)  u(x) \, \d t \d y \d x  
	= 
	\ip{u}{\So_{ij}\Go_zv},
\\
	\int_{\R_{+}} e^{tz} \ip{u}{\Pt_t\So^*_{ij}v} \, \d t 
	=
	\int_{\R_{+}\times\R^{4n-2}} e^{tz} \bar{u(x)} \hk(t,x-S_{ij}y)  v(y) \, \d t \d x \d y  
	= 
	\ip{u}{\Go_z\So^*_{ij}v},
\end{align*}
where the integrals converge absolutely (over $ \R_+ $ and over $ \R_+\times\R^{2n-4} $).
\end{enumerate}
\end{lemma}
\begin{proof}
It suffices to consider $ \So_{ij}\Pt_t $ since $ \Pt_t\So^*_{ij}=(\So_{ij}\Pt_t)^* $.

\noindent\ref{l.Ptin.out.}~
Fix $ u\in\Lsp^2(\R^{2n}) $, we use~\eqref{e.S.Four} to bound
\begin{align*}
	\norm{\So_{ij}\Pt_tu}^2
	=
	\int_{\R^{2n-2}} \Big( \int_{\R^{2}}  \hat{\Pt_t u}(\M_{ij}q)  \frac{\d q_1}{2\pi} \Big)^2 \d q_{2-n}
	=
	\int_{\R^{2n-2}} \Big( \int_{\R^{2}} e^{-\frac12 t |\M_{ij}q|^2} \hat{u}(\M_{ij}q) \frac{\d q_1}{2\pi} \Big)^2 \d q_{2-n}.
\end{align*}
On the r.h.s., bound $ |\M_{ij}q|^2 \geq \frac12 |q_1|^2 $ (as checked from~\eqref{e.M}),
and applying the Cauchy--Schwarz inequality in the $ q_1 $ integral.
We conclude the desired result
\begin{align*}
	\norm{\So_{ij}\Pt_tu}^2
	\leq
	C\int_{\R^{2}} \Big( e^{- \frac14 t|q_1|^2} \Big)^2 \d q_1 \ \norm{u}^2
	\leq 
	\frac{C}{t} \norm{u}^2.
\end{align*}

\noindent\ref{l.Ptin.out.Lap}~
Fix $ \Re(z)<0 $, integrate $ \ip{v}{\So_{ij}\Pt_t u} $ against $ e^{zt} $ over $ t\in(0,\infty) $, and use~\eqref{e.S.Four} to get
\begin{align*}
	\int_0^\infty e^{zt} \ip{v}{\So_{ij}\Pt_{t}u} \, \d t 
	=
	\int_0^\infty \int_{\R^{2n}} \bar{\hat{v}(q_{2-n})} e^{tz-\frac{t}2|\M_{ij}q|^2} \hat{u}(\M_{ij}q)(2\pi)^{-1} \, \d t \d q.
\end{align*}
This integral converges absolutely since $ \normo{\So_{ij}\Pt_t} \leq C t^{-1/2} $ and $ \Re(z)<0 $.
This being the case, we swap the integrals and evaluate the integral over $ t $ to get
\begin{align*}
	\int_0^\infty e^{zt}  \ip{v}{\So_{ij}\Pt_{t}u} \, \d t 
	=
	\int_{\R^{2n}} \bar{\hat{v}(q_{2-n})} \frac{1}{\tfrac12|\M_{ij}q|^2-z} \hat{u}(\M_{ij}q) \, \frac{\d q}{2\pi}.
\end{align*}
The last expression matches $ \ip{v}{\So_{ij}\Go_z u} $, as seen from~\eqref{e.SG.fourier}.
\end{proof}

\begin{lemma}\label{l.Ptbi}
\begin{enumerate}[label=(\alph*)]
\item[]
\item \label{l.Ptbi.}
For distinct pairs $ (i<j)\neq (k<\ell) $, $ t\in\R_+ $, $ \Pt_t\So_{k\ell}^*(\Lsp^2(\R^{2n-2})) \subset \Dom(\So_{ij}) $,
so the operator $ \So_{ij}\Pt_t \So_{k\ell}^* $ maps $ \Lsp^2(\R^{2n-2}) \to \Lsp^2(\R^{2n-2}) $.
Further 
\begin{align*}
	\normo{\So_{ij}\Pt_t\So_{k\ell}^*} \leq C t^{-1}.
\end{align*}
\item \label{l.Ptbi.Lap}
For distinct pairs $ (i<j)\neq (k<\ell) $, $ v,w\in\Lsp^2(\R^{2n-2}) $, and $ \Re(z)<0 $,
\begin{align}
	\label{e.Ptbi.Lap}
	\int_{\R_+\times\R^{4n-4}} e^{zt} \bar{w}(y) \hk(t,S_{ij}y-S_{k\ell}y') v(y') \, \d t \d y \d y'
	=
	\ip{ w }{ \So_{ij}\Go_z\So^*_{k\ell} v},
\end{align}
where the integral converges absolutely.
\end{enumerate}
\end{lemma}
\begin{remark}
Unlike in the case for incoming and outgoing operators,
here our bound on $ C t^{-1} $ on the mediating operator
does not ensure the integrability of $ \normo{\So_{ij}\Pt_t\So_{k\ell}^*} $ near $ t=0 $.
Nevertheless, the integral in~\eqref{e.Ptbi.Lap} still converges absolutely.
\end{remark}
\begin{proof}
Fix distinct pairs $ (i<j)\neq (k<\ell) $ and $ v,w\in\Lsp(\R^{2n-2}) $.

\noindent\ref{l.Ptbi.}~
As argued just before Lemma~\ref{l.Ptin.out}, we have $ \Pt_t\So^*_{k\ell}v \in \Lsp^2(\R^{2n}) $.
To check the condition  $ \Pt_t\So^*_{k\ell}v \in \Dom(\So_{ij}) $, consider
\begin{align}
	\label{e.dell.conseq.}
	\int_{\R^{2n}} \Big| \hat{w}(q_{2-n}) 
	e^{-\frac{t}{2}|\M_{ij}q|^2}\hat{\So_{k\ell}^*v}(M_{ij}q) \Big| \, \frac{ \d q}{2\pi}
	=
	\int_{\R^{2n}} \Big| \hat{w}(p_i+p_j,p_{\bar{ij}}) e^{-\frac{t}{2}|p|^2}\hat{\So_{k\ell}^*v}(p) \Big| \, \frac{\d p}{2\pi},
\end{align}
where the equality follows by a change of variable $ q=M_{ij}^{-1}p $,
together with $ (p_i+p_j,p_{\bar{ij}})=[\M_{ij}^{-1}p]_{2-n} $ and $ |\det(M_{ij})|=1 $ (as readily verified from~\eqref{e.M}).
In~\eqref{e.dell.conseq.}, bound $ e^{-\frac{t}{2}|p|^2} \leq C\,(t|p|^2)^{-1} $ and use~\eqref{e.dell} to get
\begin{align}
	\label{e.dell.conseq..}
	\eqref{e.dell.conseq.}
	\leq
	C\, t^{-1} \norm{v} \, \norm{w}.
\end{align}
Referring to the definition~\eqref{e.domSo} of $ \Dom(\So_{ij}) $,
since~\eqref{e.dell.conseq} holds for all $ w\in\Lsp^2(\R^{2n-2}) $,
we conclude $ \Pt_t\So^*_{k\ell}v \in \Dom(\So_{ij}) $
and $ |\ip{ w}{ \So_{ij}\Pt_t\So^*_{k\ell}v }| = |\ip{ \So^*_{ij}w }{ \Pt_t\So^*_{k\ell}v}| \leq C t^{-1} \norm{w}\,\norm{v} $.

\ref{l.Ptin.out.Lap}
To prove~\eqref{e.Ptbi.Lap}, 
assume for a moment $ z=-\lambda \in(-\infty,0) $ is real, and $ v(y),w(y) \geq 0 $ are positive.
In~\eqref{e.Ptbi.Lap}, express the integral over $ y,y' $ as $ \ip{w}{\So_{ij}\Pt_t\So^*_{k\ell}v} = \ip{\So^*_{ij}w}{\Pt_t\So_{k\ell}^*v} $, and use~\eqref{e.S*ijS*kl} to get
\begin{align*}
	\int_{\R_+\times\R^{4n-4}} e^{zt} \bar{w}(y) \hk(t,S_{ij}y-S_{k\ell}y') v(y') \, \d t \d y \d y'
	=
	\int_0^\infty e^{-\lambda t} \Big( \int_{\R^{2n}} \bar{\hat{w}(p_i+p_j,p_{\bar{ij}})} e^{-\frac{t}{2}|p|^2} \hat{v}(p_k+p_\ell,p_{\bar{k\ell}}) \, \d p \Big) \, \d t.
\end{align*}
The integral on the r.h.s.\ converges absolutely over $ \R_+\times\R^{2n} $, i.e., jointly in $ t,p $.
This follows by using~\eqref{e.dell} together with $ \int_0^\infty e^{-\lambda t -\frac{t}{2}|p|^2} \d t = \frac{1}{\lambda+\frac12|p|^2} $.
Given the absolute convergence, we swap the integrals over $ t $ and over $ p $, 
and evaluate the former to get
the expression for $ \ip{w}{\So_{ij}\Go_z\So^*_{k\ell}v} $ on the right hand side of \eqref{e.GbiQ.Four}.
For general $ v(y),w(y) $, the preceding calculation done for $ (v(y),w(y)) \mapsto (|v(y)|,|w(y)|) $ and for $ z\mapsto \Re(z) $
guarantees the relevant integrability.
\end{proof}

Recall $ \jfn(t,\betaC) $ from~\eqref{e.jfn}.
For the diagonal mediating operator, let us first settle some properties of $ \jfn $.
\begin{lemma}\label{l.jfn}
For each $ \Re(z)<-e^{\betaC} $, the Laplace transform of $ \jfn(t,\betaC) $ evaluates to
\begin{align}
	\label{e.jfn.Lap}
	\int_0^\infty e^{zt} \jfn(t,\betaC) \, \d t = \frac{1}{\log(-z) - \betaC},
\end{align}
where the integral converges absolutely, and
$ \jfn(t,\betaC) $ has the following pointwise bound
\begin{align}
	\label{e.jfn.ptwise}
	\jfn(t,\betaC)
	=
	|\jfn(t,\betaC)| \leq C \, t^{-1} \, |\log (t\wedge\tfrac12)|^{-2} e^{(\betaC+1)C t},
	\quad
	t\in\R_+.
\end{align}
\end{lemma}
\begin{proof}
To evaluate the Laplace transform, 
assume for a moment that $ z\in(-\infty,-e^{\betaC}) $ is real.
Integrate~\eqref{e.jfn} against $ e^{zt} $ over $ t $.
Under the current assumption that $ z $ is real, the integrand therein is positive,
so we apply Fubini's theorem to swap the $ t $ and $ \alpha $ integrals to get
\begin{align*}
	\int_0^\infty e^{z t} \jfn(t,\betaC) \, \d t 
	=
	\int_0^\infty \frac{e^{\betaC\alpha}}{\Gamma(\alpha)} \Big( \int_0^\infty t^{\alpha-1} e^{-(-z t)} \d t \Big) \d \alpha.
\end{align*}
The integral over $ t $, upon a change of variable $ -zt \mapsto t $, evaluates to $ \Gamma(\alpha) /(-z)^{\alpha} $.
Canceling the $ \Gamma(\alpha) $ factors and evaluating the remaining integral over $ \alpha $ yields \eqref{e.jfn.Lap} for $ z\in(-\infty,-e^{\betaC}) $.
For general $ z\in\C $ with $ \Re(z)<-e^{\betaC} $, since $ |e^{zt}| = e^{\Re(z)t} $, the preceding result guarantees integrability
of $ |e^{-zt+\alpha\betaC} t^{\alpha-1} \Gamma(\alpha)^{-1}| $ over $ (t,\alpha) \in \R^2_+ $.
Hence Fubini's theorem still applies, and~\eqref{e.jfn.Lap} follows.

To show~\eqref{e.jfn.ptwise}, in~\eqref{e.jfn}, we separate the integral (over $ \alpha\in\R_+ $) 
into two integrals over $ \alpha>1 $ and over $ \alpha<1 $,
denoted by $ I_{+} $ and $ I_{-} $, respectively.
For $ I_{+} $, we use the bound $ \exp(-\log \Gamma(\alpha)) \leq \frac{\alpha}2\log\alpha - C\alpha $
(c.f., \cite[6.1.40]{abramowitz65}) to write
$
	I_{+}
	\leq 
	\int_1^\infty \exp( -\alpha(\tfrac12\log\alpha - (C+\betaC) -\log t) ) \d \alpha $.
It is now straightforward to check that $ I_{+} \leq e^{(\betaC+1)Ct} $.
Using $ |\frac{1}{\Gamma(\alpha)}| \leq C\alpha $, $ \alpha\in(0,1) $ (c.f., \cite[6.1.34]{abramowitz65}),
we bound $ I_{-} $ as
$
	I_{-}
	\leq 
	Ct^{-1}e^{\betaC} \int_0^1 \alpha t^{\alpha} \d \alpha$.
For all $ t \geq \frac12 $, the last integral is indeed bounded by $ e^{(\betaC+1)Ct} $.
For $ t < \frac12 $, we write $ t^{\alpha}=e^{-\alpha|\log t|} $ we perform a change of variable $ \alpha|\log t| \to t $ to get
$
	I_{-}
	\leq
	C \, t^{-1} e^{\betaC} |\log t|^{-2} \int_0^{|\log t|} \alpha e^{-\alpha} \d \alpha
	\leq
	C \, t^{-1} e^{\betaC} |\log t|^{-2}$.
Collecting the preceding bounds and adjusting the constant $C$ gives~\eqref{e.jfn.ptwise}.
\end{proof}

Referring to the definition~\eqref{e.Ptj} of $ \Ptj{t} $, we see that 
this operator has an integral kernel
\begin{align}
	\label{e.Ptj.ker}
	\big( \Ptj{t} v \big)(y) 
	= 
	\int_{\R^{2n-2}} \hk^\Jo(t,y,y') v(y') \, \d y',
	\quad
	\hk^\Jo(t,y,y') := \jfn(t,\betaC) \hktwo(\tfrac{t}{2},y_2-y'_2) \prod_{i=3}^n \hktwo(t,y_i-y'_i).
\end{align}
\begin{lemma}\label{l.Ptj}
\begin{enumerate}[label=(\alph*)]
\item[]
\item \label{l.Ptj.}
For each $ t\in\R_+ $, $ \Ptj{t}: \Lsp^2(\R^{2n-2})\to\Lsp^2(\R^{2n-2}) $ is a bounded operator with
\begin{align}
	\normo{\Ptj{t}} \leq C\,(t\wedge\tfrac12)^{-1} \, |\log (t\wedge\tfrac12)|^{-2} e^{(\betaC+1)C t}.
\end{align}
\item \label{l.Ptj.Lap}
Further, for each $ v,w\in\Lsp^2(\R^{2n-2}) $ and $ \Re(z)<-e^{\betaC} $,
\begin{align}
	\label{e.Ptj.Lap}
	\int_{\R_+}  e^{zt} \,\ip{w}{\Pt^\Jo_t v} \,  \d t 
	=	
	\int_{\R_+\times\R^{4n-4}}  e^{zt} \,\bar{w(y)} \hk^\Jo(t,y,y') v(y') \,  \d t \d y \d y'
	=
	\ip{ w }{ (\Jo_z-\betaC\Io)^{-1} v},
\end{align}
where the integrals converge absolutely (over $ \R_+ $ and over $\R_+\times\R^{4n-4} $).
\end{enumerate}
\end{lemma}
\begin{proof}
Part~\ref{l.Ptj.} follows from~\eqref{e.jfn.ptwise} and the fact that 
heat semigroups have unit norm, i.e., $ \normo{e^{-at\nabla_i^2}} = 1 $, $ a\geq 0 $.
For part~\ref{l.Ptj.Lap}, we work in Fourier domain and write
\begin{align*}
	\int_{\R^{4n-4}} \bar{w(y)} \hk^\Jo(t,y,y') v(y') \,  \d t \d y \d y'
	=
	\jfn(t,\betaC)\int_{\R^{2n-2}}  \bar{\hat{w}(p)} e^{-\frac12 t|p|^2_{2-n}} \hat{v}(p)  \d p,
\end{align*}
where, recall that $ |p|^2_{2-n}=\frac12|p_2|^2 + |p_3|^2 +\ldots +|p_n|^2 $.
Integrate both sides against $ e^{zt} $ over $ t\in\R_+ $, and exchange the integrals over $ p $ and over $ t $. 
The swap of integrals are justified the same way as in the proof of Lemma~\ref{l.Ptbi}, so we do not repeat it here.
We now have
\begin{align*}
	\int_{\R_+\times\R^{4n-4}} e^{zt}\bar{w(y)} \hk^\Jo(t,y,y') v(y') \,  \d t \d y \d y'
	=
	\int_{\R^{2n-2}}  \Big(\int_0^\infty e^{zt-\frac12 t|p|^2_{2-n}} \jfn(t,\betaC) \d t\Big) \bar{\hat{w}(p)} \hat{v}(p)  \d p.
\end{align*}
Applying~\eqref{e.jfn.Lap}  to evaluate the integral over $ t $ yields
the expression in~\eqref{e.Jo} for $ \ip{w}{(\Jo_z-\betaC\Io)^{-1}v} $.
\end{proof}

\subsection{An identity for the semigroup property}
\label{sect.semigroup}

Our goal is to prove Lemma~\ref{l.semigroup} in the following.
Key to the proof is the  identity \eqref{e.the.identity}.  It depends on a cute fact about the $\Gamma$ function.
Set 
\begin{align}
	\label{e.pk}
	p_{k}(\alpha):=  \frac{\Gamma(\alpha+k+1)}{\Gamma(\alpha+1)} = (\alpha+k)\cdots(\alpha+1)\alpha,
	\quad
	\alpha \geq 0.
\end{align}
with the convention $ p_{-1} := 1 $.

\begin{lemma}
\label{l.Gamma}
For $ m\in\Z_{\geq 0} $,
\begin{align*}
	p_m(\alpha)
	=
	\int_{0}^\alpha
	\sum_{k=0}^{m}
	\binom{m+1}{m-k+1} (m-k) ! \,
	p_{k-1}(\alpha_1) \, \d \alpha_1.
\end{align*}
\end{lemma}

\begin{proof}
Taking derivative gives $ \frac{\d~}{\d \alpha} p_m(\alpha) = \sum_{j=0}^m \prod_{j^\mathrm{c}}^m(\alpha+i) $,
where $ \prod_{j^\mathrm{c}}^m $ denotes a product over $ i\in\{0,\ldots,m\}\setminus\{j\} $.
Our goal is to express this derivative in terms of $ p_{m-1}(\alpha), p_{m-2}(\alpha),\ldots $.
The $ j=m $ term skips the $ (\alpha+m) $ factor, and is hence exactly $ p_{m-1}(\alpha) $.
For other values of $ j $, we use $ (\alpha-m) $ to compensate the missing $ (\alpha+j) $ factor.
Namely, writing $ (\alpha+m) = (\alpha+j+(m-j)) $, we have
\begin{align}
	\label{e.mtom-1}
	\prod\nolimits^{m}_{j^\mathrm{c}}(\alpha+i) = p_{m-1}(\alpha) + (m-j) \prod\nolimits^{m-1}_{j^\mathrm{c}} (\alpha+i). 
\end{align}
This gives
\begin{align*}
	\frac{\d~}{\d \alpha} p_m(\alpha)
	=
	\sum_{j=0}^m \prod\nolimits_{j^\mathrm{c}}^m(\alpha+i) 
	=
	\sum_{j=0}^m p_{m-1}(\alpha) 
	+ 
	\sum_{j=0}^m (m-j) \prod\nolimits^{m-1}_{j^\mathrm{c}} (\alpha+i). 
\end{align*}
In~\eqref{e.mtom-1}, we have reduced $ \prod_{j^\mathrm{c}}^m(\alpha+i) $ to $ \prod_{j^\mathrm{c}}^{m-1}(\alpha+i) $,
i.e., the same expression but with $ m $ decreased by $ 1 $.
Repeating this procedure yields
\begin{align}
\label{e.p'm}
\begin{split}
	\frac{\d~}{\d \alpha} p_m(\alpha)
	&=
	\sum_{\ell=1}^m
	p_{m-\ell}(\alpha) \Big( \sum_{j=0}^m (m-j)_+(m-j-1)_+\cdots (m-j-\ell)_+ \Big)
\\
	&=
	\sum_{\ell=1}^m
	p_{m-\ell-1}(\alpha) \sum_{j=0}^m \prod_{i=0}^{\ell-1} (j-i)_+
	=
	\sum_{\ell=1}^m p_{m-\ell-1}(\alpha)\binom{m+1}{\ell+1} \ell ! \, ,
\end{split}
\end{align}
where $ \prod_{i\in\varnothing}(\Cdot) := 1 $.
Within the last equality, we have used the identity
$
	\sum_{j=0}^m \prod_{i=0}^{\ell-1} (j-i)_+ = \binom{m+1}{\ell+1} \ell !.
$
In~\eqref{e.p'm}, perform a change of variable $ m-\ell := k $, and integrate in $ \alpha $, using $ p_m(0)=0 $ to get the result.
\end{proof}

\begin{lemma}
For $ s<t \in \R_+ $, $ i<j $, we have
\begin{align}
	\label{e.the.identity}
	\jfn(t,\betaC) = \int_{0<t_1<s} \int_{s<t_2<t} \jfn(t_1,\betaC) (t_2-t_1)^{-1} \jfn(t-t_2,\betaC) \, \d t_1 \d t_2.
\end{align}
\end{lemma}
\begin{proof}
Write $ \jfn(t,\betaC)=\jfn(t) $ to simplify notation.
Let the r.h.s.\ of~\eqref{e.the.identity} be denoted by $ F(s,t) $.
It is standard to check that $ F(s,t) $ is continuous on $ 0 < s < t<\infty $.
Hence it suffices to show
\begin{align}
	\label{e.id.goal.}
	\int_0^t F(s,t) s^{m} \, \d s = \jfn(t) \int_0^t s^{m} \, \d s = \jfn(t) \, (m+1)^{-1}t^{m+1},
	\quad
	m \in \Z_{\geq 0}.
\end{align}
From~\eqref{e.jfn.ptwise}, it is readily checked both sides of~\eqref{e.id.goal.}
grow at most exponentially  in $ t$.
Taking Laplace transform on both sides of~\eqref{e.jfn.ptwise}, the problem is further reduced to showing, for some $ C(m,\betaC)<\infty $,
\begin{align}
	\label{e.id.goal}
	\int_0^\infty \int_0^t e^{-\lambda t} F(s,t) s^{k} \, \d t \d s 
	= 
	\int_0^\infty e^{-\lambda t}\jfn(t) \, (m+1)^{-1}t^{m+1} \d t,
	\quad
	\lambda > C(m,\betaC).
\end{align}
The left hand side can be computed
\begin{align}
	\label{e.id.lhs}
	\text{l.h.s.\ of \eqref{e.id.goal}}
	=
	\int_0^\infty \frac{e^{\betaC\alpha} \lambda^{-\alpha-m-1}}{m+1}
	\Big(
		\int_0^\alpha
		\sum_{k=0}^m \binom{m+1}{m-k+1} (m-k)! \, p_{k-1}(\alpha_1) \d \alpha_1
	\Big)
	\d\alpha.
\end{align}
The integral~\eqref{e.id.lhs} is indeed finite for large enough $ \lambda \geq C(\betaC,m) $.
The right hand side is given by
\begin{align}
	\label{e.id.rhs}
	\text{r.h.s.\ of \eqref{e.id.goal}}
	=
	\int_0^\infty \frac{e^{\betaC\alpha} \lambda^{-\alpha-m-1}}{m+1}  p_m(\alpha) \, \d \alpha.
\end{align}
By Lemma~\ref{l.Gamma} the two coincide.
\end{proof}

\begin{lemma}
\label{l.semigroup}
For $ t'<s<t \in \R_+ $, $ i<j $, we have
\begin{align}
	\label{e.semigroup}
	\int_{t'<t_1<s} \int_{s<t_2<t} (4\pi \Pt^\Jo_{t_1-t'}) \So_{ij} \Pt_{t_2-t_1} \So^*_{ij} (4\pi\Pt^\Jo_{t-\tau_2}) \, \d t_1 \d t_2
	=
	\Pt^\Jo_{t-t'}.
\end{align}
\end{lemma}
\begin{remark}
The integral~\eqref{e.semigroup} converges absolutely in operator norm. 
This is seem by writing $ \So_{ij} \Pt_{t_2-t_1} \So^*_{ij} = (\So_{ij} \Pt_{s-t_1})(\Pt_{t_2-s} \So^*_{ij}) $,
and by using the bounds from Lemmas~\ref{l.Ptin.out}\ref{l.Ptin.out.} and~\ref{l.Ptj}\ref{l.Ptj.}.
\end{remark}
\begin{proof}
For $ \tau >0 $, the operator $ \So_{ij}\Pt_\tau\So^*_{ij} $ has an integral kernel 
$	
	\hk(\tau,S_{ij}(y-y')) = (\hktwo(\tau,y_2-y_2))^2 \prod_{i=3}^n \hktwo(\tau,y_i-y_i),
$
where $ \hktwo $ denotes the two-dimensional heat kernel.
From this and $ (\hktwo(\tau,y))^2 = \frac{1}{4\pi\tau} \hktwo(\frac{\tau}{2},y) $, we have
$ \So_{ij} \Pt_{\tau} \So^*_{ij} = \frac{1}{4\pi\tau} \exp(-\frac{\tau}{4}\nabla_2^2 -\frac{\tau}{2}\sum_{i=3}^n\nabla^2_i ) $.
Recall that $ \Pt^\Jo_\tau := \jfn(\tau,\betaC)\exp(-\frac{\tau}{4}\nabla_2^2 -\frac{\tau}{2}\sum_{i=3}^n\nabla^2_i ) $.
We obtain
\begin{align*}
	\text{l.h.s.\ of }\eqref{e.semigroup}
	=
	4\pi e^{-\frac{t-t'}{4}\nabla_2^2 -\frac{t-t'}{2}\sum_{i=3}^n\nabla^2_i }
	\int_{t'<t_1<s} \int_{s<t_2<t} \jfn(t_1-t') (t_2-t_1)^{-1} \jfn(t-t_2) \, \d t_1 \d t_2.
\end{align*}
The desired result now follows from~\eqref{e.the.identity}.
\end{proof}

\subsection{Proof of Theorem~\ref{t.main}}

We begin with a quantitative bound on $ \Dio^{\Vec{(i,j)}}_t $.
\begin{lemma}
For $ \Vec{(i,j)}=((i_k,j_k))_{k=1}^m \in \diag(n,m) $, $ t \in\R_+ $ and $ \lambda \geq 2 $, we have
\begin{align}
	\label{e.dio.bd}
	\Normo{ \Dio_{t}^{\Vec{(i,j)}} }
	\leq
	C 
	(\log(\tfrac{t}{2m+1}\wedge\tfrac12))^{-1}
	m^2 e^{\lambda C\,(\betaC+1) t}  (C/\log\lambda)^{m-1}.
\end{align}
\end{lemma} 
\begin{proof}
To simplify notation, we index the incoming and outgoing operators by $ 0 $ and by $ m $:
$ \geno^{(0)}_{\tau_0}:=\Pt_{\tau_0} \So^*_{i_1j_1} $, $ \geno^{(m)}_{\tau_m}:=\So_{i_mj_m} \Pt_{\tau_{m}} $,
index the diagonal mediating operators by half integers:
$ \geno^{(a)}_{\tau_{a}} := 4\pi\Pt^\Jo_{\tau_{a}} $, $ a\in(\frac12+\Z)\cap(0,m) $,
and index the off-diagonal mediating operators by integers:
$ \geno^{(a)}_{\tau_{a}} := \So_{i_{a}j_{a}} \Pt_{\tau_{a}} \So^*_{i_{a+1}j_{a+1}} $, $ a\in\Z\cap(0,m) $.
Under these notation
\begin{align}
	\tag{\ref*{e.diagram.op}'}
	\label{e.diagram.op.}
	\Dio^{\Vec{(i,j)}}_t
	=
	\int_{\Sigma_m(t)}
	\geno^{(0)}_{\tau_0} \, \geno^{(1/2)}_{\tau_{1/2}} \, \cdots \, \geno^{(m)}_{\tau_{m}} 
	\
	\d \vec{\tau}.
\end{align}

In general,
integrals like the one on the r.h.s.\ \eqref{e.diagram.op.} should be defined as operator-valued integrals.
Here we appeal to a simpler \emph{alternative} definition.
Recall from \eqref{e.Pt.in.ker}--\eqref{e.Pt.out.ker}, \eqref{e.Pt.med.ker}, and \eqref{e.Ptj.ker}
that each $ \geno^{(a)}_{\tau_{a}} $ has an integral kernel.
Accordingly, for each $ u,u'\in\Lsp^{2}(\R^{2n}) $,
we interpret 
$ 
	\ip{ u' }{
	\int_{\Sigma_m(t)}
	\geno^{(0)}_{\tau_0} \, \geno^{(1/2)}_{\tau_{1/2}} \, \cdots \, \geno^{(m)}_{\tau_{m}} 
	\
	\d \vec{\tau}
	\,u }
$
as an integral over $ \Sigma_m(t)\times(\R^{2n})^{2m+1} $ by expressing each $ \geno^{(a)}_{\tau_{a}} $ by its kernel.
Our subsequent analysis implies that this integral is absolutely convergent for each $ u,u'\in\Lsp^{2}(\R^{2n}) $,
and therefore \eqref{e.diagram.op.} defines an operator on $ \Lsp^2(\R^{2n}) $.
Since all the kernels are positive (c.f., \eqref{e.Pt.in.ker}--\eqref{e.Pt.out.ker}, \eqref{e.Pt.med.ker}, and \eqref{e.Ptj.ker}), we have
\begin{align}
	\notag
	|\ip{ u' } { \Dio^{\Vec{(i,j)}}_t\, u}|
	&=
	\Big|
	\IP{ u' }{
	\int_{\Sigma_m(t)}
	\geno^{(0)}_{\tau_0} \, \geno^{(1/2)}_{\tau_{1/2}} \, \cdots \, \geno^{(m)}_{\tau_{m}} 
	\
	\d \vec{\tau}
	\,u }
	\Big|
\\
	&\leq	
	\label{e.diagram.op.int}
	\int_{\Sigma_m(t)}
	\Big|
	\IP{ u' }{ \prod_{a\in A} \geno^{(a)}_{\tau_a}u }
	\Big|
	\
	\d \vec{\tau}
=
	\int_{\Sigma_m(t)}
	\IP{ |u'| }{ \prod_{a\in A} \geno^{(a)}_{\tau_a}|u| }
	\
	\d \vec{\tau}.
\end{align}%

We now seek to bound \eqref{e.diagram.op.int}.
An undesirable feature of~\eqref{e.diagram.op.int} is the constraint $ \tau_0+\tau_{1/2}+\ldots+\tau_m =t $ from $ \Sigma_m(t) $.
To break such a constraint, fix $ \lambda\geq 2 $.
In~\eqref{e.diagram.op.int}, multiply and divide by $ e^{\lambda\betaC t} $, and use 
$
	\Sigma_m(t) \subset (\cup_{a\in A} \{ \tau_a \geq \tfrac{t}{2m+1} \}) \cap (0,t)^{2m+1}
$
to obtain
\begin{align}
	\label{e.diagram.bd1}
	\normo{ \Dio_t^{\Vec{(i,j)}} }
	\leq
	e^{\lambda\betaC t}
	\sum_{a\in A} 
	F_a, 
	\quad
	F_a
	:=
	\Big( \sup_{\tau\in[\frac{t}{2m+1},t]} \!\! e^{-\lambda\betaC \tau}\normo{\geno^{(a)}_{\tau}} \Big)
	\prod_{a'\in A\setminus\set{a}} \!\! \NOrmo{ \int_0^t e^{-\lambda\betaC \tau} \geno^{(a')}_{\tau} \d \tau }.
\end{align}
To bound the `sup' term in~\eqref{e.diagram.bd1},
forgo the exponential factor (i.e., $  e^{-\lambda \betaC \tau} \leq 1 $),
and use the bound on $ \normo{\geno^{(a)}_{\tau}} $ from 
Lemmas~\ref{l.Ptin.out}\ref{l.Ptin.out.}, \ref{l.Ptbi}\ref{l.Ptbi.}, and \ref{l.Ptj}\ref{l.Ptj.}.
We have
\begin{align}
	\notag
	\sup_{\tau\in[\frac{t}{2m+1},t]} e^{-\lambda\betaC \tau} \normo{\geno^{(a)}_{\tau}}
	&\leq
	C 
	\left\{\begin{array}{l@{,}l}
		(t/m)^{-1/2}	&	\text{ for } a=0,m,
		\\
		(t/m)^{-1}		&	\text{ for } a \in \Z\cap(0,m),
		\\
		(t/m)^{-1}\, (\log(\frac{t}{2m+1}\wedge\frac12))^{-2} e^{C(1+\betaC)t}		&	\text{ for } a \in (\frac12+\Z)\cap(0,m),
	\end{array}\right\}
\\
	\label{e.diagram.bd2}
	&\leq
	C  m e^{C(1+\betaC)t}
	\left\{\begin{array}{l@{,}l}
		t^{-1/2}	&	\text{ for } a=0,m,
		\\
		t^{-1}		&	\text{ for } a \in \Z\cap(0,m),
		\\
		t^{-1}\, (\log(\frac{t}{2m+1}\wedge\frac12))^{-2}	&	\text{ for } a \in (\frac12+\Z)\cap(0,m).
	\end{array}\right.
\end{align}
Moving on, to bound the integral terms in~\eqref{e.diagram.bd1},
for $ a'\in\{0,m\}\cup ((\frac12\Z)\cap(0,m)) $, we forgo the exponential factor, and use the bound from Lemma~\ref{l.Ptin.out}\ref{l.Ptin.out.} to get
\begin{align}
	\label{e.diagram.bd3}
	&\NOrmo{ \int_0^t e^{\lambda\betaC \tau}\geno^{(a')}_{\tau} \d \tau }
	\leq
	\int_0^t \Normo{ \geno^{(a')}_{\tau} } \d \tau 
	\leq
	C t^{1/2},	
	\quad
	\text{for } a'=0,m,
\\
	\label{e.diagram.bd3.}
	&\NOrmo{ \int_0^t e^{\lambda\betaC \tau} \geno^{(a')}_{\tau} \d \tau }
	\leq
	\int_0^t \Normo{ \geno^{(a')}_{\tau} } \d \tau 
	\leq
	C  
	\, (\log(\tfrac{t}{2m+1}\wedge\tfrac12))^{-1} e^{C(1+\betaC)t},
	\quad
	\text{for } a' \in (\tfrac12+\Z)\cap(0,m).
\end{align}
The bound~\eqref{e.diagram.bd3.} gives a useful logarithmic decay in $ t\to 0 $,
but has an undesirable exponential growth in $ t\to \infty $.
We will also need a bound that does not exhibit the exponential growth.
For $ a' \in (\frac12\Z)\cap(0,m) $,
we use the fact that $ \geno^{(a')}_{\tau} $ is an integral operator with a \emph{positive} kernel to write
\begin{align*}
	\NOrmo{ \int_0^t e^{-\lambda\betaC\tau}\geno^{(a')}_{\tau} \d\tau }
	\leq
	\NOrmo{ \int_0^\infty e^{-\lambda\betaC\tau}\geno^{(a')}_{\tau} \d\tau }.
\end{align*}
The last expression is a Laplace transform, and has been evaluated in Lemmas \ref{l.Ptbi}\ref{l.Ptbi.Lap} and \ref{l.Ptj}\ref{l.Ptj.Lap}, whereby
\begin{align*}
	\NOrmo{ \int_0^t e^{-\lambda\betaC\tau}\geno^{(a')}_{\tau} \d\tau }
 	\leq
	\left\{\begin{array}{l@{,}l}
		\normo{\So_{ij}\Go_{-\lambda\betaC}\So^*_{k\ell}} &	\text{ for } a'\in (0,m)\cap\Z,
		\\
		\normo{(\Jo_{-\lambda\betaC}-\betaC)^{-1}}	&	\text{ for } a' \in (0,m)\cap(\frac12+\Z).
	\end{array}\right.
\end{align*}
Here $(i<j)\neq (k<\ell)$ corresponds to the index $a'$.
Using the bounds on $ \normo{\So_{ij}\Go_z\So^*_{k\ell}} $ from Lemma~\ref{l.off.diagonal}
and the bound $ \norm{(\Jo_{-\lambda\betaC}-\betaC)^{-1}} \leq 1/\log\lambda $ (c.f.,~\eqref{e.Jo})
we have
\begin{align}
	\label{e.diagram.bd4}
 	\NOrmo{ \int_0^t e^{-\lambda\betaC\tau}\geno^{(a')}_{\tau} \d\tau }
 	\leq
 	C\,
	\left\{\begin{array}{l@{,}l}
		1 &	\text{ for } a'\in (0,m)\cap\Z,
		\\
		(\log \lambda)^{-1}	&\text{ for } a' \in (0,m)\cap(\frac12+\Z).
	\end{array}\right.
\end{align}

For $ a \in \frac12\Z $,
inserting the bounds \eqref{e.diagram.bd2}--\eqref{e.diagram.bd3}, \eqref{e.diagram.bd4} into~\eqref{e.diagram.bd1} gives
\begin{align*}
	F_a
	\leq
	C m e^{\lambda C\,(\betaC+1) t} 
	(\log(\tfrac{t}{2m+1}\wedge\tfrac12))^{-2}
	\,
	t^{-1+\frac12+\frac12}
	\,
	(\log\lambda)^{m-1}
	C^{2m+1}.
\end{align*}
For $ a \not\in \frac12\Z $, in~\eqref{e.diagram.bd1},
use the bound~\eqref{e.diagram.bd2} for the sup term,
use \eqref{e.diagram.bd3.} for $ a'=\frac12 $,
and use \eqref{e.diagram.bd3} and \eqref{e.diagram.bd4} for other $ a' $.
This gives
\begin{align*}
	F_a
	\leq
	C m e^{\lambda C\,(\betaC+1) t} 
	(\log(\tfrac{t}{2m+1}\wedge\tfrac12))^{-1}
	\,
	\left\{\begin{array}{l@{,}l}	
		t^{-1/2+1/2}	&\text{ for } a\in \{0,m\}
	\\
		t^{-1+1/2+1/2}  &\text{ for } a\in \Z\cap(0,m)
	\end{array}\right\}
	\,
	(\log\lambda)^{m-1} C^{2m+1}.
\end{align*}
Inserting these bounds on $ F_a $ into~\eqref{e.diagram.bd1}, we conclude the desired result~\eqref{e.dio.bd}.
\end{proof}

\subsubsection*{Proof of Theorem~\ref{t.main}\ref{t.main.diagram}}
Sum the bound~\eqref{e.dio.bd} over $ \Vec{(i,j)}\in\diag(n) $, and note that $ | \diag(n,m) |\leq (n(n-1)/2)^{m} $ (c.f., \eqref{e.diag.set.m}).
In the result, choose $ \lambda = Cn^2 $ for some large but fixed $ C<\infty $, we have
\begin{align}
	\label{e.diagram.op.bd.}
	\Normo{ \Dio^{\diag(n)}_t }
	&\leq
	\sum_{m=1}^\infty
	m^2n^2 (\log(\tfrac{t}{2m+1}\wedge\tfrac12))^{-1} \, 2^{-(m-1)} \exp\big( Ce^{Cn^2}(\betaC+1)t\big)
\\
	\label{e.diagram.op.bd}
	&\leq
	C\,n^2\exp\big( e^{Cn^2}(\betaC+1)Ct\big).
\end{align} 
This verifies that $ \Dio^{\diag(n)}_t $ defines a bounded operator on $ \Lsp^2(\R^{2n}) $.

To show the semigroup property, we fix $ s<t\in\R_+ $ and calculate $ (\Pt_s+\Dio^{\diag(n)}_s)(\Pt_{t-s}+\Dio^{\diag(n)}_{t-s}) $,
which boils down to calculating
$\Pt_s\Pt_{t-s}$, $\Pt_{s}\Dio^{\Vec{(i',j')}}_{t-s}$, $\Dio^{\Vec{(i,j)}}_s\Pt_{t-s}$, $\Dio^{\Vec{(i,j)}}_s\Dio^{\Vec{(i',j')}}_{t-s}$,
for $ \Vec{(i,j)} \in \diag(n,m) $ and $ \Vec{(i',j')} \in \diag(n,m') $.
To streamline notation, we relabel time variables as $ t_{k} := \tau_0+\ldots+\tau_{k/2-1} $, and set
\begin{align*}
	B^{\Vec{(i,j)}}(\vec{t}\,)
	:=
	\Pt_{t_{1}} \So^*_{i_1j_1} \big( 4\pi\Pt^\Jo_{t_2-t_1} \big) 
	\Big( \prod_{k=1}^{{m}-1} \So_{i_{k}j_{k}} \Pt_{t_{2k+1}-t_{2k}} \So^*_{i_{k+1}j_{k+1}} \, (4\pi\Pt^\Jo_{\tau_{2k+2}-t_{2k+1}}) \Big) 
	\So_{i_{m}j_{m}} \Pt_{t-t_{2m}}.
\end{align*}
Using~\eqref{e.diagram.op.} and the semigroup property of $ \Pt_\Cdot $, we have $ 	\Pt_s \Pt_{t-s} = \Pt_{t} $,
\begin{align}
	\label{e.sg1}
	\Pt_s\Dio^{\Vec{(i',j')}}_{t-s}
	&=
	\int_{(s,t)^{2m'}_<} B^{\Vec{(i',j')}}(\vec{t}\,) \, \d \vec{t},
\\	
	\label{e.sg2}
	\Dio^{\Vec{(i,j)}}_s \Pt_{t-s}
	&=
	\int_{(0,s)^{2m}_<} B^{\Vec{(i,j)}}(\vec{t}\,) \, \d \vec{t},
\\	
	\label{e.sg3}
	\Dio^{\Vec{(i,j)}}_s \Dio^{\Vec{(i',j')}}_{t-s}
	&=
	\int_{\Omega_{2m,2m'}(s,t)} 
	B^{\Vec{(i'',j'')}}(\vec{t}\,) \, \d \vec{t},
\end{align}
where $ (a,b)^{k}_{<} := \{ \vec{t} \in (a,b)^{k} : a<t_1<\ldots<t_{k}<b \} $,
$ \Omega_{k,\ell}(s,t) := \{ \vec{t}\in (0,t)^{k+\ell} : \ldots<t_{k}<s<t_{k+1}<\ldots<t_{k+\ell}<t \} $,
and $ \Vec{(i'',j'')} $ is obtained by concatenating $ \Vec{(i,j)} $ and $ \Vec{(i',j')} $,
i.e.,
\begin{align*}
	\Vec{(i'',j'')}
	=
	(i''_k,j''_k)_{k=1}^{m+m'}
	:=
	\big((i_1<j_1),\ldots, (i_m<j_m), (i'_1<j'_1), \ldots, (i_{m'}<j_{m'})  \big).
\end{align*}
Such an index is not necessarily in $ \diag(n) $, because we could have $ (i_m<j_m)=(i'_1<j'_1) $.
When this happens, applying Lemma~\ref{l.semigroup} with $ (i,j)=(i_m,j_m) $ and with $ (t',t)\mapsto (t_{2m-1},t_{2m+2}) $ gives
\begin{align}
	\tag{\ref*{e.sg3}'}
	\label{e.sg3'}
	\Dio^{\Vec{(i,j)}}_s \Dio^{\Vec{(i',j')}}_{t-s}
	&=
	\int_{\Omega_{2m-1,2m'-1}(s,t)} 
	B^{\Vec{(i''',j''')}}(\vec{t}\,) \, \d \vec{t},
\end{align}
where $ \Vec{(i''',j''')} $ is obtained by removing $ (i'_1<j'_1) $ from $ \Vec{(i'',j'')} $, i.e.,
\begin{align*}
	\Vec{(i''',j''')}
	:=
	\big((i_1<j_1),\ldots, (i_m<j_m), (i_{2}<j_{2})  \ldots, (i_{m'}<j_{m'})  \big)
	\in
	\diag(n).
\end{align*}
Summing~\eqref{e.sg1}--\eqref{e.sg3}, \eqref{e.sg3'} over $ \Vec{(i,j)},\Vec{(i',j')}\in\diag(n) $ verifies the desired semigroup property:
\begin{align*}
	\Pt_{s} \Pt_{t-s} + \big( \Pt_{s}\Dio^{\diag(n)}_{t-s} + \Dio^{\diag(n)}_s \Pt_{t-s} + \Dio^{\diag(n)}_s\Dio^{\diag(n)}_{t-s} \big)
	=
	\Pt_{t} + \Dio^{\diag(n)}_{t-s}.
\end{align*}

We now turn to norm continuity.
Given the semigroup property, it suffices to show continuity at $ t=0 $.
The heat semigroup $ \Pt_t $ is indeed continuous at $ t=0 $.
As for $ \Dio^{\diag(n)}_t $, 
we have $ \Dio^{\diag(n)}_0 := 0 $, and from~\eqref{e.diagram.op.bd.}
$
	\lim_{t\to 0} \normo{\Dio^{\diag(n)}_t} = 0.
$

\subsubsection*{Proof of Theorem~\ref{t.main}\ref{t.main.dcnvg}}
Given~\eqref{e.semigroup.cnvg}, proving Part~\ref{t.main.dcnvg} amounts to showing $ \Pt_t+\Dio^{\diag(n)}_t = e^{-t\Ho} $. 
Equivalently, for fixed $ u,u'\in\Lsp^2(\R^{2n}) $ %
and for $ f(t) := \ip{u'}{(\Pt_t+\Dio^{\diag(n)}_t) u} $ and $ g(t) := \ip{u'}{e^{-t\Ho} u} $, %
the goal is to show $ f(t) = g(t) $ for all $ t\ge 0 $. %
Both functions are continuous since $ \Pt_t+\Dio^{\diag(n)}_t $ and $ e^{-t\Ho} $ are norm-continuous. %
Further, by \eqref{e.diagram.op.bd} and from $ \sigma(\Ho) \subset [-C(n,\betaC),\infty) $ %
we have $ \normo{\Pt_t+\Dio^{\diag(n)}_t}+ \normo{e^{-t\Ho}} \leq C(n,\betaC) \exp( C(n,\betaC)t) $. %
Hence it suffices to match the Laplace transforms of $ f(t) $ and $ g(t) $ for sufficiently large values $ \lambda \geq C(n,\betaC) $ of the Laplace variable.

To evaluate the Laplace transform of $ f(t)=\ip{u'}{(\Pt_t+\Dio^{\diag(n)}_t) u} $, assume for a moment $ u(x),u'(x) \geq 0 $,
we integrate \eqref{e.diagram.op.} (viewed as in integral operator) against $ e^{-\lambda t} \bar{u'}(x)u(x') $ over $ t\in\R_+ $ and $ x,x'\in\R^{2n} $,
and sum the result over all $ \Vec{(i,j)}\in\diag(n) $.
This gives
\begin{align}
	\label{e.laplace.matching.f1}
	 \int_0^\infty e^{-\lambda t} f(t) \, \d t  
	= \int_0^\infty e^{-\lambda t} \Pt_t \, \d t 
	+
	\sum_{\Vec{(i,j)}\in\diag(n)}  \Bigg\langle u',  \Big( \prod_{a\in A} \int_0^\infty e^{-\lambda t} \geno^{(a)}_{t}  \, \d t \Big) u \Bigg\rangle,
\end{align}
where, the operator $ \geno^{(a)}_{t} $ are indexed as described in the preceding.
In deriving~\eqref{e.laplace.matching.f1},
we have exchanged sums and integrals, which is justified because each $ \geno^{(a)}_t $ has a positive kernel,
and $ u(x'),u'(x) \geq 0 $ under the current assumption.
On the r.h.s.\ of~\eqref{e.laplace.matching.f1},
the Laplace transforms $ \int_0^\infty e^{-t\lambda} \geno^{(a)}_{t}  \, \d t $ are evaluated as
in Lemmas~\ref{l.Ptin.out}\ref{l.Ptin.out.Lap}, \ref{l.Ptbi}\ref{l.Ptbi.Lap}, and \ref{l.Ptj}\ref{l.Ptj.Lap}.
Putting together the expressions from these lemmas, and comparing the result to~\eqref{e.resolvent.},
we now have 
\begin{align*}
	\int_0^\infty e^{-\lambda t} f(t) \, \d t 
	=
	\Ip{u'}{ \big(\text{r.h.s.\ of }\eqref{e.resolvent}\big|_{z=-\lambda}\big) u }
	=
	\Ip{u'}{ \Ro_{-\lambda} u }
	= \int_0^\infty e^{-\lambda t} g(t) \, \d t.
\end{align*}
For general $ u,u'\in\Lsp^2(\R^{2n}) $, the preceding calculation done for $ (u(x),u'(x'))\mapsto (|u(x)|,|u'(x')|) $
guarantees the relevant integrability, and justifies the exchange of sums and integrals.

\bibliographystyle{alphaabbr}
\bibliography{2dSHEmom}
\end{document}